\documentclass[final,reqno]{siamart1116}

\usepackage{amsfonts}
\usepackage{graphicx}
\usepackage{mathrsfs}
\usepackage{bm}
\usepackage{multirow}
\usepackage{color}

\usepackage{caption, subcaption}%
\usepackage{array}
\usepackage{enumitem}

\usepackage{geometry}

\newsiamthm{claim}{Claim}
\newsiamremark{remark}{Remark}
\newsiamremark{expl}{Example}
\numberwithin{theorem}{section}

\allowdisplaybreaks

\def\vec#1{{\bf #1}}

\def\D{\mathcal{D}}

\newcommand{\figref}[1]{Fig.~\ref{#1}}

\title{
  Provably Physical-Constraint-Preserving 
  Discontinuous Galerkin Methods    
  for Multidimensional Relativistic MHD Equations
}

\author{
  Kailiang Wu\thanks{Department of Mathematics, The Ohio State University, 
  	Columbus, OH 43210, USA  ({\tt wu.3423@osu.edu}).}
  \and
  Chi-Wang Shu\thanks{Division of Applied Mathematics, Brown University, Providence, RI 02912, USA ({\tt Chi-Wang\_Shu@brown.edu.}). Research is supported in part by NSF grant DMS-1719410. }  
}

\headers{Provably PCP methods for RMHD}
{Kailiang Wu and Chi-Wang Shu}

\ifpdf
\hypersetup{ pdftitle={Provably PCP methods for RMHD} }
\fi

\begin{document}
\maketitle

\begin{center}
February 7, 2020
\end{center}

\vspace{3mm}


\begin{abstract}
We propose and analyze a class of robust, uniformly high-order accurate discontinuous Galerkin (DG) schemes for multidimensional relativistic magnetohydrodynamics (RMHD) on general meshes. 
A distinct feature of the schemes is their physical-constraint-preserving (PCP) property, 
i.e., they are proven to  
 preserve the subluminal constraint on the fluid velocity and the positivity of density, pressure,  and specific internal energy. 
Developing PCP high-order schemes for RMHD is highly desirable but remains a challenging task, especially in the multidimensional cases, due to 
the inherent strong nonlinearity in the constraints and the effect of the magnetic divergence-free condition. 
Inspired by some crucial observations at the 
PDE level, we construct the provably PCP 
schemes by using the locally divergence-free DG schemes 
of the recently proposed symmetrizable RMHD equations as the base schemes, 
a limiting technique to enforce the PCP property of the
DG solutions, and the 
strong-stability-preserving methods for 
time discretization. 
We rigorously prove the PCP property 
by using a novel ``quasi-linearization'' approach to handle the
highly nonlinear physical constraints, technical splitting to offset the influence of divergence error,  
and sophisticated estimates to analyze the beneficial effect of the additional source term in the symmetrizable RMHD system. 
Several two-dimensional numerical examples are provided to confirm 
the PCP property and to demonstrate the accuracy, effectiveness and robustness 
of the proposed PCP schemes.

\end{abstract}



\section{Introduction} 
\label{sec:intro}

This paper is concerned with developing robust high-order accurate numerical methods for the special relativistic magnetohydrodynamics (RMHD) equations, which are used to 
describe the dynamics of electrically-conducting fluids at nearly the speed of light 
in the presence of magnetic field. 
RMHD play an important role in many fields, such as astrophysics and high energy physics, 
and have been used to investigate a number of 
astrophysical scenarios from stellar to galactic scales, e.g., gamma-ray bursts,  
formation
of black holes, astrophysical jets, blast waves of supernova
explosions, gravitational collapse and accretion, etc. 

The special RMHD equations are often formulated as a nonlinear system of hyperbolic conservation laws
\begin{equation}\label{eq:RMHD}
{\bf U}_t + \nabla \cdot {\bf F} ( {\bf U} ) = {\bf 0},
\end{equation}	
where $\nabla \cdot = \sum_{i=1}^d \frac{\partial }{\partial x_i} $ is the divergence operator 
with $d \in \{1,2,3\}$ denoting the spatial dimensionality. 
Here we employ the geometrized unit system so that the speed of light $c=1$. In \eqref{eq:RMHD}, the conservative vector $\vec U = ( D,\vec m,\vec B,E )^{\top}$, and the flux ${\bf F}=({\bf F}_1, \dots, {\bf F}_d)$ is defined by 
\begin{equation*}
	\vec F_i(\vec U) = \left( D v_i,  v_i \vec m  -  B_i \big( W^{-2} \vec B + (\vec v \cdot \vec B) \vec v \big)  + p_{tot}  \vec e_i, v_i \vec B - B_i \vec v ,m_i \right)^{\top},  
\end{equation*}
with the mass density $D = \rho W$, the momentum vector $\vec m = (\rho H {W^2} + |\vec B|^2) \vec v - (\vec v \cdot \vec B)\vec B$, the magnetic field ${\bf B}=(B_1,B_2,B_3)$, the energy $E=\rho H W^2 - p_{tot} +|\vec B|^2$, and 
 the vector $\vec e_i$ denoting the $i$-th row of the unit matrix of size $3$. 
Additionally, $\rho$ is the rest-mass density, $\vec v=(v_1,v_2,v_3)$ denotes the velocity field of 
the fluid, $W=1/\sqrt{1- |{\bf v}|^2}$ is the Lorentz factor, $p_{tot}$ is the total pressure consisting of the thermal pressure $p$ and the magnetic pressure $p_m:=\frac12 \left(W^{-2} |\vec B|^2 +(\vec v \cdot \vec B)^2 \right)$,
$H = 1 + e + \frac{p}{\rho}$ represents the specific enthalpy, and $e$ is the specific internal energy. 
The equation
of state (EOS) is needed to close the system \eqref{eq:RMHD}. A general EOS can be expressed as
\begin{equation}\label{eq:gEOS}
H = H(p,\rho).
\end{equation}
A simple example is the ideal EOS 
\begin{equation}\label{eq:iEOS}
H = 1 + \frac{\Gamma p}{(\Gamma -1) \rho},
\end{equation}
where $\Gamma \in (1,2]$ is a constant and denotes the adiabatic index, for which the restriction $\Gamma \le 2$ is required by the compressibility assumptions and the relativistic causality. 
Given an EOS, the conservative vector $\vec U$ and the flux ${\bf F}$ can be explicitly expressed by 
the primitive variables $\{ \rho, p, {\bf v}, {\bf B}\}$. 
However, unlike the non-relativistic case, 
there are no explicit formulas for either
the flux ${\bf F}$ or the primitive variables $\{ \rho, p, {\bf v}\}$ in terms of ${\bf U}$, due 
to the relativistic effect, especially the appearance of the Lorentz factor.

The magnetic field should also satisfy an additional divergence-free condition 
\begin{equation}\label{eq:DivB}
\nabla \cdot {\bf B} := \sum_{i=1}^d \frac{\partial B_i}{\partial x_i}=0, 
\end{equation}
which is a reflection of the principle that there are no magnetic monopoles. 
Although the satisfaction of 
\eqref{eq:DivB} is not explicitly included in 
the system \eqref{eq:RMHD}, 
the exact solution of \eqref{eq:RMHD} always preserves zero divergence for $\bf B$ in future time if the initial divergence is zero.  
Besides the
standard difficulty in solving the nonlinear hyperbolic systems,
an additional numerical challenge for the RMHD system \eqref{eq:RMHD} comes from
the divergence-free condition \eqref{eq:DivB}, which is also involved in the ideal non-relativistic MHD system. 
It is widely realized that the condition \eqref{eq:DivB} is important for robust computations, since  
large divergence error in the numerical magnetic field 
can lead to numerical instabilities or nonphysical structures in the computed solutions, cf.~\cite{Evans1988,BalsaraSpicer1999,Toth2000,Li2005}. 
In the one-dimensional case ($d=1$), $B_1$ is constant so that the condition \eqref{eq:DivB} can be easily enforced in numerical computations. 
However, in the multidimensional cases ($d\ge 2$), numerical preservation of \eqref{eq:DivB} 
is highly nontrivial, and various techniques have been proposed to reduce the divergence error or enforce the
divergence-free condition in the discrete sense; see e.g.,  
\cite{Evans1988,POWELL1999284,Toth2000,Dedner2002,Torrilhon2005,Li2005,Li2011,XuLiu2016,Fu2018} and the references therein.


In physics, the density, thermal pressure and internal energy are positive, and the fluid velocity must be 
slower than the speed of light in the vacuum $c=1$.  
Mathematically, an equivalent description is that the conservative vector ${\bf U}$ must stay in the set of physically admissible states 
defined by 
\begin{equation}\label{eq:RMHD:definitionG}
{\mathcal G} := \left\{ \vec U=(D,{\vec m},{\vec B},E)^{\top}:~\rho(\vec U)>0,~p(\vec U)>0,~e({\bf U})>0,~|{\bf v}(\vec U)| <1  \right\},
\end{equation}
where the functions $\rho(\vec U)$, $p(\vec U)$, $e({\bf U})$ and ${\bf v}(\vec U)$ are highly nonlinear and cannot be explicitly
formulated in terms of $\bf U$, due to the relativistic effect. 
In numerical computations, preserving the numerical solutions in $\mathcal G$ is highly desirable and 
crucial for the robustness of the numerical schemes. 
This is because once any physical constraints in \eqref{eq:RMHD:definitionG} are violated in the numerical simulations, the discrete problem becomes ill-posed due to the loss of hyperbolicity, causing 
the breakdown of the simulation codes. 
In the past several decades, various numerical schemes have been developed for the RMHD, e.g., \cite{komissarov1999godunov,del2007echo,MignoneHLLCRMHD,Host:2008,HeTang2012RMHD,Zanotti2015,BalsaraKim2016,ZhaoTang2017}. 
However, none of them were proven to preserve all these constraints, even though they have been used to simulate some RMHD flows successfully. 
In fact, most of the existing RMHD schemes do not always preserve these constraints, 
and thus may suffer from a large risk of failure when simulating RMHD problems with large Lorentz factor, low density or pressure, or strong discontinuity. 
It is therefore highly significant and desirable to develop 
physical-constraint-preserving (PCP) numerical schemes whose solutions always stay in the set ${\mathcal G}$.

During the past decade, 
significant progress has been made for constructing bound-preserving high-order accurate schemes for hyperbolic systems, mainly built on two types of limiters.
One is a simple scaling limiter
for the reconstructed or evolved solution polynomials in
finite volume or discontinuous Galerkin (DG) methods; see, e.g.,  \cite{zhang2010,zhang2010b,ZHANG2017301,WuShu2018,WuShu2019,ZouYuDai2019}. 
Another one is a flux-correction limiter, see, e.g.,  \cite{Xu2014,Hu2013,Liang2014,Christlieb}. 
For more developments, we refer interested readers to the survey \cite{Shu2018} and references therein. 
With these limiting approaches, several PCP methods were developed for the special relativistic hydrodynamics (RHD)
without the magnetic field, including high-order accurate PCP 
finite difference schemes \cite{WuTang2015}, PCP DG schemes \cite{QinShu2016}, PCP central DG schemes \cite{WuTang2017ApJS}, and PCP Lagrangian finite volume 
schemes \cite{LingDuanTang2019}. 
Extension of the PCP methods from special to general RHD is highly nontrivial.
An earlier effort \cite{Radice2014} was made in this direction but only enforced the 
positivity of density. 
Recently, frameworks of designing provably PCP high-order finite difference, finite volume and DG methods were established in \cite{Wu2017} for the general RHD. All of the aforementioned PCP methods were restricted to RHD without the
magnetic field.

Seeking PCP schemes for the RMHD is highly challenging, largely due to the intrinsic complexity of the RMHD equations and strong 
nonlinearity contained in the physical constraints in \eqref{eq:RMHD:definitionG}. 
As mentioned above, there are no explicit expressions of the highly nonlinear functions $\rho(\vec U)$, $p(\vec U)$, $e({\bf U})$ and ${\bf v}(\vec U)$ for the RMHD. Taking the ideal EOS case \eqref{eq:iEOS} as example, 
in order to obtain the values of $\{\rho,p,e,{\bf v}\}$ from a given vector $\vec U=(D,{\bf m}, {\bf B}, E)^\top$, one has to solve a nonlinear algebraic equation \cite{MignoneHLLCRMHD}: 
\begin{equation}\label{eq:RMHD:fU(xi)}
\theta
  - \frac{{\Gamma  - 1}}{\Gamma }\left( {\frac{ \theta }{{{{\Upsilon^2_{ \bf U}( \theta )}}}} - \frac{D}{ 
		\Upsilon_{ \bf U}( \theta )}} \right) + {\left| \vec B \right|^2} - \frac{1}{2}\left( {\frac{{{{\left| \vec B \right|}^2}}}{{{{\Upsilon^2_{ \bf U}( \theta )}}}} + \frac{{{{(\vec m \cdot \vec B)}^2}}}{{{ \theta ^2}}}} \right) - E = 0,  
\end{equation}
for the unknown $\theta \in \mathbb R^+$, where the function $\Upsilon_{ \bf U}( \theta )$ is defined by 
\begin{equation*}
\Upsilon_{ \bf U}( \theta ) = \left( \frac{ { \theta^2}{{( \theta + {|\vec B|^2})}^2} - \left[ { \theta^2}{|\vec m|^2} + (2 \theta+{|\vec B|^2})  {{(\vec m \cdot \vec B)}^2} \right] } {  { \theta^{2}} {{( \theta + {|\vec B|^2})}^{2}} }
\right)^{ - {1}/{2} }.
\end{equation*}
Assume that an admissible solution of the equation  \eqref{eq:RMHD:fU(xi)} exists for the given state $\bf U$, and denote it by $\hat{\theta}=\hat{\theta} (\vec U)$,
then the primitive variables in \eqref{eq:RMHD:definitionG} can be computed by
\begin{equation}\label{getprimformU}
\begin{aligned}
&
\vec v(\vec U) = \left( {\vec m + \hat \theta ^{ - 1}(\vec m \cdot \vec B) \vec B} \right)/( \hat \theta+ {|\vec B|^2}),
&\rho (\vec U) = \frac{D}{ \Upsilon_{ \bf U} ( \hat \theta )},\\
&
p(\vec U) = \frac{{\Gamma  - 1}}{{\Gamma \Upsilon_{ \bf U}^2( \hat \theta )}}\Big( {{ \hat \theta } - D \Upsilon_{ \bf U} ( \hat \theta )} \Big), & e({\bf U}) = 
\frac{p({\bf U})}{(\Gamma - 1) \rho ({\bf U}) }.
\end{aligned}
\end{equation}
As clearly shown in the above procedure, checking the admissibility of a given state ${\bf U}$ is already a very difficult task. 
On the other hand, in most of the numerical RMHD schemes, the conservative quantities are themselves evolved according to their own conservation laws, which are seemingly unrelated to and numerically do not necessarily guarantee the desired bounds of 
the computed primitive variables $\{\rho,p,e,{\bf v}\}$. 
In theory, it is indeed a challenge to make an a priori judgment on whether a scheme is always PCP under all circumstances or not.
Therefore, the study of PCP schemes for the RMHD has remained blank until the recent work in \cite{WuTangM3AS}, where several 
important mathematical properties of the set $\mathcal G$ were first derived 
and 
PCP finite volume and DG methods were developed for the conservative RMHD equations \eqref{eq:RMHD} in one space dimension.  
Moreover, for the multidimensional conservative RMHD equations, the theoretical analysis in \cite{WuTangM3AS} revealed that the PCP property of standard finite volume and DG methods is closely connected with 
a discrete divergence-free condition on the numerical magnetic field. 
This finding was further extended on general meshes in \cite{WuTangZAMP} and is 
consistent with the ideal non-relativistic MHD case \cite{Wu2017a}. 
It was also shown in \cite{WuTangM3AS,WuTangZAMP} that if the discrete divergence-free condition is slightly violated, even the first-order
multidimensional Lax-Friedrichs scheme for \eqref{eq:RMHD} is not PCP in general. 
Unfortunately, the required discrete divergence-free condition 
relies on certain combination of the information on adjacent cells, so that it could not be naturally 
enforced by any existing divergence-free techniques that also work  
in conjunction with the standard local scaling PCP limiter \cite{WuTangM3AS}.   
Therefore, the design of multidimensional 
PCP schemes for the RMHD has challenges essentially different from the one-dimensional case. As a result, 
provably PCP high-order schemes have {\em not} yet been obtained 
for the conservative RMHD system \eqref{eq:RMHD} 
in the multidimensional cases.

The focus of this paper is to 
develop a class of 
provably PCP high-order DG schemes 
for the multidimensional RMHD with a general EOS on general meshes. 
Towards achieving this goal, we will make the following efforts in this paper:   

1. First,	we investigate the PCP property of the exact solutions of the conservative RMHD system \eqref{eq:RMHD} at the PDE level. 
	We observe that, if the divergence-free condition \eqref{eq:DivB} is (slightly) violated, 
	the exact smooth solution of \eqref{eq:RMHD} may fail to be PCP, i.e., $\mathcal G$ is not 
	an invariant region for the exact solution of \eqref{eq:RMHD}. 
Therefore, before seeking provably PCP numerical schemes, our first task is to reformulate 
	the RMHD equations so as to accommodate the PCP property at the
PDE level. 
	We consider a symmetrizable formulation of the RMHD equations, which we recently proposed in \cite{WuShu2019SISC}, by 
	building the divergence-free condition \eqref{eq:DivB} into 
	the equations through adding a source term. 
	We show that the exact smooth solutions of 
	the symmetrizable RMHD system 
	always retain the PCP property even if the magnetic field is {\em not} divergence-free.  

2. Based on the symmetrizable formulation, we establish a framework of constructing 
	provably PCP high-order DG schemes 
	for the multidimensional RMHD with a general EOS on general meshes. 
	The key is to properly discretize the symmetrizable RMHD equations so as to eliminate the effect of divergence error on the  PCP property 
	of the resulting DG schemes. 
	We adopt the locally divergence-free DG elements, which enforce zero divergence within each cell, and a suitable discretization of the symmetrization source term, which brings some discrete divergence terms into our schemes and  
	exactly offsets the influence of divergence error on the PCP property.

3. A significant innovation in this paper is that we discover and rigorously prove the 
	PCP property of the proposed DG schemes, without requiring any discrete divergence-free condition. 
	There are two main technical challenges in the proof. 
	One is how to explicitly and analytically verify the admissibility of any given conservative state ${\bf U}$,  
	without solving the nonlinear equation \eqref{eq:RMHD:fU(xi)}.  
	This difficulty has been addressed in \cite{WuTangM3AS} based on two equivalent forms of the admissible state set $\mathcal G$. 
	The other is how to take the advantages of the 
	locally divergence-free property and our suitable discretization of the source term in the symmetrizable RMHD formulation, to eliminate the effect of divergence error on the PCP property.  
	Due to the locally divergence-free property and the source term,  
	 the limiting values of the numerical solution at the interfaces of each cell  
	are intrinsically coupled, making some standard analysis techniques (\cite{zhang2010b}) inapplicable. 
	We will overcome this difficulty   
	by using a novel ``quasi-linearization'' approach to handle the highly nonlinear constraints in \eqref{eq:RMHD:definitionG}, 
	technical splitting to offset the influence of divergence error, 
	and sophisticated estimates to analyze the beneficial effect of the symmetrization source term. 

4. 	We implement the proposed PCP DG schemes on two-dimensional Cartesian meshes and 
demonstrate  
their accuracy, effectiveness and robustness for several numerical examples.  
We will show that our PCP schemes, 
without any artificial treatments, are able to successfully simulate several    
challenging problems, including a strongly magnetized bast problem with extremely low plasma-beta ($2.5 \times 10^{-10}$) 
and highly supersonic RMHD jets, which are rarely considered in the literature. 	

The present study is also motivated by our recent work \cite{WuShu2018,WuShu2019} on the positivity-preserving DG schemes for 
the ideal non-relativistic MHD. 
Compared to the non-relativistic case, 
the present study is much more challenging, due to the highly nonlinear coupling of the RMHD equations and the complicated mapping from the conservative to primitive variables. 
Additional technical challenges also arise from the suitable discretization of the symmetrization source term and especially some novel estimate techniques required to analyze its beneficial effect 
on the PCP property.


\section{Auxiliary observations on the PCP property at the PDE level} 

This section introduces our observations on 
 the PCP property of the exact smooth solutions of   
the conservative formulation \eqref{eq:RMHD} and a symmetrizable formulation of the RMHD equations, respectively, with the ideal EOS \eqref{eq:iEOS}. 
The findings will provide some insights that guide us to successfully construct the  
PCP schemes for the RMHD.

We observe that negative pressure may appear in the exact smooth solution of the conservative RMHD system \eqref{eq:RMHD} if $\nabla \cdot {\bf B}\neq 0$. 
An evidence, rather than rigorous proof,  
for this claim may be given by considering the following initial condition 
\begin{equation}\label{eq:np:data1}
\begin{aligned}
&	\rho({\bf x},0)= 1, \qquad p({\bf x},0)=1-\exp(-|{\bf x}|^2),
\\
& {\bf v} ({\bf x},0)=  (0.01,~0.01,~0.01 ), \qquad {\bf B}({\bf x},0) 
= (1+\delta B_1,~1+\delta B_2,~1+\delta B_d), 
\end{aligned}
\end{equation}
where ${\bf x}=(x_1,\dots,x_d)$, and 
$\delta B_i=\epsilon \arctan x_i$, $1\le i \le 3$, are small perturbations 
with $0<\epsilon \ll 1$. 
Since the initial solution \eqref{eq:np:data1} is bounded and infinitely differentiable, it is reasonable to assume:     
there exists a small time interval $[0,t_*)$ such that 
the exact solution of the system \eqref{eq:RMHD} with \eqref{eq:np:data1} exists and is smooth for $t \in [0,t_*)$. 
Since $|{\bf v} ({\bf 0},0)|-1 = -0.97 < 0$ and $\rho({\bf 0},0) = 1>0$, by 
the sign-preserving property for continuous functions, there exists a neighborhood $\Omega$ of $\bf 0$ in $\mathbb R^d$ and $t_0 \in (0,,t_*)$ such that $|{\bf v} ({\bf x},t)|-1<0$ and $\rho({\bf x},t)>0$ for all $({\bf x},t) \in \Omega \times [0,t_0)$. 
Let us then study the initial time derivative of $p\rho^{-\Gamma}$ at $({\bf x},t)=({\bf 0},0)$. 
For smooth solutions, we derive from  
\eqref{eq:RMHD} that 
$
\frac{\partial }{\partial t} \left( p\rho^{-\Gamma} \right) + {\bf v} \cdot \nabla \left( p\rho^{-\Gamma} \right)   
+ (\Gamma-1) \rho^{-\Gamma}  ( {\bf v} \cdot {\bf B} )  \nabla \cdot {\bf B} = 0.
$
At $({\bf x},t)=({\bf 0},0)$, we have $\nabla \left( p\rho^{-\Gamma} \right)={\bf 0}$ and $\nabla \cdot {\bf B}=  d \epsilon  > 0$, which yield  
$
\frac{\partial \left( p\rho^{-\Gamma} \right) }{\partial t}   ({\bf 0},0) = - 0.03 d (\Gamma-1)  \epsilon < 0.
$ 
Note that $p\rho^{-\Gamma} ({\bf 0},0)=0$. 
Thus there exists $t_1 \in [0,t_0)$ such that 
$
p\rho^{-\Gamma} ({\bf 0},t)   < 0,~\forall t \in (0,t_1).
$ 
Because $\rho({\bf x},t)>0$ for all $({\bf x},t) \in \Omega \times [0,t_0)$, we have 
$
p ({\bf 0},t)   < 0,~\forall t \in (0,t_1).
$ 

The above analysis infers that the exact smooth solution of the 
conservative RMHD system \eqref{eq:RMHD} 
may fail to be PCP    
if the divergence-free condition \eqref{eq:DivB} is violated.  
This observation, along with the results in \cite{WuTangM3AS} at the
numerical level, demonstrate
the unity of continuous and discrete objects, and 
clearly reveal the intrinsic connection between the PCP property and divergence-free condition. 
In most of the numerical RMHD schemes including the standard DG methods, the divergence error in magnetic field is 
generally unavoidable, although there exist a few numerical techniques to enforce exactly or globally divergence-free property (e.g., \cite{Li2011,XuLiu2016,Fu2018}). 
On the other hand, 
the standard PCP limiting technique (cf.~\cite{zhang2010b,WuTangM3AS}) with local scaling can destroy the globally divergence-free property. It is therefore difficult to find a numerical technique which can enforce the globally 
divergence-free property and meet the PCP requirement  
at the same time. 
In order to address the above issue, we propose to consider a symmetrizable formulation of the RMHD equations \cite{WuShu2019SISC}
\begin{equation}\label{ModRMHD}
{\bf U}_t 
+ \nabla \cdot {\bf F} ({\bf U})   = - {\bf S} ( {\bf U} )  \nabla \cdot {\bf B}, 
\end{equation}
where 
\begin{equation}\label{eq:source}
{\bf S} ( {\bf U} ) := \left( 0,~  ( 1-|{\bf v}|^2 ) {\bf B}  + ({\bf v} \cdot {\bf B}) {\bf v},~ {\bf v},~ {\bf v} \cdot {\bf B} \right)^\top. 
\end{equation}
The system \eqref{ModRMHD} is analogous to the  Godunov--Powell system \cite{Godunov1972,Powell1994} for the ideal non-relativistic MHD. 
The right-hand side term of \eqref{ModRMHD}  is proportional 
to $\nabla \cdot {\bf B}$. 
This implies, at the continuous level, the two formulations \eqref{ModRMHD} and \eqref{eq:RMHD} are equivalent 
under the condition \eqref{eq:DivB}. 
However, the ``source term'' ${\bf S} ( {\bf U} )  \nabla \cdot {\bf B}$ in \eqref{ModRMHD} 
modifies the character of the equations, 
making 
the system \eqref{ModRMHD} symmetrizable, admit a convex thermodynamic entropy pair, 
and play a key role in designing entropy stable schemes \cite{WuShu2019SISC}. 
These good properties do not hold for the conservative RMHD system \eqref{eq:RMHD}.

Interestingly, we find that the exact smooth solutions of 
the symmetrizable RMHD system \eqref{ModRMHD} 
always retain the desired PCP property at the PDE level, even if the divergence-free condition \eqref{eq:DivB} is not satisfied. 
Consider the initial-value problem of the system 
\eqref{ModRMHD}, for ${\bf x}\in {\mathbb R}^d$ and $t>0$, with initial data 
\begin{equation}\label{eq:initialA}
(\rho,{\bf v}, p,{\bf B}) ( {\bf x}, 0 ) = 
(\rho_0,{\bf v}_0, p_0,{\bf B}_0) ({\bf x}),
\end{equation}
where the magnetic field is not necessarily divergence-free. 
Using the method of characteristics one can show the following result, whose proof is given in Appendix \ref{app:proof}.

\begin{proposition}\label{thm:PCP_PDElevel} 
	Assume the initial data \eqref{eq:initialA} 
	are in $C^1({\mathbb R}^d)$ with $\rho_0({\bf x})>0$, $p_0({\bf x})>0,$ 
	and $|{\bf v}_0( {\bf x} )|<1$, $\forall {\bf x} \in {\mathbb R}^d$. 
	If the initial-value problem of \eqref{ModRMHD} with  \eqref{eq:initialA} has a $C^1$ solution $(\rho,{\bf v}, p,{\bf B}) ( {\bf x}, t )$ for ${\bf x} \in {\mathbb R}^d$ and $0\le t \le T$, 
	then the solution satisfies 
	$$\rho({\bf x},t)>0,~ 
	p({\bf x},t)>0,~ e({\bf x},t)>0,~ |{\bf v}( {\bf x}, t )|<1, \quad \forall {\bf x} \in {\mathbb R}^d,~ \forall t \in [0, T].$$ 
	In addition, if assuming the solution is $C^2$, then it holds
	\begin{equation}\label{maxminDivB}
	\min_{ {\bf x} \in \mathbb R^d}  \frac{\nabla \cdot {\bf B}}{\rho W}  ( {\bf x}, 0 ) \le 
	\frac{\nabla \cdot {\bf B}}{\rho W }  ( {\bf x}, t ) \le \max_{ {\bf x} \in \mathbb R^d}  \frac{\nabla \cdot {\bf B}}{\rho W }  ( {\bf x}, 0 ), \qquad \forall t \in [0, T].
	\end{equation}
\end{proposition}
For smooth solutions of the modified RMHD system \eqref{ModRMHD}, the estimate \eqref{maxminDivB} implies that the ``relative'' divergence $\| 
\rho^{-1} W^{-1} {\nabla \cdot {\bf B}}{ }  ( \cdot, t ) \|_{L^\infty}$ does not grow with $t$.

Analogous to the Powell source term for the ideal non-relativistic MHD system 
\cite{POWELL1999284,WuShu2018,WuShu2019}, 
the source term 
in the symmetrizable RMHD system \eqref{ModRMHD} is 
non-conservative, but is necessary to accommodate the PCP property at 
the PDE level when the divergence-free condition \eqref{eq:DivB} is not exactly satisfied. 
Therefore, in order to achieve the PCP property at the discrete level, our schemes 
in this paper will be constructed using the symmetrizable formulation \eqref{ModRMHD}, which renders additional technical challenges in discretizing the source term properly to ensure its compatibility with the PCP property. 
As mentioned in \cite{WuShu2018,WuShu2019} on the non-relativistic MHD, there is 
a conflict between the PCP property which requires the non-conservative source term, and the conservation property which is lost due to the source term. 
The loss of conservation property leaves the possibility that it 
may lead to incorrect resolutions for some discontinuous problems, which will be investigated carefully in 
a separate study. 

\section{Numerical analysis techniques}

In this section, we will introduce 
several important properties of $\mathcal G$ and derive some technical estimates, 
which will be useful in the PCP analysis of the proposed numerical schemes.

\subsection{Properties of admissible states}\label{sec:states}
Throughout the rest of this paper, we consider a general causal EOS \eqref{eq:gEOS} satisfying 
\begin{equation}\label{eq:EOScond}
\begin{cases}
\mbox{The function $H(p,\rho)$ in \eqref{eq:gEOS} is differentiable in $\mathbb R^+\times \mathbb R^+$}, &
\\
H(p,\rho) \ge \sqrt{1+p^2/\rho^2}+p/\rho, \qquad \qquad \quad~ \forall p, \rho>0,
\\
H(p,\rho) \left(\frac1{\rho} - \frac{\partial H(p,\rho)}{\partial p} \right) < \frac{\partial H(p,\rho)}{\partial \rho} < 0, \qquad \forall p, \rho>0,
\end{cases}
\end{equation}
where the second condition is revealed by the relativistic kinetic theory \cite{WuTang2017ApJS}, and 
the third condition can be derived from the relativistic causality and the assumption that the fluid's coefficient of thermal expansion is positive \cite{WuTang2017ApJS}, which is valid for most compressible fluids, e.g., the gases. 
The conditions in \eqref{eq:EOScond} are valid for the ideal EOS \eqref{eq:iEOS} and some other commonly used EOSs; see \cite{WuTang2017ApJS}.

In order to overcome the challenges arising from the lack of
explicit formulas of the functions in \eqref{eq:RMHD:definitionG}, 
the following two equivalent forms of ${\mathcal G}$ were rigorously derived in \cite{WuTangM3AS} 
for the ideal EOS \eqref{eq:iEOS} and in \cite{WuTangZAMP} for a general EOS \eqref{eq:gEOS} satisfying \eqref{eq:EOScond}.

\begin{lemma}[\bf First equivalent form]\label{theo:RMHD:CYconditionFINAL2}
	The admissible state set~${\mathcal G}$ is {\bf \em equivalent} to the set
	\begin{align}
	{\mathcal G}_1 := \left\{   \vec U=(D,\vec m,\vec B,E)^{\top}:~  D>0,~q(\vec U)>0,~ \Psi (\vec U) > 0 \right\},
	\label{eq:RMHD:definitionG2}
	\end{align}
	where $q(\vec U):= E-\sqrt{D^2+|\vec m|^2}$ and 
	$$
	\Psi (\vec U) := \big( \Phi(\vec U)-2(|\vec B|^2-E) \big) \sqrt{\Phi(\vec U)+|\vec B|^2-E} - \sqrt{ \frac{27}{2} \bigg( D^2|\vec B|^2+(\vec m \cdot \vec B)^2 \bigg)},
	$$
	with ${\Phi(\vec U):}= \sqrt{({|\vec B|^2} - E)^2 + 3({E^2} - {D^2} - |\vec m|^2)}$.
\end{lemma}

\begin{lemma}[\bf Second equivalent form]\label{theo:RMHD:CYcondition:VecN}
	The admissible state set ${\mathcal G}$ or ${\mathcal G}_1$ is {\bf \em equivalent} to the set
	\begin{align}
	\nonumber
	{\mathcal G}_2 := \left\{   \vec U=(D,\vec m,\vec B,E)^{\top}:~  D>0,~ \vec U \cdot
	{{ {\bm \xi}^*}} + {p^*_m} >0,\forall {\bf B}^* \in \mathbb R^3, \forall {\bf v}^* \in \mathbb B_1({\bf 0}) \right\},
	\label{eq:RMHD:CYcondition:VecNG1} 
	\end{align}
	where $\mathbb B_1({\bf 0}):= \{ {\bf x} \in \mathbb R^3: |{\bf x}|<1  \}$ denotes the open unit ball centered at ${\bf 0}$ in $\mathbb R^3$, and  
	\begin{align}\label{eq:RMHD:vecns}
	&{\bm \xi}^* = {\left( - \sqrt {1 - {|\vec v^*|}^2} ,~
		- {\vec v}^*,~ - (1 - {|\vec v^*|}^2) {\vec B}^* - ({\vec v}^* \cdot {\vec B}^*) {\vec v}^*,~1 \right)^{\top}},\\
	& p_{m}^*  = \frac{ (1-{|\vec v^*|}^2) |{\vec B}^*|^2 +({\vec v}^* \cdot {\vec B}^*)^2 }{2}. \label{eq:RMHD:vecns2}
	\end{align}
\end{lemma}

\begin{remark}
	Note that all the constraints in the above two equivalent forms are {\bf \em explicit} with respect to 
	$\bf U$. This is a very helpful feature.  
	The first equivalent form ${\mathcal G}_1$ is particularly useful for checking the admissibility of a given state $\vec U$ and constructing 
	the limiter for developing PCP high-order accurate RMHD schemes. 
	Moreover, the two constraints in the second equivalent form ${\mathcal G}_2$, 
	are both {\bf \em linear} with respect to $\bf U$, although two (additional) auxiliary variables ${\bf B}^*$ and ${\bf v}^*$ are introduced. 
	Such linearity makes ${\mathcal G}_2$ quite suitable for 
	analytically
	verifying the PCP property of RMHD schemes. 
	It will provide a novel quasi-linearization approach to handle nonlinear physical constraints and   
	play an important role in our PCP analysis. 
\end{remark}

It is also proven in \cite{WuTangM3AS} that the admissible state set is a convex set.

\begin{lemma}\label{theo:RMHD:convex}
	The admissible state set ${\mathcal G}_1$ is convex.
\end{lemma}

\subsection{Technical estimates}
In order to handle the effect of the source term in the symmetrizable RMHD system \eqref{ModRMHD} on the PCP property of numerical schemes, we derive the following inequality \eqref{eq:widelyusedWU_EQ}, whose discovery is highly nontrivial.

\begin{lemma}\label{lem:est_su}
	For any ${\bf U} \in {\mathcal G}$, any ${\bf B}^* \in {\mathbb R}^3$ and any ${\bf v}^* \in \mathbb B_1({\bf 0})$, it holds 
	\begin{equation}\label{eq:widelyusedWU_EQ}
	\left|  {\bf S}({\bf U}) \cdot {\bm \xi}^* + {\bf v}^* \cdot  {\bf B}^* \right|  \le \frac{1}{\sqrt{\rho H}} \left( {\bf U} \cdot {\bm \xi}^* + p_m^* \right),
	\end{equation}
	where ${\bm \xi}^*$ and $p_{m}^*$ are defined in \eqref{eq:RMHD:vecns} and \eqref{eq:RMHD:vecns2}, respectively. 
\end{lemma}

\begin{proof}
	We observe that  
	\begin{equation*} 
{\bf S}({\bf U}) \cdot {\bm \xi}^*  + {\bf v}^* \cdot  {\bf B}^* = ( {\bf v} - {\bf v}^* ) \cdot \Big( (1-|{\bf v}|^2) {\bf B} + ({\bf v}\cdot {\bf B} ) {\bf v} - (1-|{\bf v}^*|^2) {\bf B}^* - ({\bf v}^*\cdot {\bf B}^* ) {\bf v}^*  \Big) . 
\end{equation*}
Let $\Pi_1 :=  {\bf U} \cdot {\bm \xi}^* + p_m^*$ and 	
$$
\Pi_2 := \sqrt{\rho H} ( {\bf v} - {\bf v}^* ) \cdot \Big( (1-|{\bf v}|^2) {\bf B} + ({\bf v}\cdot {\bf B} ) {\bf v} - (1-|{\bf v}^*|^2) {\bf B}^* - ({\bf v}^*\cdot {\bf B}^* ) {\bf v}^*  \Big).
$$ 
Then, we need to prove  
\begin{equation}\label{eq:2wk}
 \Pi_1 \ge |\Pi_2|.
\end{equation}
We reformulate $\Pi_1$ and split it into two parts as follows: 
	\begin{align*}
	\Pi_1 
	& = \rho H W^2 (1-{\bf v}\cdot {\bf v}^*) - p
	- \rho W \sqrt{1- |{\bf v}^*|^2 } 
	\\
	& \quad + \Big( |{\bf B}|^2 {\bf v} - ({\bf v}\cdot {\bf B}) {\bf B}  \Big) \cdot ( - {\bf v}^* ) 
	+ \Big( (1-|{\bf v}^*|^2) {\bf B}^* + ({\bf v}^*\cdot {\bf B}^*) {\bf v}^*  \Big) \cdot  ( - {\bf B} )
	\\
	& \quad  + \frac{  (1+|{\bf v}|^2) |{\bf B}|^2 - ( {\bf v}\cdot {\bf B} )^2  }2
	+ \frac{ (1 - |{\bf v}^*|^2) |{\bf B}^*|^2 + ( {\bf v}^* \cdot {\bf B}^* )^2 } 2
	\\
	& = \left[ \rho H W^2 (1-{\bf v}\cdot {\bf v}^*) - p
	- \rho W \sqrt{1- |{\bf v}^*|^2 }  \right]
	\\
	& \quad + \left[ \frac{ (1-|{\bf v}^*|^2) | {\bf B} - {\bf B}^* |^2 } 2 
	+ \frac{ |{\bf v} - {\bf v}^*|^2 | {\bf B} |^2 } 2 
	- \frac{  ( ( {\bf v} - {\bf v}^* )\cdot  {\bf B} )^2 } 2 
	+ \frac{ ( {\bf v}^* \cdot ( {\bf B} - {\bf B}^* ) )^2 } 2 \right]
	\\
	& =: \Pi_1^{(1)} + \Pi_{1}^{(2)}.
	\end{align*} 
According to the second condition in \eqref{eq:EOScond}, 
the first part $\Pi_1^{(1)}$ satisfies 
	\begin{align*}
	\frac{	\Pi_1^{(1)} }{\rho H} & =   W^2 (1-{\bf v}\cdot {\bf v}^*) - \frac{p}{\rho H}
	- \frac{1}{H}  W \sqrt{1- |{\bf v}^*|^2 }
	 \ge  W^2 (1-{\bf v}\cdot {\bf v}^*) - \frac{ \frac{H^2-1}{ 2H} \rho }{\rho H}
	- \frac{1}{H}  W \sqrt{1- |{\bf v}^*|^2 }
	\\
	& =  W^2 (1-{\bf v}\cdot {\bf v}^*)  - \frac12 
	+ \frac{1}2 \left( \frac{1}{H} -  W \sqrt{1- |{\bf v}^*|^2 } \right)^2 - \frac{ W^2 ( 1-|{\bf v}^*|^2 ) } 2
	\\
	& \ge W^2 (1-{\bf v}\cdot {\bf v}^*)  - \frac12 
	- \frac{ W^2 ( 1-|{\bf v}^*|^2 ) } 2 = \frac12 W^2 | {\bf v} - {\bf v}^* |^2. 
	\end{align*}
	It follows that 
	$$
	\Pi_1 \ge  \frac12 \rho H W^2 | {\bf v} - {\bf v}^* |^2 + \Pi_1^{(2)} := \Pi_3.
	$$
	Therefore, in order to prove \eqref{eq:widelyusedWU_EQ} or \eqref{eq:2wk}, it suffices to show 
	$$
	\Pi_3 \ge | \Pi_2 |.
	$$
	Let us introduce the vector ${\bf Z} =  ( {\bf B}^*, {\bf B}, \sqrt{\rho H} )^\top \in \mathbb R^7 $. We observe that both $\Pi_3$ and $\Pi_2$ can be formulated into 
	quadratic forms in the variables ${\bf Z}$.  
	{\bf \em This highly nontrivial observation is a key of our proof.}  
	Specifically, we discover that 
	$$
	\Pi_3 = \frac12 {\bf Z}^\top {\bf A}_3 {\bf Z}, \qquad \Pi_2 = \frac12 {\bf Z}^\top {\bf A}_2 {\bf Z}
	$$
	with 
	$$
	{\bf A}_3 = \begin{pmatrix}
	{\bf G} & -{\bf G} & {\bf 0}^\top\\
	-{\bf G} & {\bf H} & {\bf 0}^\top\\
	{\bf 0} &  {\bf 0} &    W^2 | {\bf v} - {\bf v}^* |^2
	\end{pmatrix},
	\qquad 
	{\bf A}_2 = \begin{pmatrix}
	{\bf O} & {\bf O} &  {\bf b}_1^\top \\
	{\bf O} & {\bf O} &  {\bf b}_2^\top  \\
	{\bf b}_1 &  {\bf b}_2 &  0
	\end{pmatrix},
	$$
	where ${\bf 0}=(0,0,0)$, ${\bf O}$ is $3\times 3$ zero matrix, and (note that ${\bf v}$ and ${\bf v}^*$ are row vectors)
	\begin{align*}
	&	{\bf G} = ( 1-|{\bf v}^*|^2 ) {\bf I}_3 + ( {\bf v}^* )^\top  {\bf v}^* ,
	\\
	&	{\bf H} = \big(1 + | {\bf v} - {\bf v}^* |^2  - |{\bf v}^*|^2 \big) {\bf I}_3 + ({\bf v}^*)^\top {\bf v}^* -  ( {\bf v} - {\bf v}^* )^\top ( {\bf v} - {\bf v}^* ),
	\\
	&	{\bf b}_1 =   (1-|{\bf v}^*|^2)({\bf v}^*-{\bf v}) + ( |{\bf v}^*|^2-{\bf v} \cdot {\bf v}^* ) {\bf v}^*,
	\\
	&	{\bf b}_2  =   (1-|{\bf v}|^2)({\bf v}-{\bf v}^*) + ( |{\bf v}|^2-{\bf v} \cdot {\bf v}^* ) {\bf v}. 
	\end{align*}
	Then, it suffices to show that both $ {\bf A}_3 + {\bf A}_2  $ and $ {\bf A}_3 - {\bf A}_2  $ are positive semi-definite. 
	
	Note that 
	${\bf G}$ is symmetric, and its eigenvalues are $\{ 1-|{\bf v}^*|^2, 1-|{\bf v}^*|^2, 1 \}$ and all positive, 
	implying the positive definiteness of $\bf G$. 
	Define a nonsingular matrix
	$$
	{\bf P}_1 = \begin{pmatrix}
	{\bf I}_3 & {\bf O} &  {\bf 0}^\top  \\
	{\bf I}_3 & {\bf I}_3 &  {\bf 0}^\top  \\
	- {\bf b}_1 {\bf G}^{-1} &  {\bf 0} &  1
	\end{pmatrix},
	$$
	where $ - {\bf b}_1 {\bf G}^{-1} = {\bf v} - {\bf v}^* $. Then 
\begin{equation}\label{eq:WKLproof3}
	{\bf P}_1 ( {\bf A}_3 + {\bf A}_2 ) {\bf P}_1^\top = 
	\begin{pmatrix}
	{\bf G} & {\bf O} & {\bf 0}^\top  \\
	{\bf O} & {\bf H} - {\bf G}& {\bf b}^\top_1 + {\bf b}_2^\top \\
	{\bf 0} &  {\bf b}_1+ {\bf b}_2   &    W^2 | {\bf v} - {\bf v}^* |^2 - {\bf b}_1 {\bf G}^{-1} {\bf b}_1^\top
	\end{pmatrix},
\end{equation}
	where 
	$$
	{\bf b}_1 + {\bf b}_2 = ( | {\bf v}^* |^2 - {\bf v} \cdot {\bf v}^* ) {\bf v} 
	+ ( | {\bf v} |^2 - {\bf v} \cdot {\bf v}^* ) {\bf v}^*,
	$$
	and the matrix ${\bf H}-{\bf G}$ is symmetric and given by 
	$$
	{\bf H}-{\bf G} = | {\bf v} - {\bf v}^* |^2 {\bf I}_3 - ( {\bf v} - {\bf v}^* )^\top ( {\bf v} - {\bf v}^* ).
	$$
	The eigenvalues of ${\bf H}-{\bf G}$ are $\{ 0, | {\bf v} - {\bf v}^* |^2, | {\bf v} - {\bf v}^* |^2 \}$, which are all nonnegative, implying that ${\bf H}-{\bf G}$ is positive semi-definite. 
	
	Now, we would like to show that ${\bf P}_1 ( {\bf A}_3 + {\bf A}_2 ) {\bf P}_1^\top$ is positive semi-definite. 
	Let us first consider two trivial cases: 
	\begin{itemize}
		\item If ${\bf v} = {\bf v}^*$, then ${\bf P}_1 ( {\bf A}_3 + {\bf A}_2 ) {\bf P}_1^\top = {\bf O}$, which is positive semi-definite.
		\item If ${\bf v} = {\bf 0}$, then ${\bf b}_1 = {\bf v}^*= - {\bf b}_2$ and $W^2 | {\bf v} - {\bf v}^* |^2 - {\bf b}_1 {\bf G}^{-1} {\bf b}_1^\top = 0$. In this case, ${\bf P}_1 ( {\bf A}_3 + {\bf A}_2 ) {\bf P}_1^\top 
		= {\rm diag}\{ {\bf G}, {\bf H}-{\bf G}, 0 \}$, which is positive semi-definite.
	\end{itemize}
	 In the following, we shall focus on the nontrivial case that 
	${\bf v} \neq {\bf v}^*$ and ${\bf v} \neq {\bf 0}$. 
	For any $\varepsilon > 0$, we define 
	$$
	{\bf Q}_\varepsilon = 	\begin{pmatrix}
	& {\bf H} - {\bf G} + \varepsilon {\bf I}_3 & {\bf b}_1^\top+ {\bf b}_2^\top \\
	&  {\bf b}_1 + {\bf b}_2  &    W^2 | {\bf v} - {\bf v}^* |^2 - {\bf b}_1 {\bf G}^{-1} {\bf b}_1^\top
	\end{pmatrix}.
	$$
	Some algebraic manipulations yield that 
	$$
	\det( {\bf Q}_\varepsilon ) = \frac{ \varepsilon }{1- |{\bf v}|^2 } ( \varepsilon + | {\bf v} - {\bf v}^* |^2 ) \Big( \varepsilon \Pi_4 +  |{\bf v}|^2 |{\bf v}-{\bf v}^*|^4  \Big),
	$$
	where 
	$$
	\Pi_4 := ( 1-|{\bf v}|^2 ) \Big(  ( v_1 v_2^* - v_2 v_1^* )^2 + 
	(v_2 v_3^* - v_3 v_2^*)^2 + 
	( v_3 v_1^* - v_1 v_3^*)^2  \Big) +  |{\bf v}|^2  |{\bf v}-{\bf v}^*|^2.
	$$
	It is evident that $\Pi_4 \ge |{\bf v}|^2  |{\bf v}-{\bf v}^*|^2$. 
	For any $\varepsilon > 0$, the matrix ${\bf H} - {\bf G} + \varepsilon {\bf I}_3$ is positive definite, and when 
	${\bf v} \neq {\bf v}^*$ and ${\bf v} \neq {\bf 0}$, it holds 
	$$
	\det( {\bf Q}_\varepsilon ) \ge  \frac{ \varepsilon }{1- |{\bf v}|^2 } ( \varepsilon + | {\bf v} - {\bf v}^* |^2 ) |{\bf v}-{\bf v}^*|^2 |{\bf v}|^2  \Big( \varepsilon    +   |{\bf v}-{\bf v}^*|^2  \Big) > 0.
	$$
	This implies that the leading principal minors of ${\bf Q}_\varepsilon$ are all positive,  and thus ${\bf Q}_\varepsilon$ is positive definite for any $\varepsilon > 0$, ${\bf v} \neq {\bf v}^*$ and ${\bf v} \neq {\bf 0}$. Taking the limit $\varepsilon \to 0$, we obtain that ${\bf Q}_0$ is 
	positive semi-definite, which further yields that ${\bf P}_1 ( {\bf A}_3 + {\bf A}_2 ) {\bf P}_1^\top 
	= {\rm diag} \{ {\bf G}, {\bf Q}_0 \}$ is positive semi-definite, 
	for the nontrivial case (${\bf v} \neq {\bf v}^*$ and ${\bf v} \neq {\bf 0}$). 
	In conclusion, for all the cases, ${\bf P}_1 ( {\bf A}_3 + {\bf A}_2 ) {\bf P}_1^\top$ is positive semi-definite. 
	
	Because ${\bf A}_3 + {\bf A}_2$ and  ${\bf P}_1 ( {\bf A}_3 + {\bf A}_2 ) {\bf P}_1^\top$ are congruent, 
	${\bf A}_3 + {\bf A}_2$ is positive semi-definite. 
	Similar arguments imply that 
	${\bf A}_3 - {\bf A}_2$ is also positive semi-definite. Hence 
	$$
	\Pi_3 \pm \Pi_2 = \frac12 {\bf Z}^\top  ( {\bf A}_3 \pm {\bf A}_2 ) {\bf Z} \ge 0,
	$$ 
	which yields $\Pi_1 \ge \Pi_3 \ge |\Pi_2|$. The proof is complete. 
\end{proof}

We also need the following technical inequality \eqref{eq:RMHD:LLFsplit}, which was constructed in 
\cite{WuTangM3AS}, to handle the effect of flux in numerical PCP analysis. 

\begin{lemma}\label{theo:RMHD:LLFsplit}
	If $\vec U \in {\mathcal G}$, then
	for any $\theta \in [-1,1]$, any ${\vec B}^* \in \mathbb R^3 $ and any ${\bf v}^* \in \mathbb B_1({\bf 0}) $, 
	it holds
	\begin{equation}\label{eq:RMHD:LLFsplit}
	\big( \vec U + \theta \vec F_i(\vec U) \big) \cdot {\bm \xi}^* +  p_{m}^* + \theta \big( v_{i}^* p_{m}^* - B_i (\vec v^* \cdot \vec B^*)\big)\ge 0,
	\end{equation}
	where $i\in\{1,2,\cdots,d\}$, and ${\bm \xi}^*$ and $p_{m}^*$ are defined in \eqref{eq:RMHD:vecns} and \eqref{eq:RMHD:vecns2}, respectively.
\end{lemma}

For any vector 
${\bf n}=(n_1,\cdots,n_d) \in {\mathbb{R}}^d$, we define the inner products
\begin{equation} \label{eq:innerprod}
\langle {\bf n }, {\bf v} \rangle 
:= \sum_{k=1}^d n_k v_k,
\quad ~
\langle {\bf n }, {\bf B} \rangle 
:= \sum_{k=1}^d n_k B_k,
\quad ~ \langle {\bf n }, {\bf F} \rangle 
:= \sum_{k=1}^d n_k {\bf F}_k,
\end{equation}
which will be frequently used in this paper. Then we can generalize Lemma \ref{theo:RMHD:LLFsplit}.

\begin{lemma}\label{theo:RMHD:LLFsplitgen}
	If $\vec U \in {\mathcal G}$, then
	for any $\theta \in [-1,1]$, any ${\vec B}^* \in \mathbb R^3 $, any ${\bf v}^* \in \mathbb B_1({\bf 0}) $, and 
	any unit vector ${\bf n} \in {\mathbb{R}}^d$, 
	it holds
	\begin{equation}\label{eq:RMHD:LLFsplitgen}
	\big( \vec U + \theta \big\langle {\bf n}, \vec F(\vec U) \big \rangle \big) \cdot {\bm \xi}^* +  p_{m}^* + \theta \left( \langle {\bf n}, {\bf v}^*\rangle  p_{m}^* - \langle {\bf n}, {\bf B} \rangle (\vec v^* \cdot \vec B^*)\right)\ge 0.
	\end{equation}
\end{lemma}

\begin{proof}
This can be proven by using Lemma \ref{theo:RMHD:LLFsplit} and the rotational invariance of the RMHD system. The proof is omitted.  	
\end{proof}

\section{Provably PCP DG Schemes}\label{sec:2Dpcp}

In this section, we construct  
PCP high-order DG schemes for 
the multidimensional RMHD based on the symmetrizable form \eqref{ModRMHD}.  
For the sake of clarity, we shall mainly focus on the 2D case ($d=2$), 
keeping in mind that our PCP methods and 
analyses are also extendable to the 3D case ($d=3$). 

\subsection{Outline of the PCP schemes}

Assume that the 2D spatial domain $\Omega$ is partitioned into a mesh ${\mathcal T}_h$, which may be unstructured 
and consists of 
polygonal cells. 
The time interval is also divided into the mesh $\{t_0=0, t_{n+1}=t_n+\Delta t_{n}, 0\le n <N_t\}$
with the time step-size $\Delta t_{n}$ determined by some CFL condition. 
Throughout this section, 
the lower-case letter $k$ is used to denote 
the DG polynomial degree, 
while the 
capital letter $K$ always represents a cell in ${\mathcal T}_h$.

Let ${\bf x}  \in \mathbb R^d$ denote the spatial coordinate vector. We define the locally divergence-free DG finite element 
space \cite{Li2005} 
\begin{equation*}
{\mathbb W}_h^{k} 
= \left\{ {\bf u}=(u_1,\cdots,u_8)^\top:~ u_\ell \big|_{K} \in {\mathbb P}^{k} (K),
\forall \ell,~   
\sum_{i=1}^d \frac{ \partial u_{4+i}}{\partial {x_i}} 
\bigg|_{K}  = 0,~\forall K \in {\mathcal T}_h   \right\},
\end{equation*}
where ${\mathbb P}^{k} (K)$ denotes the space of polynomials, in cell $K$, of total degree up to $k$. 
To define the PCP DG schemes, we also introduce the following two subsets of ${\mathbb W}_h^{k}$: 
\begin{align}
& \overline {\mathbb G}_h^{k} 
:= \left\{ {\bf u} \in {\mathbb W}_h^{k}:~ \frac{1}{|K|} \int_K {\bf u}({\bf x}) d {\bf x} \in {\mathcal G},~~ \forall K \in  {\mathcal T}_h  \right\},
\\  \label{eq:Ghk}
&  {\mathbb G}_h^{k} 
:= \left \{ {\bf u} \in \overline {\mathbb G}_h^{k}:~ {\bf u} \big|_K ( {\bf x} )  \in {\mathcal G},~~ \forall {\bf x} \in \mathbb S_K,~ \forall K \in  {\mathcal T}_h  \right \},
\end{align}
where $|K|$ denotes the area of 
the cell $K$, and $\mathbb S_K$ denotes the set of some critical points in $K$ which will be specified later. 

\begin{definition}
A DG scheme is defined to be PCP if its solutions always stay in ${\mathbb G}_h^{k}$. 
For clarity, if a DG scheme preserves 
the numerical solutions in $\overline {\mathbb G}_h^{k} $, then we say it satisfies a ``weak'' PCP property.
\end{definition}

\begin{lemma}\label{lem:Convex2}
	The sets $\overline {\mathbb G}_h^{k}$ and $ {\mathbb G}_h^{k}$ are both convex. 
	In addition, for any vector function ${\bf u}\in [L^2 (\Omega)]^8$ satisfying ${\bf u}({\bf x}) \in {\mathcal G},~\forall {\bf x} \in \Omega$, we have 
	${\bf P}_{w} ({\bf u}) \in \overline {\mathbb G}_h^{k}  $, where 
	${\bf P}_{w}$ denoting the $L^2$-projection into $\mathbb W_h^k$.
\end{lemma}

\begin{proof}
These conclusions directly follow from that ${\mathcal G}={\mathcal G}_1$ is a convex set, which is implied by Lemmas \ref{theo:RMHD:CYconditionFINAL2} and \ref{theo:RMHD:convex}. The proof is omitted. 
\end{proof}

We aim at constructing PCP high-order accurate DG schemes that always preserve the DG solution function ${\bf U}_h({\bf x},t)$ in the set $ {\mathbb G}_h^{k}$ for all $t\in \{t_n: 0 \le n \le N_t\}$.  This goal will be achieved by following three steps:

\begin{enumerate}
	\item First, we will seek in Sect.~\ref{sec:Lh} a suitable spatial discretization of symmetrizable RMHD system \eqref{ModRMHD}, such that the resulting discrete equation, which can be put in ODE form 
	as $
	\frac{d}{dt} {\bf U}_h = {\bf L}_h ( {\bf U}_h ) 
	$, 
	satisfies the ``weak'' PCP property
	\begin{equation}\label{eq:Lhcondition}
	\mbox{if} ~~
	{\bf U}_h  \in  {\mathbb G}_h^{k}, \qquad \mbox{then} ~~
	{\bf U}_h + \Delta t {\bf L}_h ( {\bf U}_h ) \in \overline {\mathbb G}_h^{k},
	\end{equation}
	under some CFL condition on $\Delta t $. 
	{\em The property \eqref{eq:Lhcondition} is very important. It is extremely nontrivial to find a DG discretization for the RMHD that can be proven to satisfy \eqref{eq:Lhcondition}.} 
	Some traditional methods including  
	standard DG schemes for the conservative RMHD system \eqref{eq:RMHD} do not satisfy \eqref{eq:Lhcondition}.   
	\item Then, we further discretize the ODE system $\frac{d}{dt} {\bf U}_h = {\bf L}_h ( {\bf U}_h )$ in time using a strong-stability-preserving (SSP) explicit Runge-Kutta 
	method \cite{GottliebShuTadmor2001}. 
	\item Finally, 
	a local scaling PCP limiting procedure, which will be introduced in Sect.~\ref{sec:PCPlimiter}, 
	is applied to the intermediate solutions of the Runge-Kutta discretization. 
	This procedure corresponds to an operator 
	$
	{\bf \Pi}_h:  \overline {\mathbb G}_h^{k}  \longrightarrow  {\mathbb G}_h^{k} 
	$, which maps the numerical solutions from $\overline {\mathbb G}_h^{k}$ to $ {\mathbb G}_h^{k}$ and satisfies 
	\begin{equation}\label{eq:conPCP}
	\frac{1}{|K|} \int_K {\bf \Pi}_h ( {\bf u} ) d {\bf x} =  \frac{1}{|K|} \int_K {\bf u} d {\bf x}, \qquad \forall K \in  {\mathcal T}_h,~~\forall 
	{\bf u} \in \overline {\mathbb G}_h^{k}. 
	\end{equation}
	The PCP limiter is required only for high-order DG methods with $k\ge 1$; for the first-order DG method ($k=0$), ${\bf \Pi}_h$ becomes the identity operator. 
\end{enumerate}
Let ${\bf U}_h^{n}$ denote the numerical solution at time $t=t_n$. The resulting fully discrete PCP DG methods, with a $N_r$-stage SSP Runge-Kutta method, can be written in the following form: 
\begin{itemize}
	\item Set ${\bf U}_h^0 = {\bf \Pi}_h {\bf P}_{w} ( {\bf U}( {\bf x}, 0 ) ) $; 
	\item For $n=0,\dots,N_t-1$, compute ${\bf U}_h^{n+1}$ as follows:
		\begin{enumerate}[label=(\roman*)]
			\item set ${\bf U}_h^{(0)} = {\bf U}_h^{n}$;
			\item for $i=1,\dots,N_r$ compute the intermediate solutions:
					\begin{equation}\label{eq:RK}
					{\bf U}_h^{(i)} = {\bf \Pi}_h \left\{ \sum_{\ell=0}^{i-1} 
					\bigg[
					\alpha_{i\ell} \left( {\bf U}_h^{ (\ell) } + \beta_{ i\ell } \Delta t_n {\bf L}_h ( {\bf U}_h^{(\ell)} ) \right)
					\bigg]
					  \right\};
					\end{equation}
			\item set ${\bf U}_h^{n+1} = {\bf U}_h^{(N_r)} $;
		\end{enumerate}	
\end{itemize}
where the SSP Runge-Kutta method has been written into a convex combination of the forward Euler method, and the associated 
parameters
$\alpha_{i\ell}$ and $\beta_{i\ell}$ are all non-negative and satisfy 
$
\sum_{\ell = 0}^{i-1} \alpha_{i\ell} = 1. 
$
Some SSP Runge-Kutta methods can be found in \cite{GottliebShuTadmor2001,SunShu2019}, e.g., a commonly-used three-stage third-order version is given by 
\begin{equation}\label{3SSP}
\begin{split}
&\alpha_{10}=1,~~\alpha_{20}=3/4,~~\alpha_{21}=1/4,~~\alpha_{30}=1/3,~~\alpha_{31}=0,~~\alpha_{32}=2/3,
\\
&\beta_{10}=1,~~\beta_{20}=0,~~\beta_{21}=1,~~\beta_{30}=0,~~\beta_{31}=0,~~\beta_{32}=1.  
\end{split}   
\end{equation}

\begin{remark}
At each Runge-Kutta stage, 
the PCP property of the above fully discrete DG schemes is enforced by the operator ${\bf \Pi}_h$, 
	which can only act on functions in 
	$\overline {\mathbb  G}_h^k$. 
That is, we require the convex combination 
	$
	\sum_{\ell=0}^{i-1} 
	\big[
	\alpha_{i\ell} \big( {\bf U}_h^{ (\ell) } + \beta_{ i\ell } \Delta t_n {\bf L}_h ( {\bf U}_h^{(\ell)} ) \big)
	\big] \in \overline {\mathbb  G}_h^k,
	$
which is guaranteed by the weak PCP property \eqref{eq:Lhcondition} and 
 the convexity of $\overline {\mathbb G}_h^{k}$. 
On the other hand, the PCP limiting operator ${\bf \Pi}_h$ enforces 
	${\bf U}_h^{(\ell)} \in {\mathbb G}_h^{k},~0\le \ell < i$, 
	which provides the condition required by 
	the weak PCP property \eqref{eq:Lhcondition} 
	for the next Runge-Kutta stage evolution. 
	Therefore, the weak PCP property \eqref{eq:Lhcondition} and the PCP limiting operator ${\bf \Pi}_h$ are two key ingredients of the proposed PCP schemes.
\end{remark}

In what follows, we shall describe in detail the operators ${\bf L}_h$ and ${\bf \Pi}_h$, and also specify the point set 
$\mathbb S_K$ in the definition \eqref{eq:Ghk} of $\mathbb G_h^k$. 
We will prove the weak PCP property \eqref{eq:Lhcondition} of the DG spatial discretization in Theorem \ref{thm:Pkge1} and the  PCP property of the 
 fully discrete DG schemes in Theorems \ref{thm:strongPCP1} and \ref{thm:strongPCP}.

\subsection{The operator ${\bf L}_h$ and the weak PCP property}\label{sec:Lh}

We now derive a suitable spatial discretization such that the resulting operator ${\bf L}_h$ satisfies 
the weak PCP property \eqref{eq:Lhcondition}. 
Following our previous work on the ideal non-relativistic MHD \cite{WuShu2018,WuShu2019}, 
we consider the following locally divergence-free DG methods for the symmetrizable RMHD system \eqref{ModRMHD}: 
\begin{align}\nonumber
& \frac{d}{dt} \int_{K} {\bf U}_h({\bf x},t)  \cdot  {\bf u}    d {\bf x}
=  \int_{K}    
{\bf F}  ( {\bf U}_h ) \cdot \nabla {\bf u}  d {\bf x}  
\\ \nonumber
& \quad  - \sum_{ {\mathscr E} \in \partial K }  \int_{ {\mathscr E}  } 
{\bf u}^{{\rm int}(K)} \cdot \bigg\{
\hat{\bf F} \left( {\bf U}_h^{{\rm int}(K)}, {\bf U}_h^{{\rm ext}(K)}; {\bf n}_{{\mathscr E},K} \right)   
\\ 
& \qquad \quad 
+ \left[
\frac12
\left \langle {\bf n}_{{\mathscr E},K}, {\bf B}_h^{{\rm ext}(K)} - {\bf B}_h^{{\rm int}(K)} \right \rangle 
{\bf S} \big({\bf U}_h^{{\rm int}(K)} \big)  \right] \bigg\} ds,\qquad \forall {\bf u} \in {\mathbb W}_h^{k}, \label{eq:2DDGUh}
\end{align} 
where $\partial K$ denotes the boundary of the cell $K$; 
${\bf n}_{{\mathscr E},K}$ is the outward unit normal to the edge ${\mathscr E}$ of $K$; the inner product $\langle \cdot, \cdot \rangle$ is defined in \eqref{eq:innerprod};  
the superscripts ``${\rm int}(K)$'' and ``${\rm ext}(K)$'' indicate that the associated limits of ${\bf U}_h (\bf x)$ at the cell interfaces are taken from the interior and exterior of $K$, respectively. 
In \eqref{eq:2DDGUh}, $\hat{\bf F}$ denotes the numerical flux, which we take as the global Lax-Friedrichs flux 
\begin{equation}\label{eq:LFflux}
\begin{aligned}
&\hat{\bf F} \left( {\bf U}_h^{{\rm int}(K)}, {\bf U}_h^{{\rm ext}(K)}; {\bf n}_{{\mathscr E},K} \right) 
\\
& \quad = \frac12 
\left( 
\left\langle {\bf n}_{{\mathscr E},K}, {\bf F} ({\bf U}_h^{{\rm int}(K)})  + 
 {\bf F} ({\bf U}_h^{{\rm ext}(K)}) \right\rangle  
-   a ( {\bf U}_h^{{\rm ext}(K)} - {\bf U}_h^{{\rm int}(K)} )  \right),
\end{aligned}
\end{equation}
where the numerical viscosity parameter $a$ is taken as the speed of light $c=1$, a simple upper bound of all wave speeds. 
The term inside the square bracket in \eqref{eq:2DDGUh} is derived from a suitable discretization of the source term in the symmetrizable RMHD system \eqref{ModRMHD}, where the locally divergence-free property of ${\bf B}_h$ has been taken into account. 
This term is proportional to the jump of the normal magnetic component across cell interface, 
which is zero for the exact solution and is very small (at the level of truncation error) for numerical solutions. 
The inclusion of this term is crucial for achieving the property \eqref{eq:Lhcondition}, as demonstrated by our theoretical analysis later.

Of course, we have to replace the boundary and element integrals at the right-hand side of \eqref{eq:2DDGUh} by some quadrature rules of sufficiently high-order accuracy (specifically, the algebraic degree of accuracy should be at least $2k$). 
For example, we can approximate the boundary integral by the Gauss quadrature  
with $Q=k+1$ points: 
 \begin{equation*}
\begin{split}
& \int_{ {\mathscr E} } {\bf u}^{{\rm int}(K)} \cdot 
\bigg[
\hat{\bf F} \left( {\bf U}_h^{{\rm int}(K)}, {\bf U}_h^{{\rm ext}(K)}; {\bf n}_{{\mathscr E},K} \right)  
+\frac12 
\left \langle {\bf n}_{{\mathscr E},K}, {\bf B}_h^{{\rm ext}(K)} - {\bf B}_h^{{\rm int}(K)} \right \rangle 
{\bf S} \big({\bf U}_h^{{\rm int}(K)} \big)  \bigg]  ds
\\
& \quad  \approx |{\mathscr E}| \sum_{q=1}^Q \omega_q
{\bf u}^{{\rm int}(K)} ( {\bf x}_{\mathscr E}^{(q)} ) \cdot 
\bigg[
\hat{\bf F} \left( {\bf U}_h^{{\rm int}(K)} ( {\bf x}_{\mathscr E}^{(q)},t ), {\bf U}_h^{{\rm ext}(K)} ( {\bf x}_{\mathscr E}^{(q)},t ); {\bf n}_{{\mathscr E},K} \right)  
\\ 
& \qquad 
+ \frac12
\left \langle {\bf n}_{{\mathscr E},K}, {\bf B}_h^{{\rm ext}(K)} ( {\bf x}_{\mathscr E}^{(q)},t ) - {\bf B}_h^{{\rm int}(K)} ( {\bf x}_{\mathscr E}^{(q)},t ) \right \rangle 
{\bf S} \left({\bf U}_h^{{\rm int}(K)} ( {\bf x}_{\mathscr E}^{(q)},t ) \right)  \bigg],
\end{split}
\end{equation*}
where $|{\mathscr E}|$ denotes the length of the edge ${\mathscr E}$, 
$ \{{\bf x}_{\mathscr E}^{(q)}\}_{1\le q \le Q}$ are the quadrature points on ${\mathscr E}$, 
and $\{\omega_q\}_{1\le q \le Q}$ are the associated weights with $\sum_{q=1}^Q \omega_q = 1$. 
The element integral $\int_{K}    
{\bf F}  ( {\bf U}_h ) \cdot \nabla {\bf u}  d {\bf x} $ can also be approximated by some 2D quadrature
$
|K| \sum_{ q = 1}^{\breve Q} \breve \omega_q   
{\bf F}  ( {\bf U}_h ( \breve {\bf x}_K^{(q)},t ) ) \cdot \nabla {\bf u} ( \breve {\bf x}_K^{(q)} ),
$
where $\breve {\bf x}_K^{(q)}$ and $\breve \omega_q$ denote the quadrature points and weights, respectively.  

Thus, we finally obtain the weak formulation: 
\begin{equation}\label{eq:2DDGUh-weak}
\frac{d}{dt} \int_{K} {\bf U}_h \cdot  {\bf u}    d {\bf x} = 
{\mathcal J}_K ( {\bf U}_h, {\bf u}  ), \qquad  \forall {\bf u} \in {\mathbb W}_h^{k},
\end{equation} 
where ${\mathcal J}_K ( {\bf U}_h, {\bf u}  ) = \sum_{i=1}^3 {\mathcal J}_K^{(i)} ( {\bf U}_h, {\bf u}  )$ with 
\begin{align*}
& {\mathcal J}_K^{(1)}   =  -\sum_{ {\mathscr E} \in \partial K } \left\{ 
|{\mathscr E}| \sum_{q=1}^Q \omega_q
\hat{\bf F} \left( {\bf U}_h^{{\rm int}(K)} ( {\bf x}_{\mathscr E}^{(q)} ), {\bf U}_h^{{\rm ext}(K)} ( {\bf x}_{\mathscr E}^{(q)} ); {\bf n}_{{\mathscr E},K} \right)  \cdot {\bf u}^{{\rm int}(K)} ( {\bf x}_{\mathscr E}^{(q)} ) \right\},
\\ 
& {\mathcal J}_K^{(2)}  = 
 -\frac12  \sum_{ {\mathscr E} \in \partial K } \Bigg\{ 
 |{\mathscr E}| \sum_{q=1}^Q \omega_q  \left \langle {\bf n}_{{\mathscr E},K}, {\bf B}_h^{{\rm ext}(K)} ( {\bf x}_{\mathscr E}^{(q)} ) - {\bf B}_h^{{\rm int}(K)} ( {\bf x}_{\mathscr E}^{(q)} ) \right \rangle 
{\bf S} \left({\bf U}_h^{{\rm int}(K)} ( {\bf x}_{\mathscr E}^{(q)} ) \right) \cdot {\bf u}^{{\rm int}(K)} ( {\bf x}_{\mathscr E}^{(q)} ) 
\Bigg\},
\\
& {\mathcal J}_K^{(3)}  =
|K| \sum_{ q = 1}^{\breve Q} \breve \omega_q   
{\bf F}  ( {\bf U}_h ( \breve {\bf x}_K^{(q)} ) ) \cdot \nabla {\bf u} ( \breve {\bf x}_K^{(q)} ),
\end{align*}
and for notational convenience, the $t$ dependence of all quantities is suppressed hereafter, unless confusion arises otherwise. 
As the standard DG methods (cf.~\cite{CockburnShu1989,CockburnHouShu1990}), the weak form \eqref{eq:2DDGUh-weak} 
can be rewritten in the ODE form as 
\begin{equation}\label{eq:ODE}
\frac{d}{dt} {\bf U}_h = {\bf L}_h ( {\bf U}_h), 
\end{equation}
after choosing a suitable basis of $\mathbb W_h^k$ and representing 
${\bf U}_h$ as a linear combination of the basis functions; see \cite{CockburnShu1989,CockburnHouShu1990} for details. 
Note that the corresponding cell average, denoted by 
$ \bar {\bf U}_K := \frac{1}{|K|} \int_{K} {\bf U}_h d {\bf x}$,  
satisfies the following time evolution equation 
\begin{equation}\label{eq:cellave-eq}
\frac{d}{dt} \bar {\bf U}_K  = 
\widetilde {\bm {\mathcal J}}_K ( {\bf U}_h  ),   \qquad \forall K \in {\mathcal T}_h,
\end{equation}
where $\widetilde {\bm {\mathcal J}}_K ( {\bf U}_h  ) = \widetilde {\bm {\mathcal J}}_K^{(1)} ( {\bf U}_h  ) + \widetilde {\bm {\mathcal J}}_K^{(2)} ( {\bf U}_h  ) $ with 
\begin{align*}
&\widetilde {\bm {\mathcal J}}_K^{(1)} ( {\bf U}_h )  =  - \frac{1}{|K|}  \sum_{ {\mathscr E} \in \partial K } \left\{ 
|{\mathscr E}| \sum_{q=1}^Q \omega_q
\hat{\bf F} \left( {\bf U}_h^{{\rm int}(K)} ( {\bf x}_{\mathscr E}^{(q)} ), {\bf U}_h^{{\rm ext}(K)} ( {\bf x}_{\mathscr E}^{(q)} ); {\bf n}_{{\mathscr E},K} \right)    \right\},
\\ 
&\widetilde {\bm {\mathcal J}}_K^{(2)} ( {\bf U}_h)  = 
-\frac1{2|K|}  \sum_{ {\mathscr E} \in \partial K } \Bigg\{ 
|{\mathscr E}| \sum_{q=1}^Q \omega_q  \left \langle {\bf n}_{{\mathscr E},K}, {\bf B}_h^{{\rm ext}(K)} ( {\bf x}_{\mathscr E}^{(q)} ) - {\bf B}_h^{{\rm int}(K)} ( {\bf x}_{\mathscr E}^{(q)} ) \right \rangle 
{\bf S} \left({\bf U}_h^{{\rm int}(K)} ( {\bf x}_{\mathscr E}^{(q)} ) \right) \Bigg\}.
\end{align*}

We are now in a position to rigorously prove that the above DG spatial discretization 
satisfies 
the weak PCP property \eqref{eq:Lhcondition}. 
To this end,   
  we first need to 
 specify the point set 
 $\mathbb S_K$ in the definition \eqref{eq:Ghk} of $\mathbb G_h^k$. 
Assume that there exists a special 2D quadrature on each cell $K \in {\mathcal T}_h$ satisfying:
\begin{enumerate}[label=(\roman*)]
	\item The quadrature rule is with positive weights and 
	exact for integrals of polynomials of degree up to $k$ on the cell $K$.
	\item The set of the quadrature points, denoted by $\widehat {\mathbb S}_K$, must include all the Gauss quadrature points $ {\bf x}_{\mathscr E}^{(q)}$, $q=1,\dots,Q$, 
	on all the edge ${\mathscr E} \in \partial K$. 
\end{enumerate}
In other words, we would like to have a special quadrature such that 
\begin{equation}\label{eq:decomposition}
\frac{1}{|K|}\int_K u({\bf x}) d{\bf x}  = 
\sum_{{\mathscr E} \in \partial K} \sum_{q=1}^Q \varpi_{{\mathscr E}}^{(q)} u ( {\bf x}_{{\mathscr E}}^{(q)} ) + \sum_{ q=1 }^{\widetilde Q} \widetilde \varpi_q   
u ( \widetilde {\bf x}_K^{(q)} ), \quad \forall 
u \in {\mathbb P}^{k}(K) ,
\end{equation}
where $\{\widetilde {\bf x}_K^{(q)}\}$ are the other (possible) quadrature points in $K$, and the quadrature weights 
$\varpi_{{\mathscr E}}^{(q)}$ and $\widetilde \varpi_q $ 
are positive and satisfy 
$
\sum_{{\mathscr E} \in \partial K} \sum_{q=1}^Q \varpi_{{\mathscr E}}^{(q)} + \sum_{ q=1 }^{\widetilde Q} \widetilde \varpi_q =1.
$
For rectangular cells, such a quadrature was constructed in \cite{zhang2010,zhang2010b} by tensor products of Gauss quadrature and Gauss--Lobatto quadrature. 
For triangular cells and more general polygons, see \cite{zhang2012maximum,VILAR2016416,du2018positivity} 
for how to construct such a special quadrature. 
We remark that this special quadrature is not used for computing any integrals, but only used in the following theoretical PCP analysis and the PCP limiter presented later.

Given this special quadrature, we define 
 the point set 
$\mathbb S_K$ required in \eqref{eq:Ghk} as 
\begin{equation}\label{eq:SK}
\mathbb S_K = \widehat {\mathbb S}_K \cup \breve {\mathbb S}_K,
\end{equation}
where $\breve {\mathbb S}_K := \{ \breve {\bf x}_K^{(q)}, 1\le q \le \breve Q \}$ are the quadrature points involved in ${\mathcal J}_K^{(3)}  $.  
The inclusion of $\breve {\mathbb S}_K$ means that we require 
${\bf U}_h ( \breve {\bf x}_K^{(q)} ) \in \mathcal G$. 
This special requirement does not appear in
the non-relativistic case; it is used here to ensure the existence and uniqueness of the  
physically admissible solution of the nonlinear equation \eqref{eq:RMHD:fU(xi)} and thus obtaining the physical primitive variables from ${\bf U}_h ( \breve {\bf x}_K^{(q)} )$ by \eqref{getprimformU}, so as to successfully compute    
 ${\bf F}  ( {\bf U}_h ( \breve {\bf x}_K^{(q)} ) )$ in ${\mathcal J}_K^{(3)}$. 
Such a consideration is due to 
 that the flux ${\bf F}({\bf U})$ and source ${\bf S}({\bf U})$ cannot be explicitly formulated in terms of ${\bf U}$ for the RMHD and thus must be computed using the corresponding primitive variables. Note that the edge quadrature points $\{ {\bf x}_{\mathscr E}^{(q)} \}$, involved in ${\mathcal J}_K^{(1)}  $ and ${\mathcal J}_K^{(2)} $, are already included in $\widehat {\mathbb S}_K$.

Based on the point set $\mathbb S_K$ defined above, 
we establish the weak PCP property \eqref{eq:Lhcondition} for the high-order  semi-discrete DG scheme \eqref{eq:ODE} as follows. 

\begin{theorem}\label{thm:Pkge1}
	Let $\mathbb G_K^h$ be 
	the set defined by \eqref{eq:Ghk} with $\mathbb S_K$ specified in \eqref{eq:SK}. 
	Then, the weak PCP property \eqref{eq:Lhcondition} holds under the following CFL type condition 
	\begin{equation}\label{eq:CFL:2DMHD}
	\Delta t \frac{| {\mathscr E} |}{|K|} 
	\big( a + \sigma_{K,{\mathscr E},q} ({\bf U}_h ) \big)
	 < \frac{ \varpi_{\mathscr E}^{(q)}}{\omega_q},\quad 1\le q\le Q,~ \forall {\mathscr E} \in \partial K,~ 
	\forall K\in{\mathcal T}_h,
	\end{equation}
	where 
	$\sigma_{K,{\mathscr E},q} ({\bf U}_h ) := \frac12 { \left| \big\langle {\bf n}_{ {\mathscr E}, K },  {\bf B}_{{\mathscr E},q}^{{\rm int}(K)} - {\bf B}_{{\mathscr E},q}^{{\rm ext}(K)}  \big \rangle \right| }/{  \sqrt{ \rho_{{\mathscr E},q}^{{\rm int}(K)}   H_{{\mathscr E},q}^{{\rm int}(K)} } } $ with 
	the shortened notations 
	${\bf U}^{{\rm int}(K)}_{{\mathscr E},q}:=
	{\bf U}_h^{{\rm int}(K)} ( {\bf x}_{\mathscr E}^{(q)} )$ and ${\bf U}^{{\rm ext}(K)}_{{\mathscr E},q}:=
	 {\bf U}_h^{{\rm ext}(K)} ( {\bf x}_{\mathscr E}^{(q)} )$. 
\end{theorem}

\begin{proof}
In order to prove ${\bf U}_h + \Delta t {\bf L}_h ( {\bf U}_h ) \in \overline {\mathbb G}_h^{k}$ in \eqref{eq:Lhcondition}, 
it suffices to show 
\begin{equation}\label{WKLproof222}
\bar {\bf U}_K^{\Delta t} :=   \bar {\bf U}_K + \Delta t \widetilde {\bm {\mathcal J}}_K ( {\bf U}_h  ) \in {\mathcal G},\qquad \forall K \in {\mathcal T}_h,
\end{equation}
under the CFL type condition \eqref{eq:2D1stCFLGP} and the condition that ${\bf U}_h \in \mathbb G_h^k$. 
Substituting the formula of the numerical flux \eqref{eq:LFflux} into $\widetilde {\bm {\mathcal J}}_K^{(1)}({\bf U}_h )$, we reformulate $\widetilde {\bm {\mathcal J}}_K^{(1)}({\bf U}_h )$ and decompose it into three parts: 
\begin{align*}
\widetilde {\bm {\mathcal J}}_K^{(1)}({\bf U}_h ) &=  
\frac{ a } { 2|K| } \sum_{ {\mathscr E} \in \partial K } \left[
|{\mathscr E}| \sum_{q=1}^Q \omega_q
\left(
{\bf U}^{{\rm int}(K)}_{{\mathscr E},q} - \frac{1}{a} \left\langle  {\bf n}_{{\mathscr E},K}, {\bf F}  \big( {\bf U}^{{\rm int}(K)}_{{\mathscr E},q} \big)  \right\rangle
\right)    \right] 
\\
& \quad + \frac{ a } { 2|K| } \sum_{ {\mathscr E} \in \partial K } \left[
|{\mathscr E}| \sum_{q=1}^Q \omega_q
\left(
{\bf U}^{{\rm ext}(K)}_{{\mathscr E},q} - \frac{1}{a} \left\langle  {\bf n}_{{\mathscr E},K}, {\bf F}  \big( {\bf U}^{{\rm ext}(K)}_{{\mathscr E},q} \big)  \right\rangle
\right)    \right] 
\\
& \quad 
- \frac{ a }{ |K| } \sum_{ {\mathscr E} \in \partial K } \left( 
|{\mathscr E}| \sum_{q=1}^Q \omega_q
{\bf U}^{{\rm int}(K)}_{{\mathscr E},q} 
\right)  =: \widetilde {\bm {\mathcal J}}_K^{(1,1)} + \widetilde {\bm {\mathcal J}}_K^{(1,2)} + \widetilde {\bm {\mathcal J}}_K^{(1,3)}.
\end{align*}
Then $\bar {\bf U}_K^{\Delta t}$ can be rewritten as 
\begin{equation}\label{eq:Decompose}
\bar {\bf U}_K^{\Delta t} = {\bf \Xi}_1 + {\bf \Xi}_2 + {\bf \Xi}_3 + {\bf \Xi}_4,
\end{equation}
with , ${\bf \Xi}_i:= \Delta t \widetilde {\bm {\mathcal J}}_K^{(1,i)}$, $i=1,2$, ${\bf \Xi}_3: = \bar {\bf U}_K + \Delta t \widetilde {\bm {\mathcal J}}_K^{(1,3)} $ and 
$$
{\bf \Xi}_4 := \Delta t \widetilde {\bm {\mathcal J}}_K^{(2)} = \frac{\Delta t}{2|K|}  \sum_{ {\mathscr E} \in \partial K } \left[ 
|{\mathscr E}| \sum_{q=1}^Q \omega_q  \left \langle {\bf n}_{{\mathscr E},K}, {\bf B}^{{\rm int}(K)}_{{\mathscr E},q} - {\bf B}^{{\rm ext}(K)}_{{\mathscr E},q} \right \rangle 
{\bf S} \left({\bf U}^{{\rm int}(K)}_{{\mathscr E},q} \right) \right].
$$
Since ${\mathcal G}={\mathcal G}_2$ as shown in Lemma \ref{theo:RMHD:CYcondition:VecN}, 
it remains to prove  
$\bar {\bf U}_K^{\Delta t} \in {\mathcal G}_2$,~$\forall K \in {\mathcal T}_h$.

We fist show $\bar D_K^{\Delta t}>0$.  
Because ${\bf U}_h \in \mathbb G_h^k$ and $\widehat {\mathbb S}_K \subset \mathbb S_K$, we have 
${\bf U}^{{\rm int}(K)}_{{\mathscr E},q} \in {\mathcal G}$ and ${\bf U}^{{\rm ext}(K)}_{{\mathscr E},q} \in {\mathcal G}$ for all $1\le q\le Q$, ${\mathscr E} \in \partial K$ and 
$K \in {\mathcal T}_h$. 
Note that the first component of 
${\bf U}^{{\rm int}(K)}_{{\mathscr E},q} - \frac{1}{a} \big\langle  {\bf n}_{{\mathscr E},K}, {\bf F}  \big( {\bf U}^{{\rm int}(K)}_{{\mathscr E},q} \big)  \big\rangle$ 
 equals   
$D^{{\rm int}(K)}_{{\mathscr E},q} \big( 1 - \frac{1}{a} \big\langle  {\bf n}_{{\mathscr E},K},  {\bf v}^{{\rm int}(K)}_{{\mathscr E},q}  \big\rangle \big) 
\ge D^{{\rm int}(K)}_{{\mathscr E},q} \big( 1 - \frac1{a} \big|  {\bf v}^{{\rm int}(K)}_{{\mathscr E},q} \big| \big)>0$, which implies that the first component of ${\bf \Xi}_1$ is positive. 
Similarly, we know that the first component of ${\bf \Xi}_2$ is also positive. 
Notice that the first component of ${\bf \Xi}_4$ is zero. 
Therefore, the first component of $\bar {\bf U}_K^{\Delta t}$ is larger than that of ${\bf \Xi}_3$. It gives   
\begin{align*}
\bar D_K^{\Delta t} &> \bar D_K -  \frac{ a \Delta t }{ |K| } \sum_{ {\mathscr E} \in \partial K } \left( 
|{\mathscr E}| \sum_{q=1}^Q \omega_q
D^{{\rm int}(K)}_{{\mathscr E},q} 
\right)
\\
& = \sum_{{\mathscr E} \in \partial K} \sum_{q=1}^Q \varpi_{{\mathscr E}}^{(q)} D^{{\rm int}(K)}_{{\mathscr E},q}  + \sum_{ q=1 }^{\widetilde Q} \widetilde \varpi_q   
D_h^{{\rm int}(K)} ( \widetilde {\bf x}_K^{(q)} ) -  \frac{ a \Delta t }{ |K| } \sum_{ {\mathscr E} \in \partial K } \left( 
|{\mathscr E}| \sum_{q=1}^Q \omega_q
D^{{\rm int}(K)}_{{\mathscr E},q} 
\right)
\\
& \ge \sum_{{\mathscr E} \in \partial K} \sum_{q=1}^Q \left[ \omega_q D^{{\rm int}(K)}_{{\mathscr E},q}  
\left( 
\frac{ \varpi_{{\mathscr E}}^{(q)}  }{ \omega_q }
- a \Delta t \frac{ |{\mathscr E}| }{ |K| }
\right) \right] \ge 0,
\end{align*}
where we have used in the above equality the exactness of the quadrature rule \eqref{eq:decomposition} for polynomials of degree up to $k$,  and in the last inequality the condition 
\eqref{eq:CFL:2DMHD}. 

We then prove that $\bar {\bf U}_K^{\Delta t}  \cdot
{{ {\bm \xi}^*}} + {p^*_m} >0$ for any auxiliary variables ${\bf B}^* \in \mathbb R^3 $ and ${\bf v}^* \in \mathbb B_1({\bf 0})$. It follows from \eqref{eq:Decompose} that 
\begin{equation}\label{WKLproof33}
\bar {\bf U}_K^{\Delta t} \cdot
 {\bm \xi}^* + {p^*_m} = I_1 + I_2 + I_3 + I_4,
\end{equation}
with $I_1:={\bf \Xi }_1 \cdot {\bm \xi}^* + \eta $, $I_2:={\bf \Xi }_1 \cdot {\bm \xi}^* + \eta$, 
$I_3:= {\bf \Xi }_3 \cdot {\bm \xi}^* + p^*_m - 2\eta $, $I_4 := {\bf \Xi }_4 \cdot {\bm \xi}^*$, and $\eta := \frac{ a \Delta t }{ 2 |K| } \sum _{ {\mathscr E} \in \partial K } 
|{\mathscr E}| p_m^* $. 
We now estimate the lower bounds of $I_i$ for $1\le i \le 4$ respectively. 
Using Lemma \ref{theo:RMHD:LLFsplitgen}, we deduce that 
\begin{align} \nonumber
I_1 & = \frac{ a \Delta t } { 2|K| } \sum_{ {\mathscr E} \in \partial K } \left\{ 
|{\mathscr E}| \sum_{q=1}^Q \omega_q \left[
\left(
{\bf U}^{{\rm int}(K)}_{{\mathscr E},q} - \frac{1}{a} \left\langle  {\bf n}_{{\mathscr E},K}, {\bf F}  \big( {\bf U}^{{\rm int}(K)}_{{\mathscr E},q} \big)  \right\rangle
\right) \cdot {\bm \xi}^*  + p_m^*  \right] \right\}
\\ \nonumber
& \ge \frac{ a \Delta t } { 2|K| } \sum_{ {\mathscr E} \in \partial K } \left\{ 
|{\mathscr E}| \sum_{q=1}^Q \omega_q  \left[ \frac1{a} \Big(
\left \langle {\bf n}_{ {\mathscr E}, K }, {\bf v}^* \right \rangle p_m^* - 
\left \langle {\bf n}_{ {\mathscr E}, K }, {\bf B}^{{\rm int}(K)}_{{\mathscr E},q} \right \rangle ( {\bf v}^* \cdot {\bf B}^* )
 \Big) \right] \right\}
 \\ \nonumber
& = \frac{ \Delta t } { 2|K| } \sum_{ {\mathscr E} \in \partial K } \left\{ 
|{\mathscr E}| \sum_{q=1}^Q \omega_q   \Big( - 
\left \langle {\bf n}_{ {\mathscr E}, K }, {\bf B}^{{\rm int}(K)}_{{\mathscr E},q} \right \rangle ( {\bf v}^* \cdot {\bf B}^* )
\Big) \right\}
\\ \label{eq:I1bound}
& = - \frac{  \Delta t  ( {\bf v}^* \cdot {\bf B}^* ) } { 2|K| } \sum_{ {\mathscr E} \in \partial K } 
\int_{ {\mathscr E} } \left \langle {\bf n}_{ {\mathscr E}, K }, {\bf B}^{{\rm int}(K)}_h \right \rangle d s 
=: - \frac{  \Delta t  ( {\bf v}^* \cdot {\bf B}^* ) } { 2|K| } {\rm div}_K^{\rm int} {\bf B}_h,
\end{align}
where we have used the exactness of the $Q$-point quadrature rule on 
each interface for polynomials of degree up to $k$. 
Similarly, we obtain 
\begin{equation}\label{eq:I2bound}
I_2 \ge - \frac{  \Delta t  ( {\bf v}^* \cdot {\bf B}^* ) } { 2|K| } \sum_{ {\mathscr E} \in \partial K } 
\int_{ {\mathscr E} } \left \langle {\bf n}_{ {\mathscr E}, K }, {\bf B}^{{\rm ext}(K)}_h \right \rangle d s =: - \frac{  \Delta t  ( {\bf v}^* \cdot {\bf B}^* ) } { 2|K| } {\rm div}_K^{\rm ext} {\bf B}_h. 
\end{equation}
Note  
$
I_3 = \bar{\bf U}_K \cdot {\bm \xi}^* + p_m^* - \frac{ a \Delta t }{ |K| } \sum_{ {\mathscr E} \in \partial K } \left( 
|{\mathscr E}| \sum_{q=1}^Q \omega_q
\Big( {\bf U}^{{\rm int}(K)}_{{\mathscr E},q} \cdot {\bm \xi}^* + p_m^*   \Big)
\right). 
$ Based on the exactness of the quadrature rule 
\eqref{eq:decomposition} for polynomials of degree up to $k$, one has 
\begin{align*}
\bar{\bf U}_K \cdot {\bm \xi}^* + p_m^* & = \sum_{{\mathscr E} \in \partial K} \sum_{q=1}^Q \varpi_{{\mathscr E}}^{(q)}
 \left( {\bf U}^{{\rm int}(K)}_{{\mathscr E},q} \cdot {\bm \xi}^* + p_m^* \right) + \sum_{ q=1 }^{\widetilde Q} \widetilde \varpi_q   
\left( {\bf U}_h^{{\rm int}(K)} ( \widetilde {\bf x}_K^{(q)} ) \cdot {\bm \xi}^* + p_m^* \right)
\\
& \ge \sum_{{\mathscr E} \in \partial K} \sum_{q=1}^Q \varpi_{{\mathscr E}}^{(q)}
\left( {\bf U}^{{\rm int}(K)}_{{\mathscr E},q} \cdot {\bm \xi}^* + p_m^* \right),
\end{align*}
where ${\bf U}_h^{{\rm int}(K)} ( \widetilde {\bf x}_K^{(q)} )  \in {\mathcal G}={\mathcal G}_2$ has been used. 
It follows that 
\begin{equation}\label{eq:I3bound}
I_3 \ge \sum_{{\mathscr E} \in \partial K} \sum_{q=1}^Q 
\left( {\bf U}^{{\rm int}(K)}_{{\mathscr E},q} \cdot {\bm \xi}^* + p_m^* \right) \left( 
 \varpi_{{\mathscr E}}^{(q)} - a \Delta t \omega_q  \frac{ |{\mathscr E}| }{ |K| }
 \right).
\end{equation}
Thanks to the inequality \eqref{eq:widelyusedWU_EQ} constructed in Lemma \ref{lem:est_su}, we have 
\begin{equation*}
b (  {\bf S}({\bf U}) \cdot {\bm \xi}^*)  \ge - b(  {\bf v}^* \cdot  {\bf B}^* ) - \frac{|b|}{\sqrt{\rho H}} \left( {\bf U} \cdot {\bm \xi}^* + p_m^* \right), \qquad \forall b \in \mathbb R,~~\forall {\bf U} \in {\mathcal G}.
\end{equation*}
It follows that 
\begin{equation*}
\begin{aligned}
&  \left \langle {\bf n}_{{\mathscr E},K}, {\bf B}^{{\rm int}(K)}_{{\mathscr E},q} - {\bf B}^{{\rm ext}(K)}_{{\mathscr E},q} \right \rangle 
{\bf S} \big({\bf U}^{{\rm int}(K)}_{{\mathscr E},q} \big) \cdot {\bm \xi}^* 
 \ge   \left \langle {\bf n}_{{\mathscr E},K}, {\bf B}^{{\rm ext}(K)}_{{\mathscr E},q} - {\bf B}^{{\rm int}(K)}_{{\mathscr E},q} \right \rangle  (  {\bf v}^* \cdot {\bf B}^*  )
\\
& \qquad \qquad   -   \left(   \rho_{{\mathscr E},q}^{{\rm int}(K)}   H_{{\mathscr E},q}^{{\rm int}(K)} \right)^{-\frac12}   \left| \left \langle {\bf n}_{{\mathscr E},K}, {\bf B}^{{\rm int}(K)}_{{\mathscr E},q} - {\bf B}^{{\rm ext}(K)}_{{\mathscr E},q} \right \rangle \right| 
\left(  {\bf U}^{{\rm int}(K)}_{{\mathscr E},q}  \cdot {\bm \xi}^*  
+ p_m^*
 \right).   
\end{aligned}
\end{equation*}
Let $I_5 := - \frac{\Delta t}{|K|}  \sum_{ {\mathscr E} \in \partial K } 
\sum_{q=1}^Q \omega_q 
\sigma_{K,{\mathscr E},q}
\big(  {\bf U}^{{\rm int}(K)}_{{\mathscr E},q}  \cdot {\bm \xi}^*  
+ p_m^*
\big)$. We then obtain a lower bound for $I_4$: 
\begin{align} \nonumber
I_4 & \ge   \frac{\Delta t}{2|K|}  \sum_{ {\mathscr E} \in \partial K } \left[
|{\mathscr E}| \sum_{q=1}^Q \omega_q \left \langle {\bf n}_{{\mathscr E},K}, {\bf B}^{{\rm ext}(K)}_{{\mathscr E},q} - {\bf B}^{{\rm int}(K)}_{{\mathscr E},q} \right \rangle  (  {\bf v}^* \cdot {\bf B}^*  ) \right] + I_5
\\ \label{eq:I4bound}
&=\frac{  \Delta t  ( {\bf v}^* \cdot {\bf B}^* ) } { 2|K| } \left(  {\rm div}_K^{\rm ext} {\bf B}_h - {\rm div}_K^{\rm int} {\bf B}_h  \right) + I_5. 
\end{align}
Thanks to the locally divergence-free 
property of  
${\bf B}_h ({\bf x}) $, we have 
\begin{equation}\label{eq:LDF}
{\rm div}_K^{\rm int} {\bf B}_h = \sum_{ {\mathscr E} \in \partial K } \int_{ {\mathscr E} } \left \langle {\bf n}_{ {\mathscr E}, K }, {\bf B}^{{\rm int}(K)}_h \right \rangle d s  = \int_K  \nabla \cdot {\bf B}_h^{ {\rm int}(K) } ( {\bf x}) d {\bf x} = 0,
\end{equation}
where Green's theorem has been used. Combining the estimates \eqref{eq:I1bound}--\eqref{eq:I4bound} and using \eqref{eq:LDF} and \eqref{WKLproof33}, 
we obtain 
\begin{align*}
& \bar {\bf U}_K^{\Delta t} \cdot
{\bm \xi}^* + {p^*_m}  \ge \sum_{{\mathscr E} \in \partial K} \sum_{q=1}^Q 
\left( {\bf U}^{{\rm int}(K)}_{{\mathscr E},q} \cdot {\bm \xi}^* + p_m^* \right) \left( 
\varpi_{{\mathscr E}}^{(q)} - a \Delta t \omega_q  \frac{ |{\mathscr E}| }{ |K| }
\right) + I_5
\\
& \quad  = \sum_{{\mathscr E} \in \partial K} \sum_{q=1}^Q  \omega_q
\left( {\bf U}^{{\rm int}(K)}_{{\mathscr E},q} \cdot {\bm \xi}^* + p_m^* \right) \left[
\frac{\varpi_{{\mathscr E}}^{(q)} }{ \omega_q } - \Delta t \frac{| {\mathscr E} |}{|K|} 
( a + \sigma_{K,{\mathscr E},q} )
\right] > 0,
\end{align*}
where the condition \eqref{eq:CFL:2DMHD} has been used in the last inequality. 
Therefore, we have  
$$
\bar {\bf U}_K^{\Delta t}  \cdot
{{ {\bm \xi}^*}} + {p^*_m} >0, \qquad \forall {\bf B}^* \in \mathbb R^3,~\forall {\bf v}^* \in \mathbb B_1({\bf 0}),
$$
which, along with $\bar D_K^{\Delta t}>0$, yield $\bar {\bf U}_K^{\Delta t} \in {\mathcal G}_2 = {\mathcal G}$. 
The proof is complete. 
\end{proof}

\begin{remark}
The quantities ${\rm div}_K^{\rm int} {\bf B}_h$ 
and ${\rm div}_K^{\rm ext} {\bf B}_h$, defined in the lower bounds in \eqref{eq:I1bound} and 
\eqref{eq:I2bound} respectively, denote 
 discrete divergence. They are also defined in \cite{WuTangM3AS} to quantify the 
influence of divergence error on the PCP property of the standard DG schemes, 
for which the discrete divergence-free condition ${\rm div}_K^{\rm ext} {\bf B}_h={\rm div}_K^{\rm int} {\bf B}_h=0$ is required.  
However, the present DG schemes are proven to be PCP without requiring such discrete divergence-free condition, 
thanks to two key ingredients:  
the locally divergence-free DG element and a 
suitable discretization of the source term in the symmetrizable RMHD system \eqref{ModRMHD} which gives ${\mathcal J}_K^{(2)} ( {\bf U}_h, {\bf u}  )$ in \eqref{eq:2DDGUh-weak}.
The former leads to zero divergence within each cell, so that  
the term 
${\rm div}_K^{\rm int} {\bf B}_h$ vanishes. 
The latter brings some new divergence terms, as shown in the lower bound in 
\eqref{eq:I4bound}, which exactly offset the divergence term in \eqref{eq:I2bound}. 
	In other words, these key ingredients help 
	eliminate the effect of divergence error on 
	the PCP property. 
	This feature is similar to the continuous case 
	that the inclusion of 
	source ${\bf S} ({\bf U})  \nabla \cdot {\bf B} $ makes  
	the modified RMHD system \eqref{ModRMHD} 
	able to retain the PCP property even if the magnetic field is not divergence-free. 
	Again, these findings indicate the unity of discrete and continuous objects. 
\end{remark}

For the first-order DG method ($k=0$), we have ${\bf U}_h \big|_K ({\bf x}) \equiv  \bar {\bf U}_K$ and ${\mathbb G}_h^k = \overline {\mathbb G}_h^k$ so that the PCP and weak PCP properties are equivalent in this case, 
and the PCP property can be proven under a sharper CFL condition as shown in Theorem \ref{thm:P0}. 

\begin{theorem}\label{thm:P0}
	For the first-order version ($k=0$) of the semi-discrete DG scheme \eqref{eq:2DDGUh-weak} or \eqref{eq:ODE}, the PCP property \eqref{eq:Lhcondition} holds under the following CFL type condition 
	\begin{equation}\label{eq:2D1stCFLGP}
	\Delta t \left( \frac{a}{|K|} 
	\sum_{ {\mathscr E} \in \partial K }   \big| {\mathscr E} \big|
	+ \frac{ \left|{\rm div} _{K}  {\bf B}_h\right| }{ \sqrt{\bar \rho_K \bar H_K } } \right) < 1,\quad \forall K \in {\mathcal T}_h, 
	\end{equation}
	where ${\rm div} _{K} {\bf B}_h$ denotes a discrete divergence of ${\bf B}_h$ on the cell $K$ defined by 
	\begin{equation}\label{eq:DefDisDivB}
	\mbox{\rm div} _{K}  {\bf B}_h := \frac{1}{|K|}
	\sum_{{\mathscr E} \in \partial K} \big| {\mathscr E} \big|
	\left\langle {\bf n}_{{\mathscr E},K},  
	\frac{  \bar {\bf B}_K  +  \bar {\bf B}_{ K_{\mathscr E} } }{
		2
	} \right\rangle,
	\end{equation} 
	with $K_{\mathscr E}$ denoting the adjacent cell that shares the edge ${\mathscr E}$ with the cell $K$. In \eqref{eq:2D1stCFLGP}--\eqref{eq:DefDisDivB}, the notations $\bar \rho_K$, 
	$\bar H_K$ and $\bar {\bf B}_K$ denote the rest-mass density, specific enthalpy, 
	and magnetic field corresponding to $\bar {\bf U}_K$, respectively.
\end{theorem}

\begin{proof}
The proof is similar to that of Theorem \ref{thm:Pkge1} and is thus omitted.  
\end{proof}

\subsection{The PCP limiting operator ${\Pi}_h$}\label{sec:PCPlimiter}
We now present the PCP limiting operator 
$
{\bf \Pi}_h:  \overline {\mathbb G}_h^{k}  \longrightarrow  {\mathbb G}_h^{k} 
$, which limits the numerical solutions from $\overline {\mathbb G}_h^{k}$ to $ {\mathbb G}_h^{k}$ via a simple scaling PCP limiter  
\cite{WuTangM3AS} as extension of the positivity-preserving limiter \cite{zhang2010b}. 
For any ${\bf U}_h \in \overline {\mathbb G}_h^{k}$, we construct the limited solution $ {\bf \Pi}_h {\bf U}_h=: \widetilde {\bf U}_h \in  {\mathbb G}_h^{k} $ as follows.

Let ${\bf U}_h \big|_{K} =: {\bf U}_K({\bf x})$. Note $\bar {\bf U}_K \in {\mathcal G}={\mathcal G}_1,~\forall K\in {\mathcal T}_h$. 
To avoid the effect of the rounding error, we introduce a sufficiently small positive number $\epsilon$ such that 
$ \bar {\vec U}_K \in {\mathcal G}_\epsilon$ for all $K \in {\mathcal T}_h$, where
$
{\mathcal G}_\epsilon = \left\{   \vec U=(D,\vec m,\vec B,E)^{\top}:~  D\ge\epsilon,~q(\vec U)\ge\epsilon,~{\Psi}_\epsilon(\vec U) \ge 0\right\}
$ is a convex set \cite{WuTangM3AS}, 
with
$
{\Psi}_\epsilon(\vec U) : = {\Psi}(\vec U_\epsilon)$ and $\vec U_\epsilon : = \big(D,\vec m,\vec B,E-\epsilon\big)^{\top}.
$
For each $K$, to construct $\widetilde {\bf U}_K({\bf x}):= \widetilde {\bf U}_h \big|_{K}$, we proceed as follows.  
First, we define $\widehat {\vec U}_K({\bf x}) :=  \big( \widehat D_K({\bf x}), \vec m_K({\bf x}),\vec B_K({\bf x}), E_K({\bf x}) \big)^{\top}$ with 
$
\widehat D_K({\bf x}) = \theta_1 \big( D_K({\bf x}) - \bar D_K \big) + \bar D_K$, and $ 
\theta_1 = \min \big \{ 1, ({ \bar D_K - \epsilon })/\big( { \bar D_K - \min \limits_{ {\bf x} \in {\mathbb S}_K}^{} D_K ( {\bf x} ) } \big)  \big \}.
$
Then, we define $\check {\vec U}_K({\bf x}): = \big(  \theta_2 ( \widehat {D}_K ({\bf x}) - \bar {D}_K) + \bar {D}_K,~
\theta_2 \big(  {\vec m}_K ({\bf x}) - \bar {\vec m}_K \big) + \bar {\vec m}_K,~ {\vec B}_K ({\bf x}),~
\theta_2 \big(  {E}_K ({\bf x}) - \bar {E}_K \big) + \bar {E}_K \big)^\top$ 
with 
$\theta_2 = \min \big \{ 1, \big( q(\bar {\bf U}_K) - \epsilon \big) /\big( { q(\bar {\bf U}_K) - \min \limits_{ {\bf x} \in {\mathbb S}_K}^{}   q(\check {\bf U}_K ({\bf x}) ) } \big)  \big \}$. 
Finally, we define 
\begin{equation}\label{eq:PCPpolynomial}
\widetilde {\vec U}_K({\bf x}) = \theta_3 \big( \check {\vec U}_K ({\bf x}) - \bar {\vec U}_K \big) + \bar {\vec U}_K,
\end{equation}
where $\theta_3 = \min \limits_{{\bf x}\in {\mathbb S}_K} \tilde  \theta({\bf x})$. Here $\tilde  \theta({\bf x})=1$ if ${\Psi}_\epsilon ( \check {\vec U}_K({\bf x}) ) \ge 0$; otherwise
 $\tilde \theta({\bf x}) \in [0,1) $ solves 
$
{\Psi}_\epsilon \big( (1- \tilde  \theta) \bar{\vec U}_K + \tilde  \theta \check {\vec U}_K({\bf x}) \big) =0,
$ 
which has a unique solution for the unknown $\tilde  \theta \in [0,1)$. 

\begin{lemma}\label{lem:Pih}
For any ${\bf U}_h \in \overline {\mathbb G}_h^{k}$, one has $ {\bf \Pi}_h {\bf U}_h = \widetilde {\bf U}_h \in  {\mathbb G}_h^{k}$.
\end{lemma}

\begin{proof}
The above procedure indicates that, for $\forall K \in {\mathcal T}_h$,  
the limited solution defined by \eqref{eq:PCPpolynomial} satisfies 
$\widetilde {\vec U}_K({\bf x}) \in {\mathcal G}_\epsilon  \subset {\mathcal G}_1= {\mathcal G},~\forall {\bf x}\in \mathbb S_K,$
and $ \frac1{|K|}  \int_K \widetilde {\vec U}_K d {\bf x} = \bar{\bf U}_K$.  
Besides, the limited magnetic field $\widetilde {\vec B}_K({\bf x})$ keeps locally divergence-free within $K$.  
\end{proof}

\begin{remark}
The PCP limiting operator ${\bf \Pi}_h$ keeps the conservativeness \eqref{eq:conPCP}. Such a limiter does not destroy the high-order accuracy; see  
\cite{zhang2010,zhang2010b,ZHANG2017301}.
\end{remark}

\subsection{The PCP property of fully discrete schemes}
The PCP property of our fully discrete Runge-Kutta DG scheme \eqref{eq:RK} is proven in the following theorems. 

\begin{theorem}\label{thm:strongPCP1}
Assume that ${\bf U}_h^{(0)}={\bf U}_h^n \in \mathbb G_h^k$, then the solutions ${\bf U}_h^{(i)}$, $1\le i \le N_r$ computed by the proposed DG scheme \eqref{eq:RK} belong to $\mathbb G_h^k$, under the CFL condition  
\begin{equation}\label{eq:final-CFL}
 \Delta t_n \le \min_{i,\ell } \frac{  \varpi_{\mathscr E}^{(q)} |K| }{ \beta_{i\ell}   \omega_q \left( a + \sigma_{K,{\mathscr E},q} ({\bf U}_h^{(i)} ) \right) | {\mathscr E} | }, \quad 1\le q\le Q, \forall {\mathscr E} \in \partial K, 
 \forall K\in{\mathcal T}_h. 
\end{equation}
\end{theorem}

\begin{proof}
We prove it by the second principle of mathematical induction for $i$. The hypothesis implies ${\bf U}_h^{(i)} \in \mathbb G_h^k$ for $i=0$. 
Assume that ${\bf U}_h^{(\ell)} \in \mathbb G_h^k$, $1\le \ell \le i-1$. Thanks to the weak PCP property \eqref{eq:Lhcondition} in Theorem \ref{thm:Pkge1}, we have ${\bf U}_h^{ (\ell) } + \beta_{ i\ell } \Delta t_n {\bf L}_h ( {\bf U}_h^{(\ell)} ) \in 
\overline {\mathbb G}_h^k$, $1\le \ell \le i-1$ under the CFL condition \eqref{eq:final-CFL}. 
The convexity of $\overline {\mathbb G}_h^{k}$ in Lemma \ref{lem:Convex2} implies $\sum_{\ell=0}^{i-1} 
\big[
\alpha_{i\ell} \big( {\bf U}_h^{ (\ell) } + \beta_{ i\ell } \Delta t_n {\bf L}_h ( {\bf U}_h^{(\ell)} ) \big)
\big] \in 
\overline {\mathbb G}_h^k$. Since the PCP limiting operator ${\bf \Pi}_h$ maps the 
numerical solutions from $\overline {\mathbb G}_h^{k}$ to $ {\mathbb G}_h^{k}$, 
we obtain ${\bf U}_h^{(i)} \in \mathbb G_h^k$ by \eqref{eq:RK}. Using the principle of induction, we have ${\bf U}_h^{(i)} \in \mathbb G_h^k$ for all $i \in \{0,1,\dots, N_r\}$.
\end{proof}

\begin{theorem}\label{thm:strongPCP}
Under the CFL condition \eqref{eq:final-CFL}, 
the proposed fully discrete Runge-Kutta DG scheme \eqref{eq:RK} always preserves ${\bf U}_h^n \in \mathbb G_h^k$ for all $n \in \mathbb N$.
\end{theorem}

\begin{proof}
Since ${\bf P}_{w} ({\bf U}({\bf x},0)) \in \overline {\mathbb G}_h^{k}  $ as indicated by Lemma \ref{lem:Convex2}, we known 
${\bf U}_h^0 \in \mathbb G_h^k$. 
With the help of Theorem \ref{thm:strongPCP1}, 
we obtain the conclusion by induction for $n$.  
\end{proof}

\subsection{Illustration of some details on Cartesian meshes} Assume that the mesh is rectangular with cells $\{[x_{i-1/2},x_{i+1/2}]\times [y_{\ell-1/2},y_{\ell+1/2}] \}$ 
and spatial step-sizes $\Delta x_i=x_{i+1/2}-x_{i-1/2}$ and $\Delta y_\ell=y_{\ell+1/2}-y_{\ell-1/2}$ in $x$- and $y$-directions respectively. 
Let ${\mathbb S}_i^x=\{ x_i^{(q)}  \}_{q=1}^Q$
and ${\mathbb S}_\ell^y=\{ y_\ell^{(q)}  \}_{q=1}^Q$ 
denote the $Q$-point Gauss quadrature nodes in the intervals 
$[x_{i-1/2},x_{i+1/2}]$ and 
$[y_{\ell-1/2},y_{\ell+1/2}]$ respectively. 
For the cell $K=[x_{i-1/2},x_{i+1/2}]\times [y_{\ell-1/2},y_{\ell+1/2}]$, 
 the point set $\breve {\mathbb S}_K$ required in \eqref{eq:SK} is ${\mathbb S}_i^x \otimes 
{\mathbb S}_\ell^y$, and 
the set $\widehat {\mathbb S}_K$ is given by (cf.~\cite{zhang2010})
\begin{equation}\label{eq:RectS}
\widehat {\mathbb S}_K = \big(  \widehat{\mathbb S}_i^x \otimes 
{\mathbb S}_\ell^y \big) \cup \big(  {\mathbb S}_i^x \otimes 
\widehat{\mathbb S}_\ell^y \big),
\end{equation}
where 
$\widehat{\mathbb S}_i^x=\{ \widehat x_i^{(\mu)}  \}_{\mu=1}^{\tt L}$
and $\widehat {\mathbb S}_\ell^y=\{\widehat y_\ell^{(\mu)}  \}_{\mu=1}^{\tt L}$ 
denote the $\tt L$-point (${\tt L} \ge \frac{k+3}2$) Gauss--Lobatto quadrature nodes in the intervals 
$[x_{i-1/2},x_{i+1/2}]$ and 
$[y_{\ell-1/2},y_{\ell+1/2}]$ respectively.  
With $\widehat {\mathbb S}_K$ in \eqref{eq:RectS}, a special quadrature \cite{zhang2010} satisfying \eqref{eq:decomposition} can be constructed:
\begin{equation} \label{eq:U2Dsplit}
\begin{split}
\frac{1}{|K|}\int_K u({\bf x}) d {\bf x}
&= \frac{ \Delta x_i \widehat \omega_1}{ \Delta x_i + \Delta y_\ell } \sum \limits_{q = 1}^{ Q}  \omega_q \left(  
u\big( x_i^{(q)},y_{\ell-\frac12} \big) 
+ u\big( x_i^{(q)},y_{\ell+\frac12} \big) 
\right)
\\
&
+ \frac{ \Delta y_\ell \widehat \omega_1}{ \Delta x_i + \Delta y_\ell } \sum \limits_{q = 1}^{ Q}  \omega_q \left( u \big(x_{i-\frac12},y_\ell^{(q)}\big) +
u\big(x_{i+\frac12},y_\ell^{(q)}\big) \right)
\\
& + \frac{\Delta x_i}{\Delta x_i + \Delta y_\ell} \sum \limits_{\mu = 2}^{{\tt L}-1} \sum \limits_{q = 1}^{Q}  \widehat \omega_\mu \omega_q  u\big(  x_i^{(q)},\widehat y_\ell^{(\mu)} \big)
\\
& + \frac{\Delta y_\ell}{\Delta x_i + \Delta y_\ell}  \sum \limits_{\mu = 2}^{{\tt L}-1} \sum \limits_{q = 1}^{ Q}  \widehat \omega_\mu \omega_q  u\big(\widehat x_i^{(\mu)},y_\ell^{(q)}\big),\quad~ \forall u \in {\mathbb P}^k(K),
\end{split}
\end{equation}
where $\{\widehat w_\mu\}_{\mu=1}^{\tt L}$ are the weights of the $\tt L$-point Gauss--Lobatto quadrature. 
If labeling the bottom, right, top and left edges of $K$ as  
${\mathscr E }_1$, ${\mathscr E }_2$, ${\mathscr E }_3$ and ${\mathscr E }_4$, respectively, then \eqref{eq:U2Dsplit} implies 
$
\varpi_{ {\mathscr E }_j  }^{(q)} = \frac{ \Delta x_i \widehat \omega_1 \omega_q}{ \Delta x_i + \Delta y_\ell },~j=1,3;~
\varpi_{ {\mathscr E }_j  }^{(q)} = \frac{ \Delta y_\ell \widehat \omega_1 \omega_q}{ \Delta x_i + \Delta y_\ell },~j=2,4.
$ 
According to Theorem \ref{thm:strongPCP1}, the CFL condition \eqref{eq:final-CFL} for our PCP schemes on Cartesian meshes is 
\begin{equation}\label{eq:CFL:2DRMHD}
  \Delta t_n
\left( \frac{1}{\Delta x_i} 
+ \frac{1}{\Delta y_\ell} \right) 
 < \min_{m,s,q} \frac{ \widehat \omega_1}{ \beta_{ms} \big( a + \sigma_{K,{\mathscr E}_j,q} ({\bf U}_h^{(m)} ) \big) },\quad \forall
K\in{\mathcal T}_h,~1\le j \le 4,
\end{equation}
where $\widehat  \omega_1= \frac{1}{{\tt L}({\tt L}-1)}$.  
Since $\sigma_{K,{\mathscr E}_j,q} ({\bf U}_h^{(m)} )$ depends on the numerical solutions at intermediate Runge-Kutta stages, it is difficult to rigorously enforce the condition \eqref{eq:CFL:2DRMHD}. 
Note that $\sigma_{K,{\mathscr E}_j,q} ({\bf U}_h^{(m)} )$ is proportional to the jump of the normal magnetic component across cell interface, 
which is zero for the exact solution. Thus $\sigma_{K,{\mathscr E}_j,q}$ is small and at the level of truncation error. Thus we suggest $\Delta t_n = \frac{  C_{\rm cfl} }{ a \max_{m,s} \beta_{ms} } 
\big( \frac{1}{\Delta x_i} 
+ \frac{1}{\Delta y_\ell} \big)^{-1}$, with 
the CFL number $C_{\rm cfl}$ (slightly) smaller than $\widehat \omega_1$, which works robustly in our numerical tests. For the third-order SSP Runge-Kutta method \eqref{3SSP}, $\max_{m,s} \beta_{ms} =1 $.

\section{Numerical tests}\label{sec:examples}

This section conducts numerical tests on several 2D challenging 
 RMHD problems  with either strong discontinuities, low plasma-beta $\beta:=p/p_m$, or low rest-mass density or pressure, to demonstrate our theoretical analysis, as well as  
the accuracy, high-resolution and robustness of the proposed PCP methods.  
Without loss of generality, we focus on the proposed PCP third-order ($k=2$) DG methods on uniform 
Cartesian meshes, with the third-order SSP Runge-Kutta time discretization \eqref{eq:RK}--\eqref{3SSP}. Unless otherwise stated, all 
the computations are restricted to the ideal EOS \eqref{eq:iEOS} with $\gamma=5/3$, 
and the CFL number is set as 0.15.

\subsection{Smooth problems}
Two smooth problems are tested to check the accuracy of our
method. The first one is similar to those simulated in \cite{WuTang2015,QinShu2016}, and  
its exact solution is $\vec (\rho,{\bf v},{\bf B},p)(x,y,t)=(1+0.9999999 \sin \big(2 \pi (x+y-1.1t)\big), 0.9, 0.2, 0, 1, 1, 1, 10^{-2} ),$
which describes a RMHD sine wave (with very low density and low pressure) fast propagating 
in the domain $\Omega=[0,1]^2$ with a large velocity $|{\bf v}|\approx 0.922 c$.  
The second problem describes Alfv\'en waves propagating periodically in $\Omega=[0,\sqrt{2}]^2$ with a speed of $0.9c$ higher 
than that in \cite{ZhaoTang2017}. The exact solution of this problem is given by 
$\rho(x,y,t)=1$, $v_1(x,y,t) = -0.9 \sin( 2\pi (\varsigma + t / \kappa) ) \sin \alpha$, 
$v_2(x,y,t)=0.9\sin( 2\pi (\varsigma + t / \kappa) ) \cos \alpha$, 
$v_3(x,y,t)=0.9\cos( 2\pi (\varsigma + t / \kappa) )$, $B_1(x,y,t)=\cos \alpha + \kappa v_1(x,y,t)$, $B_2(x,y,t)=\sin \alpha + \kappa v_2(x,y,t)$, $B_3(x,y,t)= \kappa v_3(x,y,t)$, $p(x,y,t)=0.1$, 
where $\kappa = \sqrt{1+\rho H W^2}$ and $\varsigma = x \cos \alpha + y \sin \alpha$ with $\alpha = \pi/4$. 

In the computations, the domain $\Omega$ is divided into $N\times N$ uniform rectangular cells 
with $N \in \{ 10, 20,40,80,160,320, 640 \}$, and periodic boundary conditions are used. 
\figref{fig:smooth} shows the numerical errors at $t=1$ 
in the numerical solutions computed by the PCP third-order DG method at different grid resolutions. 
It is seen that the magnitudes of the errors are reduced as we refine the mesh.  
Moreover, the expected third-order convergence rate is observed, indicating that 
our discretization of the added source term in the symmetrizable RMHD system \eqref{ModRMHD} 
and the PCP limiting procedure 
both maintain the desired accuracy, as expected.

\begin{figure}[htbp]
	\centering
	\begin{subfigure}[b]{0.48\textwidth}
	\begin{center}
		\includegraphics[width=1.0\linewidth]{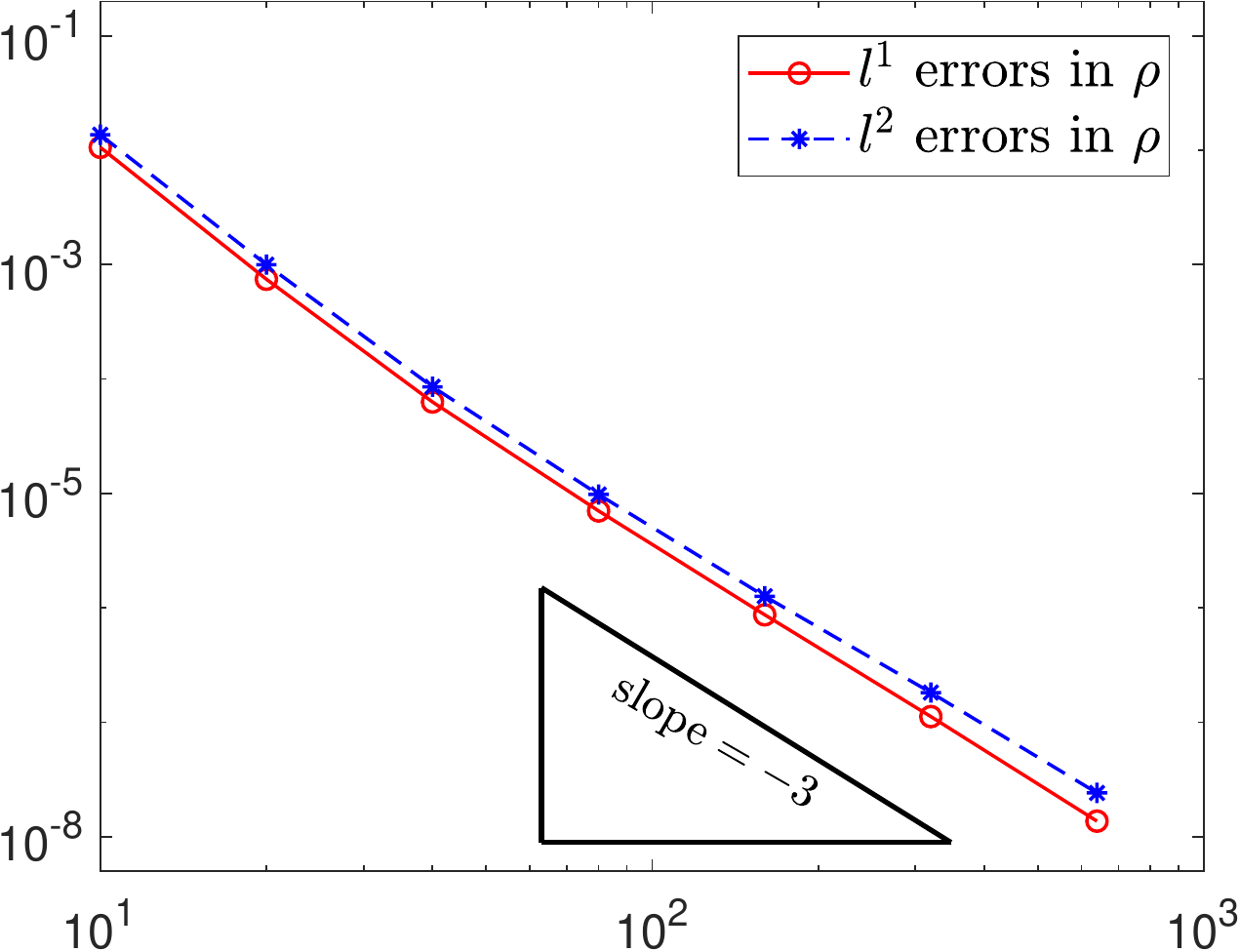}
	\end{center}
\end{subfigure}
\begin{subfigure}[b]{0.48\textwidth}
	\begin{center}
		\includegraphics[width=1.0\linewidth]{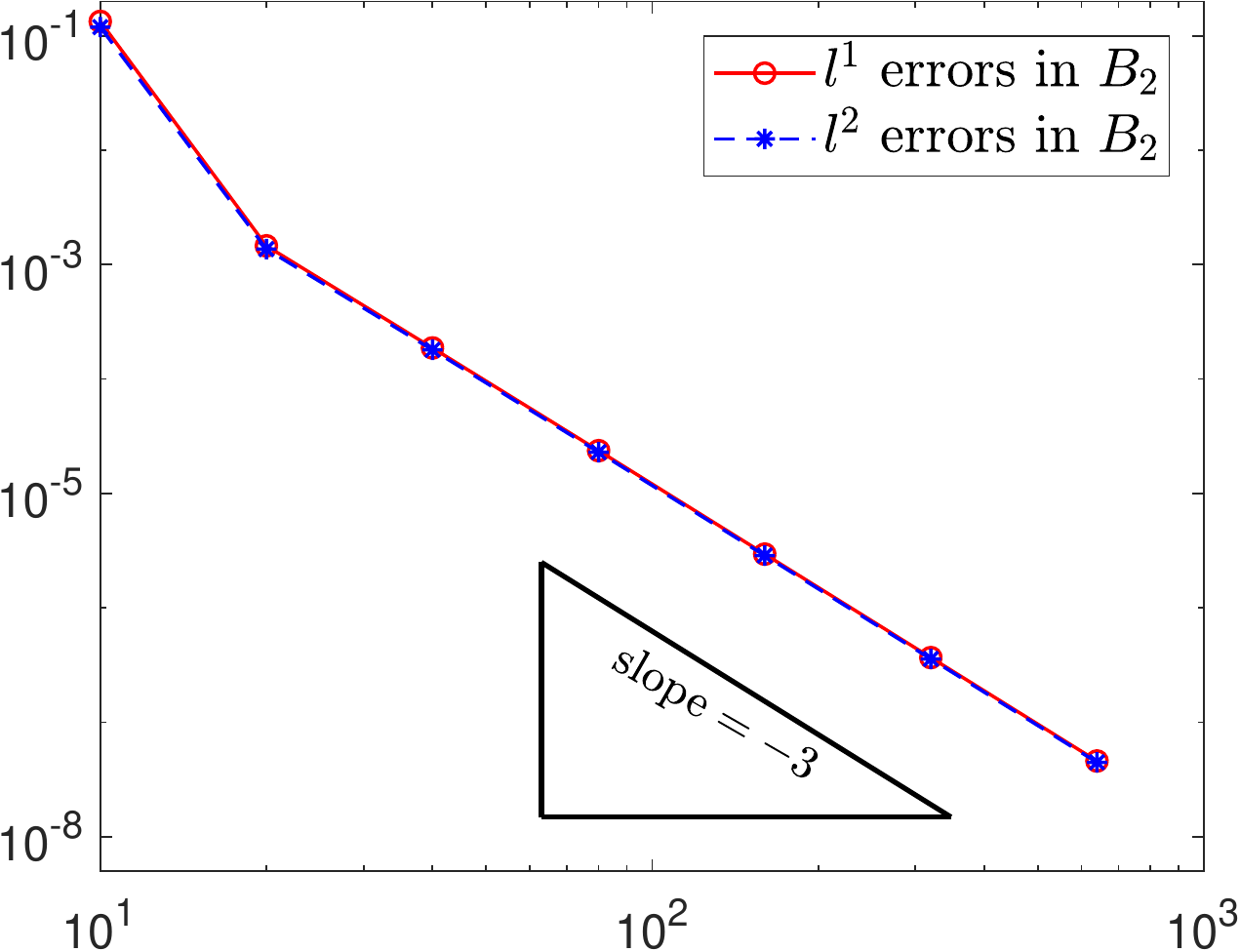}
	\end{center}
\end{subfigure}	
	\captionsetup{belowskip=-8pt}
	\caption{\small Numerical errors in $l^1$ and $l^2$ norms at $t=1$ for the first smooth problem (left) and the second smooth problem (right). The horizontal axis represents the value of $N$.}
	\label{fig:smooth}
\end{figure}

To verify the capability of the proposed PCP methods in resolving complicated wave configurations, we will simulate  
an Orszag-Tang problem, three blast problems  
and two astrophysical jets. 
For these problems, before the PCP limiting procedure, the WENO limiter \cite{Qiu2005} with locally divergence-free WENO reconstruction \cite{ZhaoTang2017} is implemented with the local characteristic decomposition 
to enhance the numerical stability of high-oder DG methods in resolving the strong discontinuities. The WENO limiter is only used 
in the ``trouble'' cells adaptively detected by the indicator in \cite{Krivodonova}.

\subsection{Orszag-Tang problem}
This test simulates an Orszag-Tang problem 
for the RMHD \cite{Host:2008}. 
Initially, the domain $\Omega = [0,2\pi]^2$ is filled 
with relativistically 
hot gas, and periodic boundary conditions are used.
We set the adiabatic index $\Gamma = 4/3$, 
the initial rest-mass density $\rho = 1$ and thermal pressure $p=10$. 
The initial velocity field of the fluid is 
$
{\bf v}(x,y,0)=   ( - A \sin(y), A\sin(x), 0   ),
$
where the parameter $A=0.99/\sqrt{2}$ so that the maximum velocity is $0.99c$ (corresponding Lorentz factor $\approx 7.09$). The magnetic field is initialized at ${\bf B}(x,y,0)=(-\sin y, \sin(2x), 0)$.  
Although the initial solution is smooth,
complicated wave structures are formed as time increases, and turbulence
behavior will be produced eventually. \figref{fig:OT} gives the numerical results computed by the third-order PCP method 
on $600\times 600$ uniform grids. 
One can see that the complicated flow structures are well captured by our method with high resolution and agree  
with those presented in \cite{Host:2008,WuShu2019SISC}. 
In this test, we observe
that it is necessary to enforce the DG solution in $\mathbb G_n^k$ by the PCP limiting procedure, otherwise the
code would break down at time $t \approx 1.98$.

\begin{figure}[htbp]
	\centering
	\begin{subfigure}[b]{0.48\textwidth}
		\begin{center}
			\includegraphics[width=0.9\linewidth]{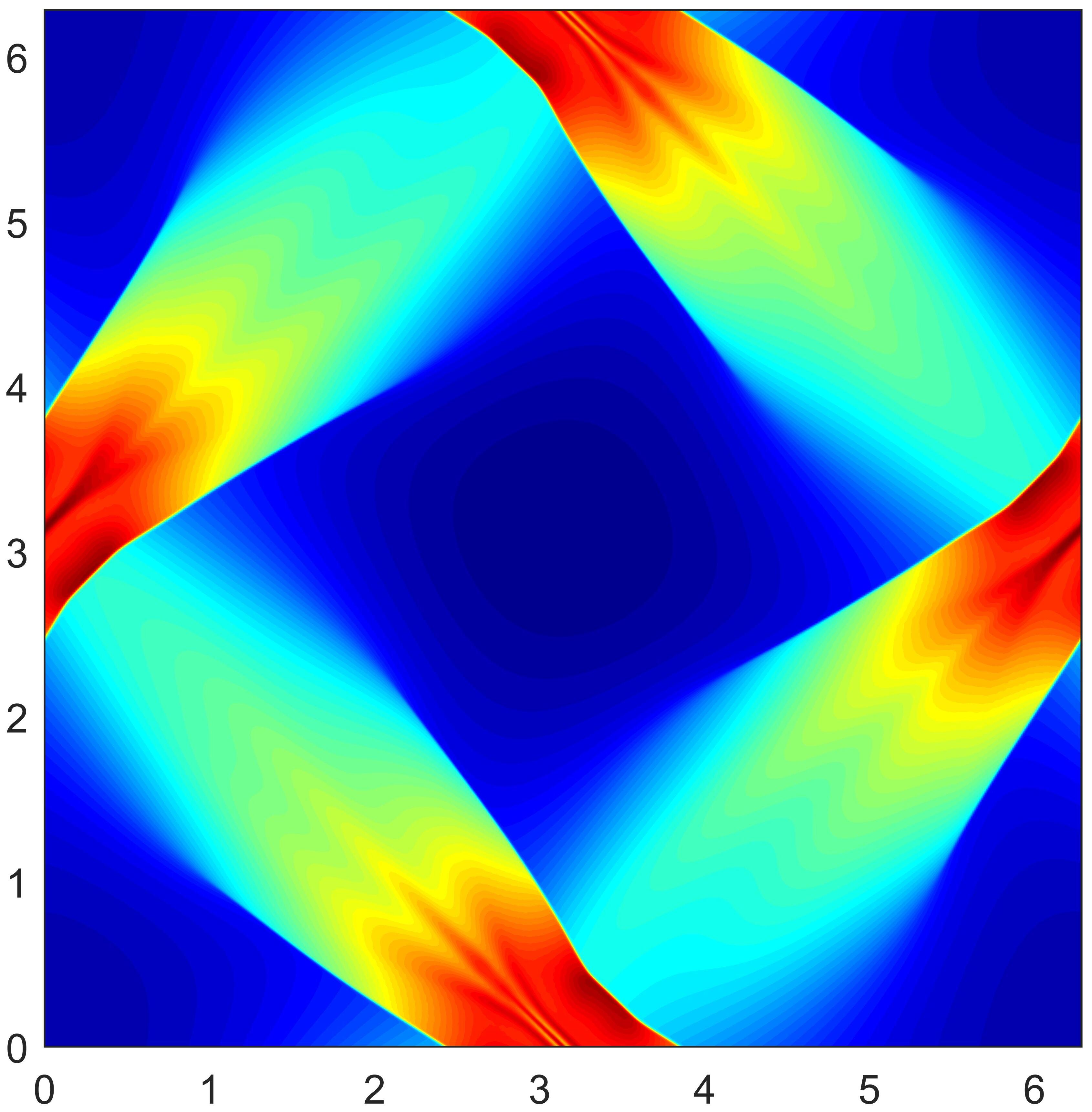}
		\end{center}
	\end{subfigure}
	\begin{subfigure}[b]{0.48\textwidth}
		\begin{center}
			\includegraphics[width=0.9\linewidth]{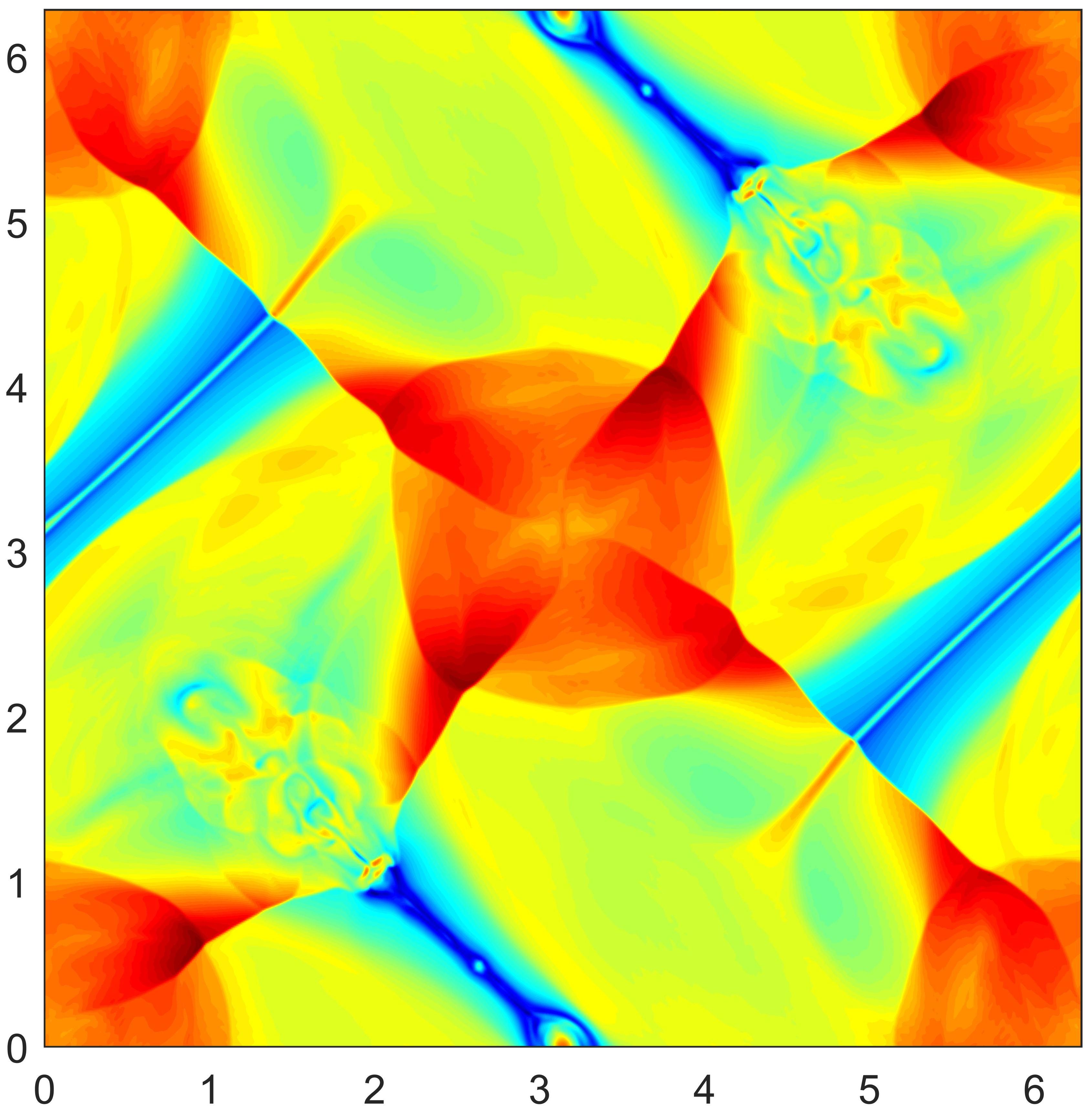}
		\end{center}
	\end{subfigure}
	\captionsetup{belowskip=-17pt}
	\caption{\small Schlieren images of $\log(\rho)$ at $t = 2.818127$ (left) and $t = 6.8558$ (right) for the Orszag-Tang problem.}
	\label{fig:OT}
\end{figure}

\subsection{Blast problems}
Blast problem is a benchmark test for RMHD numerical schemes. 
Simulating a RMHD blast problem with strong magnetic field is difficult, because nonphysical
quantities, e.g., negative pressure, are very 
likely to be produced in the 
numerical simulation. 
Our setup is similar to those in \cite{MignoneHLLCRMHD,del2007echo,BalsaraKim2016,Zanotti2015}. Initially, the domain $\Omega =[-6,6]^2$ is filled with a homogeneous gas at rest with adiabatic index $\Gamma=4/3$.
The explosion zone ($r<0.8$) has a density of $10^{-2}$ and a pressure of $1$, while the ambient medium ($r>1$) has
a density of $10^{-4}$ and a pressure of $p_a=5\times 10^{-4}$, where $r=\sqrt{x^2+y^2}$.
A linear taper is applied to the density and pressure for $r\in[0.8,1]$. The magnetic field is initialized in the $x$-direction as $B_a$.
As $B_a$ is set larger, the initial ambient magnetization becomes higher ($\beta_a:=p_a/p_m$ becomes lower) and this test becomes more challenging. In the literature \cite{MignoneHLLCRMHD,del2007echo,BalsaraKim2016}, $B_a$ is usually specified as 0.1, which corresponds to a moderate magnetized case ($\beta_a=0.1$).
A more strongly magnetized case with $B_a=0.5$ was tested in \cite{Zanotti2015}, corresponding to a lower plasma-beta $\beta_a=4\times 10^{-3}$.
Many existing methods in the literature require some artificial treatments for the strongly magnetized case; see e.g., \cite{komissarov1999godunov,MignoneHLLCRMHD,del2007echo}. It was reported in \cite{del2007echo} that the RMHD code {\tt ECHO} was not able
to run this test with $B_a>0.1$ if no ad hoc numerical strategy was employed.

\begin{figure}[htbp]
	\centering
	\begin{subfigure}[b]{0.48\textwidth}
		\begin{center}
			\includegraphics[width=0.88\linewidth]{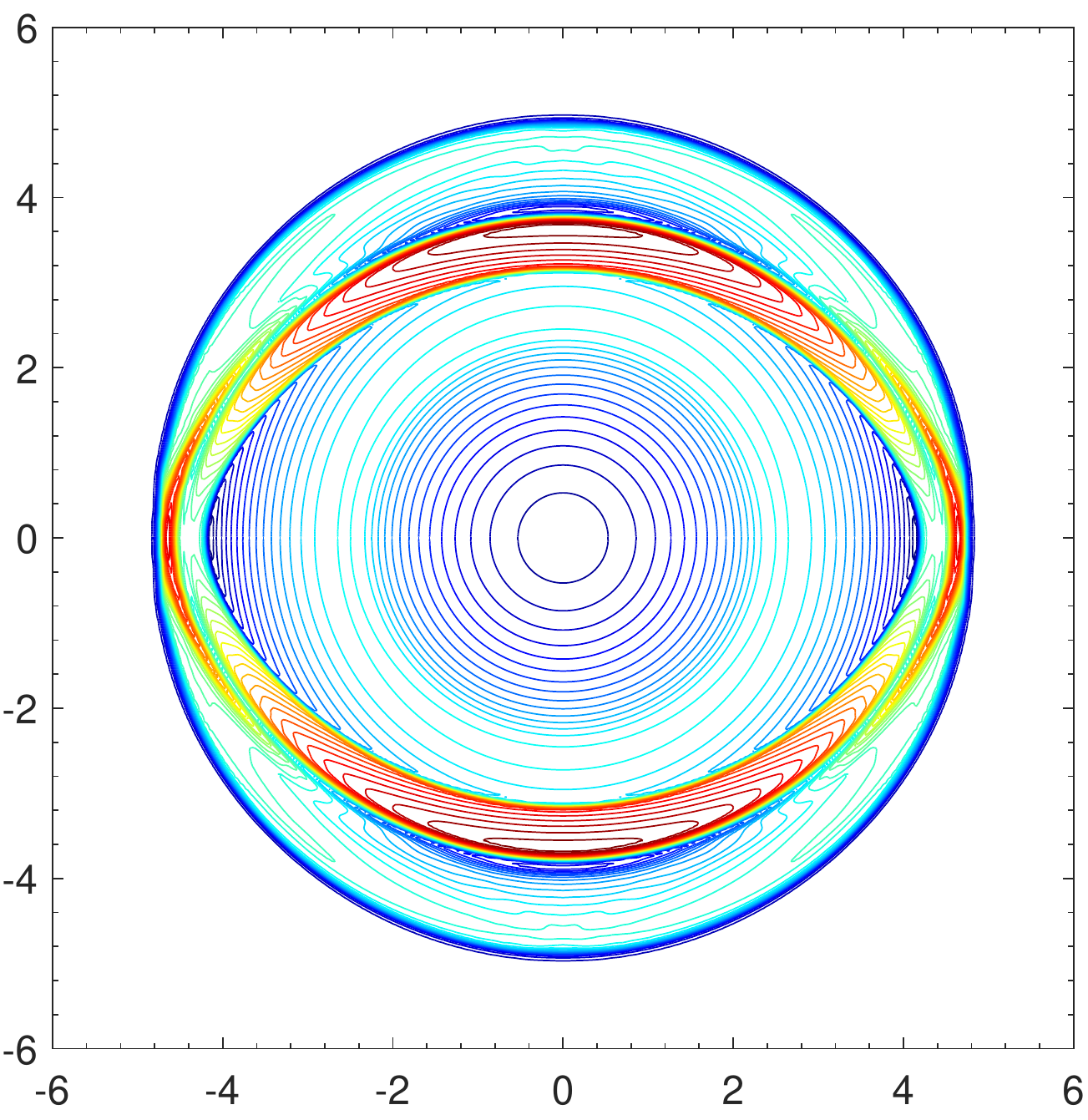}
		\end{center}
	\end{subfigure}
	\begin{subfigure}[b]{0.48\textwidth}
	\begin{center}
		\includegraphics[width=0.88\linewidth]{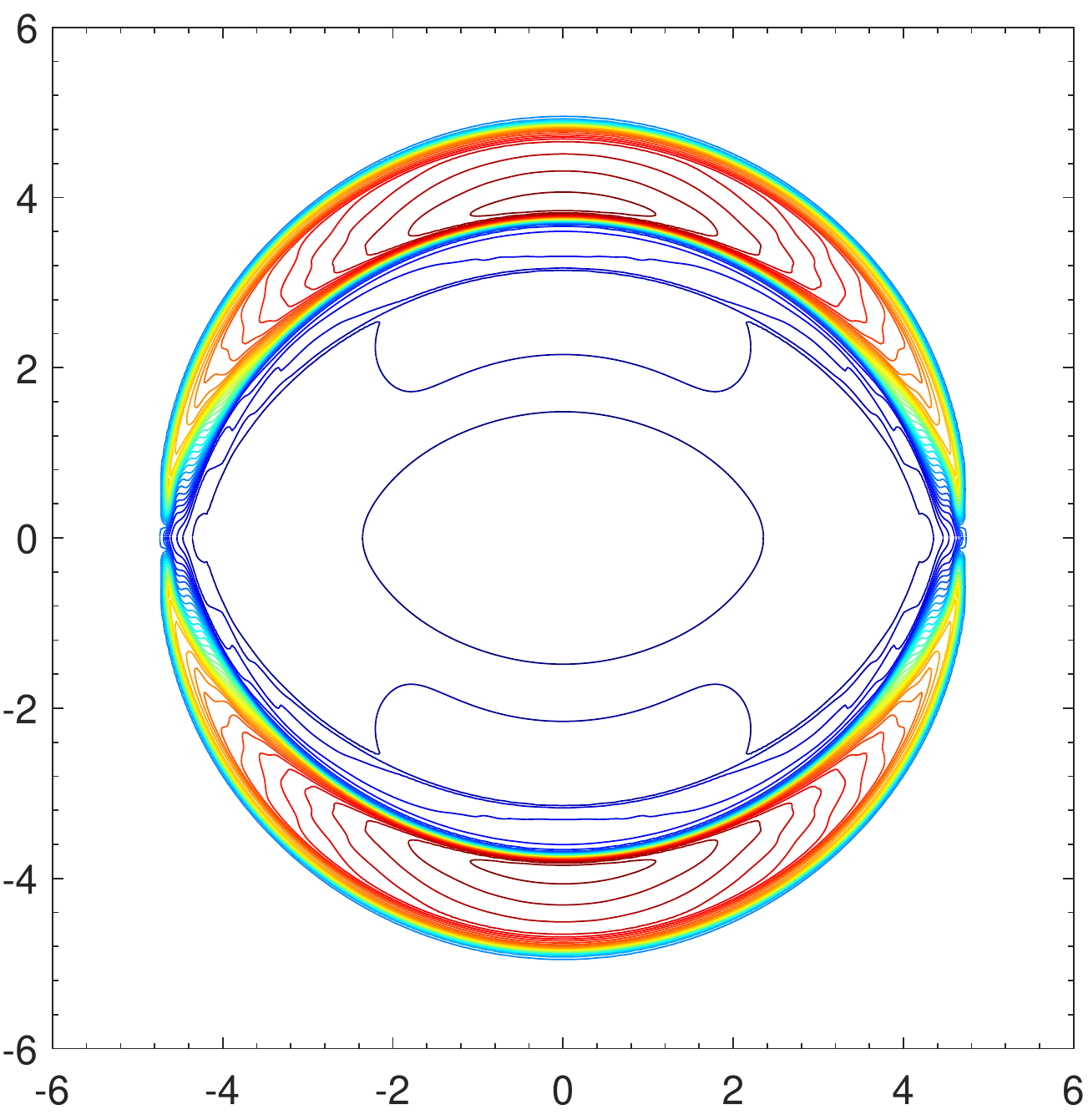}
	\end{center}
	\end{subfigure}
	\begin{subfigure}[b]{0.48\textwidth}
	\begin{center}
		\includegraphics[width=0.88\linewidth]{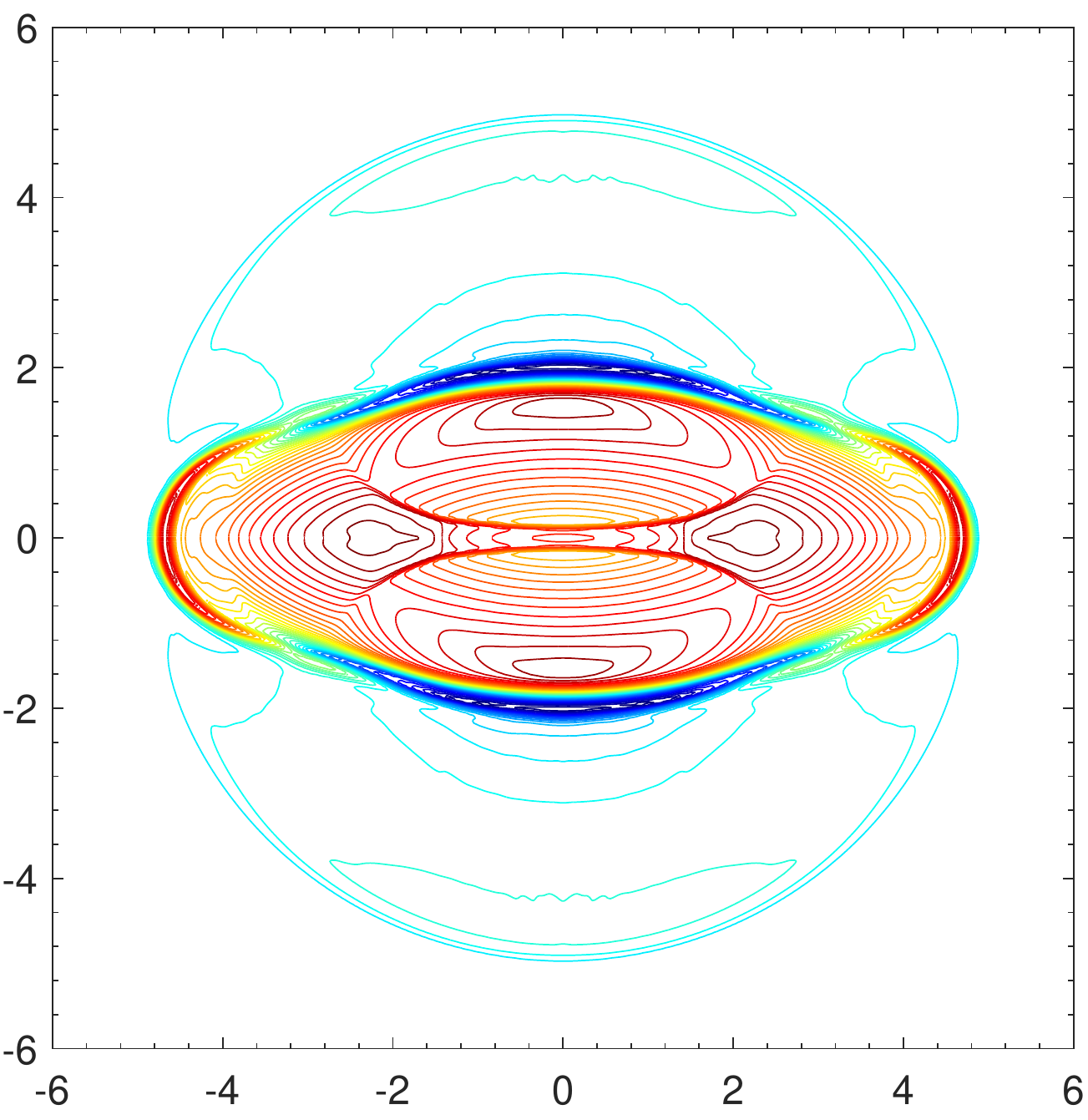}
	\end{center}
	\end{subfigure}
	\begin{subfigure}[b]{0.48\textwidth}
	\begin{center}
		\includegraphics[width=0.88\linewidth]{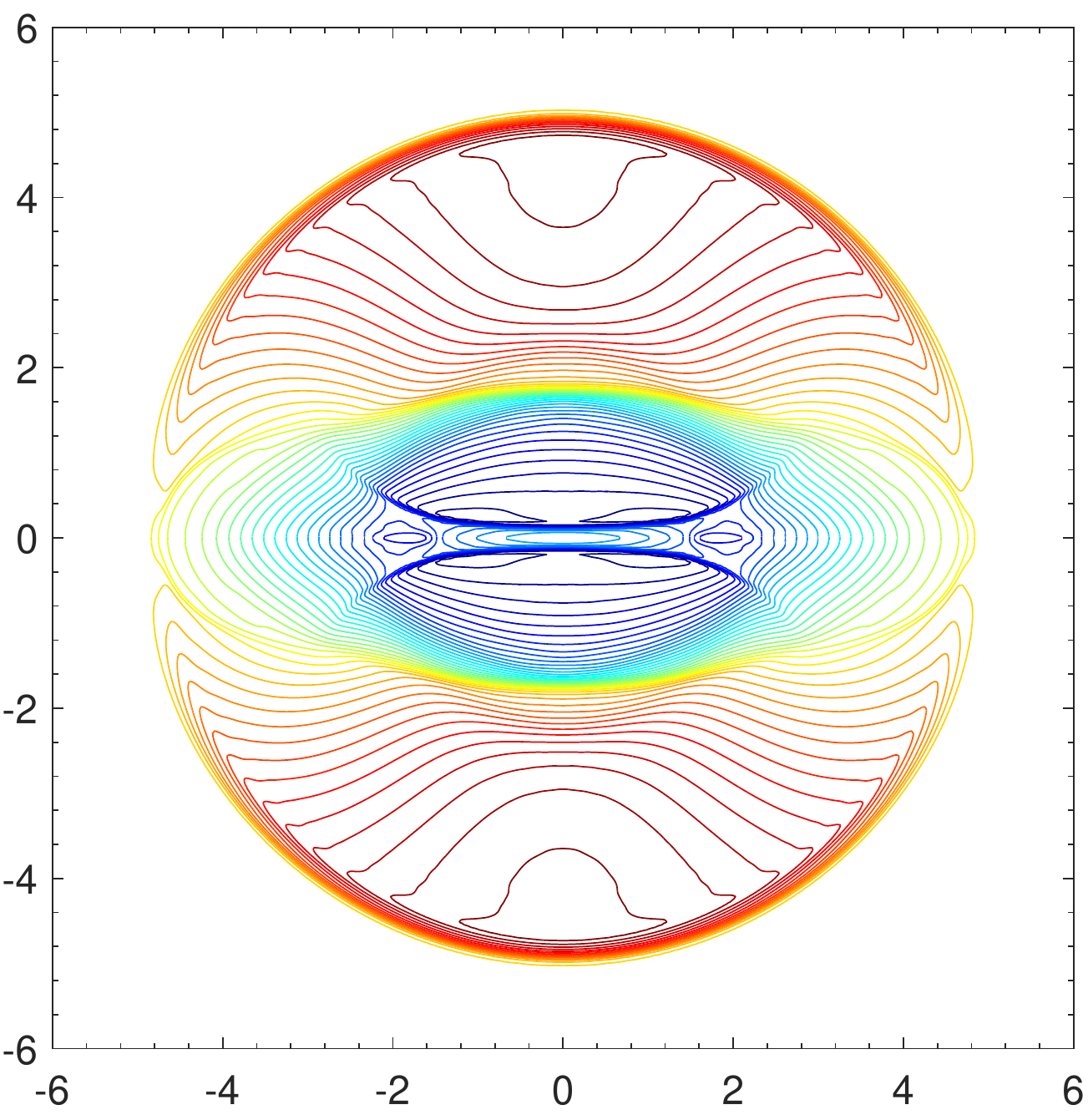}
	\end{center}
	\end{subfigure}
	\begin{subfigure}[b]{0.48\textwidth}
	\begin{center}
		\includegraphics[width=0.88\linewidth]{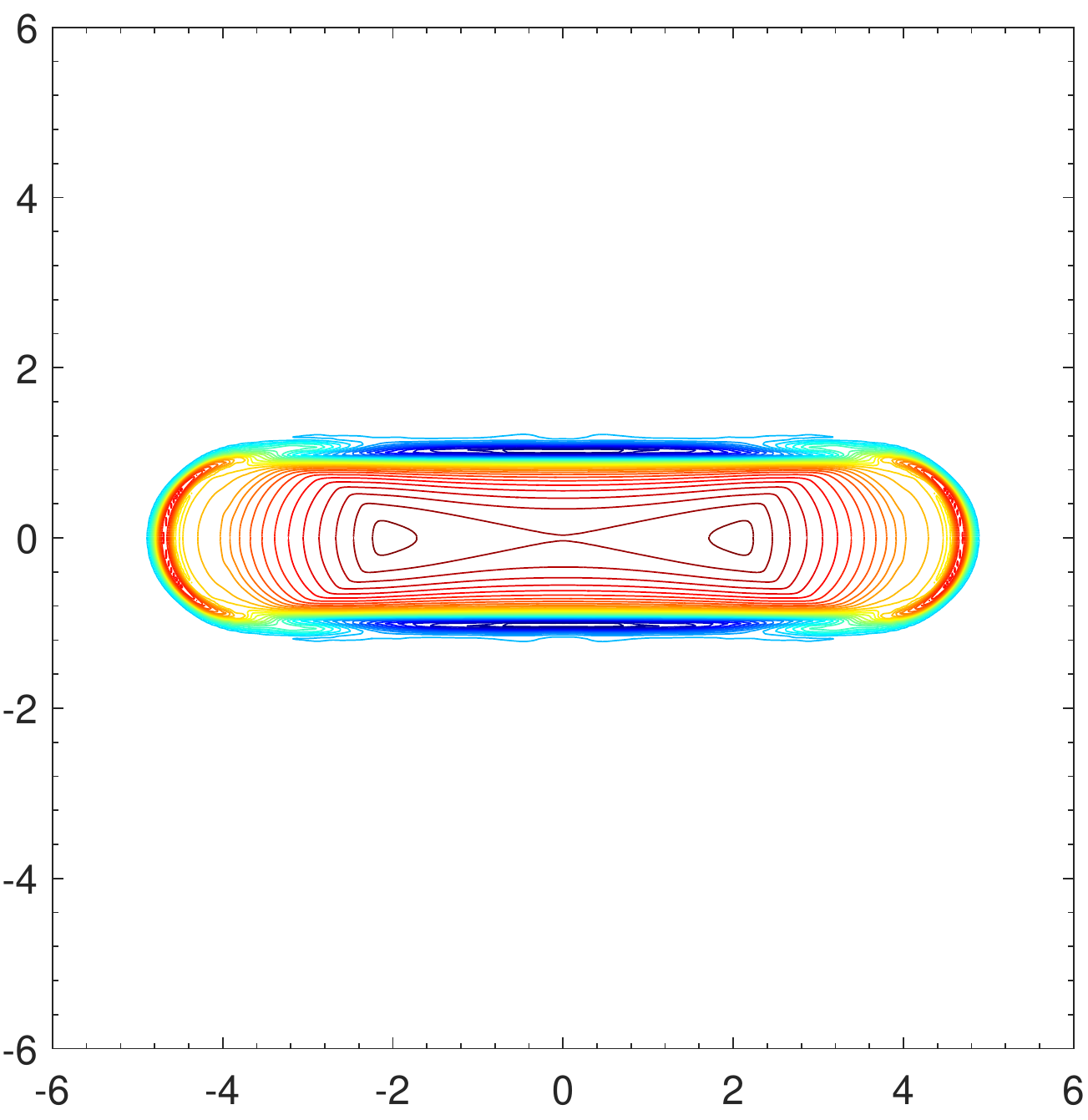}
	\end{center}
	\end{subfigure}
	\begin{subfigure}[b]{0.48\textwidth}
		\begin{center}
			\includegraphics[width=0.88\linewidth]{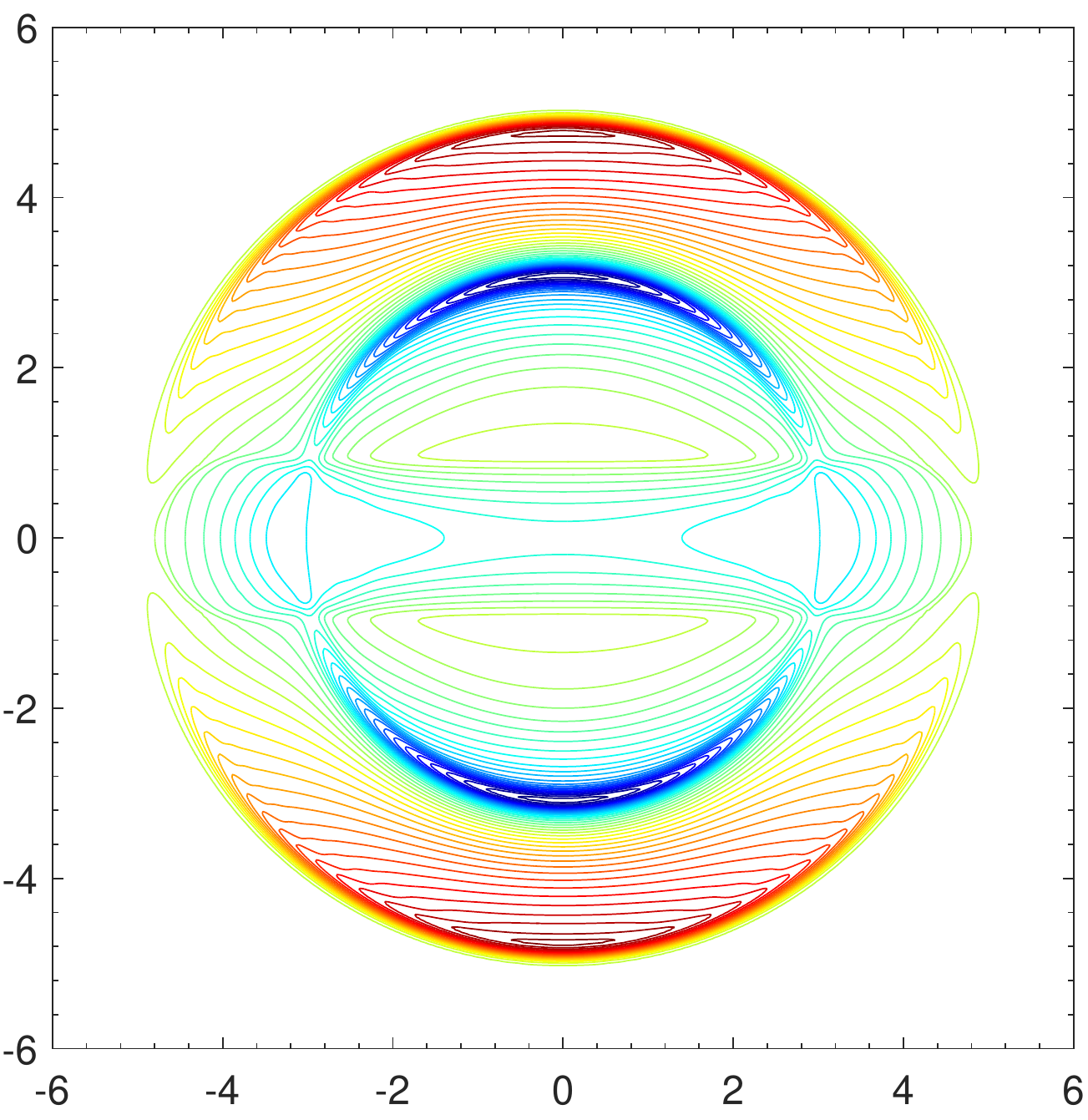}
		\end{center}
	\end{subfigure}
\caption{\small Contour plots of $\log(\rho)$ (left) and $|{\bf B}|$ (right) for the blast problems at $t=4$. Forty equally spaced contour lines are displayed. 
	Top: the moderate magnetized case ($B_a=0.1$, $\beta_a=0.1$); middle: the relatively strongly magnetized case ($B_a=0.5$, $\beta_a=4\times 10^{-3}$); 
bottom: the extremely strongly magnetized case ($B_a=2000$, $\beta_a=2.5\times 10^{-10}$).}
\label{fig:BL}
\end{figure}

In order to examine the robustness and PCP property of our method, we run this test with different $B_a\in \{ 0.1, 0.5, 20, 100, 2000 \}$. These five configurations include 
the two in \cite{MignoneHLLCRMHD,del2007echo,BalsaraKim2016,Zanotti2015} and three much more challenging cases. 
We observe that our PCP methods are able to successfully  
simulate all these test cases without any artificial treatments.  
\figref{fig:BL} shows our numerical results at $t=4$, obtained using our third-order PCP method on $400\times 400$ mesh grids, for three configurations: 
a moderate magnetized case ($B_a=0.1$), a relatively strongly magnetized case ($B_a=0.5$), and 
a extremely strongly magnetized case ($B_a=2000$). 
For the first two cases, our results agree quite well with those reported in \cite{Zanotti2015,BalsaraKim2016,WuTangM3AS}. 
The wave patterns for $B_a=20$ and $B_a=100$ 
are similar to that for $B_a=2000$ and thus omitted here. 
From \figref{fig:BL}, we see that the wave pattern for $B_a=0.1$ is composed by two main waves,
an external fast and a reverse shock waves. The former is almost circular, while the latter is elliptic.
The magnetic field is essentially confined between them, while the inner region is almost devoid of magnetization.
In the case of $B_a=0.5$, the external circular fast shock is clearly visible but very weak.
For $B_a\ge 20$, the external circular fast shock becomes much weaker and is only visible in the magnetic field. 

As far as we know, successful simulations of an extreme RMHD blast test with $B_a=2000$ and so low plasma-beta 
 ($\beta_a=2.5\times 10^{-10}$)  
 have {\em not} been reported in the literature.
We also notice that, if the PCP limiter is turned off in the strongly magnetized tests ($B_a\ge 0.5$), 
nonphysical numerical solutions exceeding the set $\mathbb G_n^k$ will appear 
in the simulations, and the DG code will break down.  
We have also performed the above tests by dropping the discrete symmetrization source term ${\mathcal J}_K^{(2)} ( {\bf U}_h, {\bf u}  )$ in our PCP scheme \eqref{eq:2DDGUh-weak}
 and keeping 
the PCP and WENO limiters turned on. 
The resulting scheme is actually the locally divergence-free DG method with PCP and WENO limiters for the conservative RMHD system \eqref{eq:RMHD}. 
We observed that this scheme, which is generally not PCP in theory \cite{WuTangM3AS}, cannot run the tests with $B_a \in \{100, 2000\}$. 
This demonstrates the importance and necessity of including 
the proper discretization of the symmetrization source term for the PCP property of the DG schemes. 

\subsection{Astrophysical jets}
The last test is to simulate two relativistic jets, where  
the internal energy is exceedingly 
small compared to the kinetic energy so that  
negative pressure could be easily produced in numerical simulation.
Moreover, there may exist strong shock wave, shear flow and interface
instabilities in high-speed jet flows. 
Successful simulation of such jet flows is indeed challenging; cf.~\cite{zhang2010b,WuTang2015,QinShu2016,WuTang2017ApJS,WuShu2018}. 

We consider a pressure-matched highly supersonic RHD jet model from \cite{WuTang2017ApJS} and add a magnetic field so as to simulate the RMHD jet flows. 
Initially, the domain $[-12,12]\times [0,25]$ 
is filled with a static uniform medium with 
an unit rest-mass density. 
A RMHD jet of Mach number $M_b=50$ is 
injected in the $y$-direction through 
the inlet part ($|x| \le 0.5$) 
on the bottom boundary ($y=0$) with a density of $\rho_b=0.1$, 
a pressure equal to the ambient pressure, and a speed of $v_b=0.99c$.  
The corresponding initial 
Lorentz factor $W \approx 7.09$ and the relativistic Mach number $M_r:=M_b W/W_s \approx 354.37$, where $W_s= 1/\sqrt{1-c_s^2}$ is the Lorentz factor associated with the local sound speed $c_s$. 
The exceedingly high Mach number and large Lorentz factor render the simulation of this problem very  challenging. 
The fixed inflow condition is specified on the nozzle $\{y=0, |x|\le 0.5\}$, 
while the other boundary conditions are outflow. 
A magnetic field with a magnitude 
of $B_a$ is initialized along the $y$-direction.  
The presence of magnetic field makes this test more extreme. 
We simulate 
a non-magnetized case with $B_a=0$ and a strongly magnetized case 
with $B_a=\sqrt{2000 p}$ (the corresponding plasma-beta 
$\beta_a=10^{-3}$). 
The computational domain is taken as  
$[0,12]\times [0,30]$ and divided into $240 \times 500$ uniform cells with the reflecting boundary condition on $\{x=0,0\le y \le 25\}$. 

\begin{figure}[htbp]
	\centering
	\begin{subfigure}[b]{0.32\textwidth}
		\begin{center}
			\includegraphics[width=1.0\linewidth]{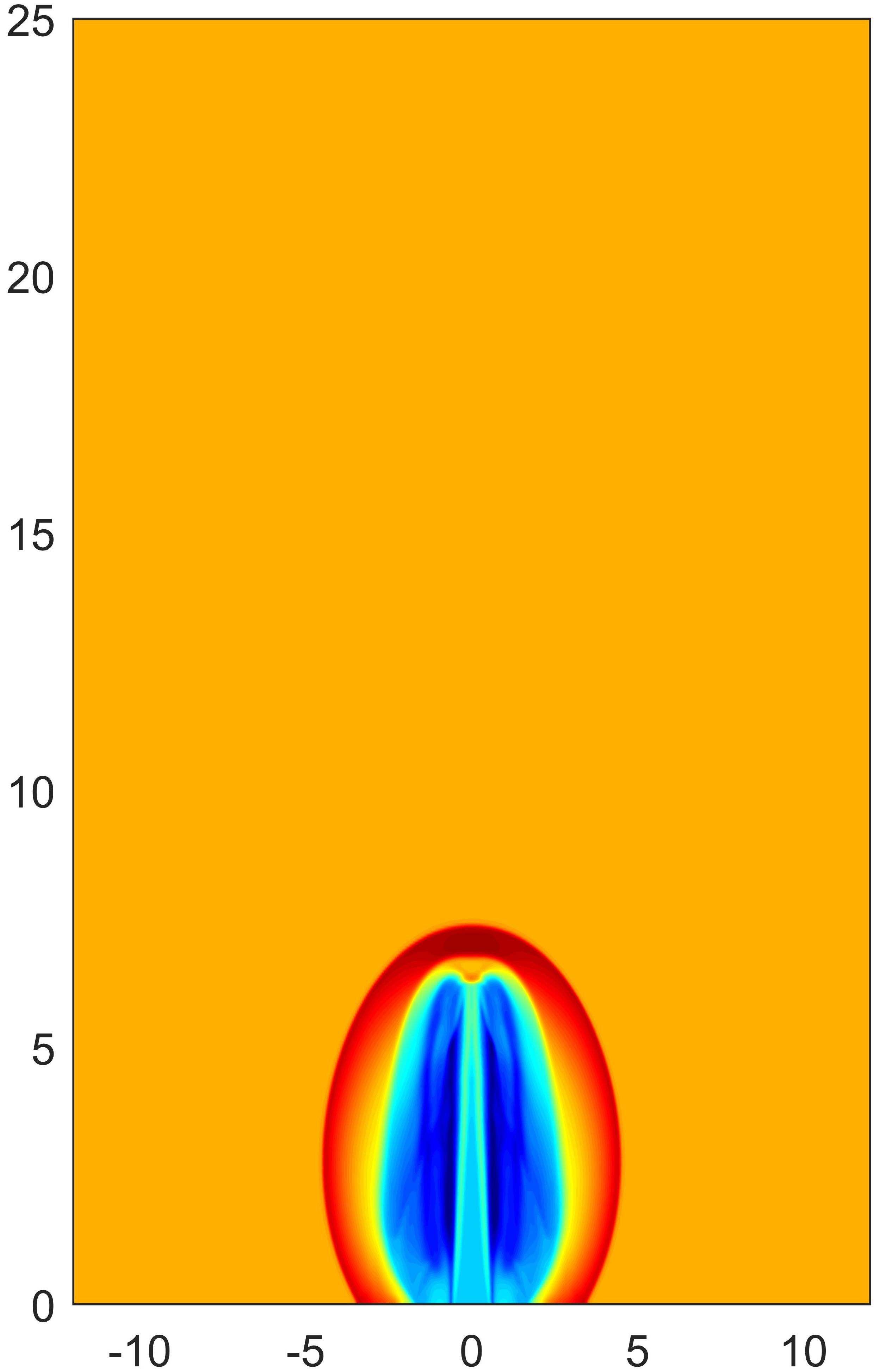}
		\end{center}
	\end{subfigure}
	\begin{subfigure}[b]{0.32\textwidth}
		\begin{center}
			\includegraphics[width=1.0\linewidth]{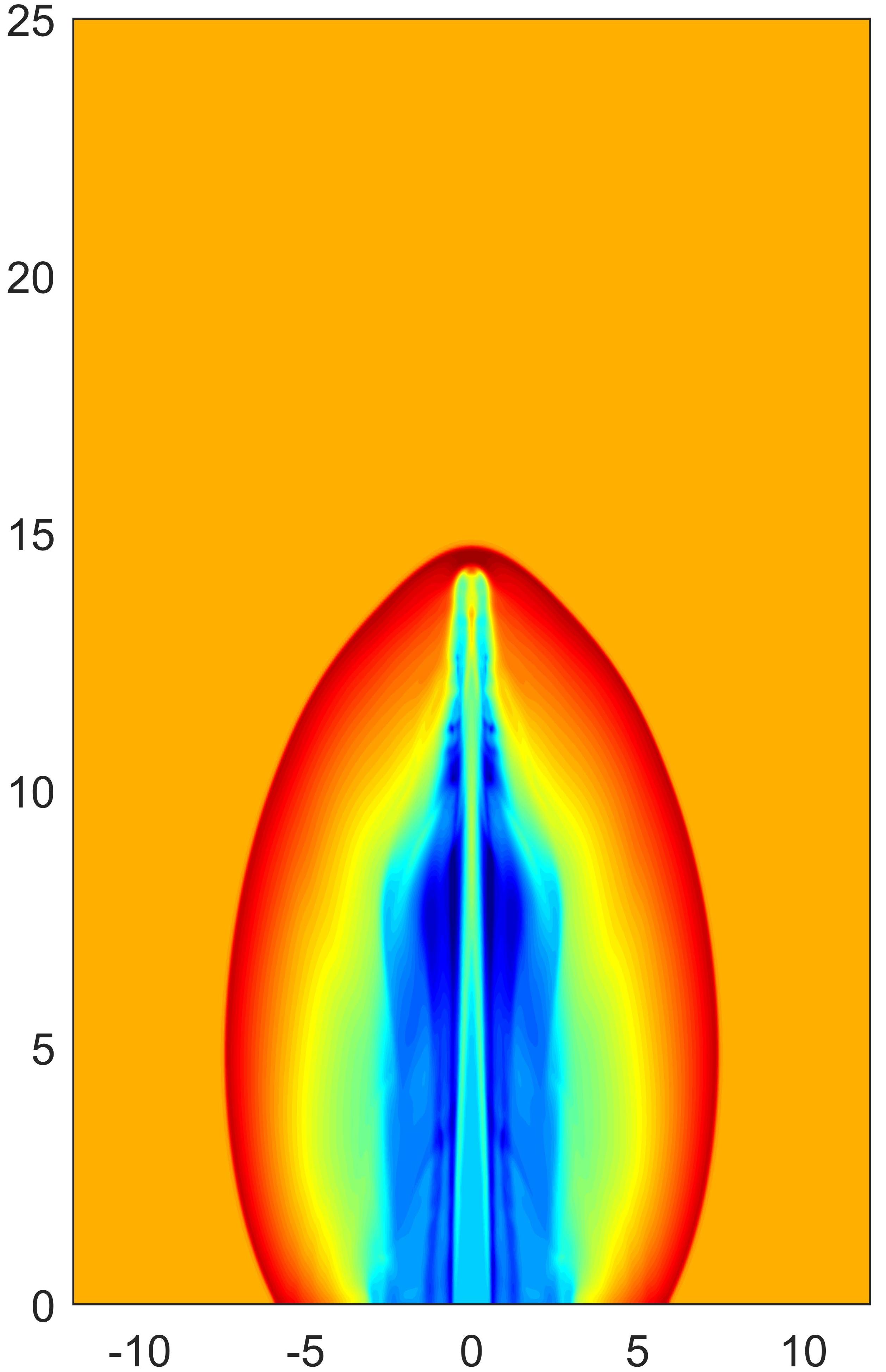}
		\end{center}
	\end{subfigure}
	\begin{subfigure}[b]{0.32\textwidth}
		\begin{center}
			\includegraphics[width=1.0\linewidth]{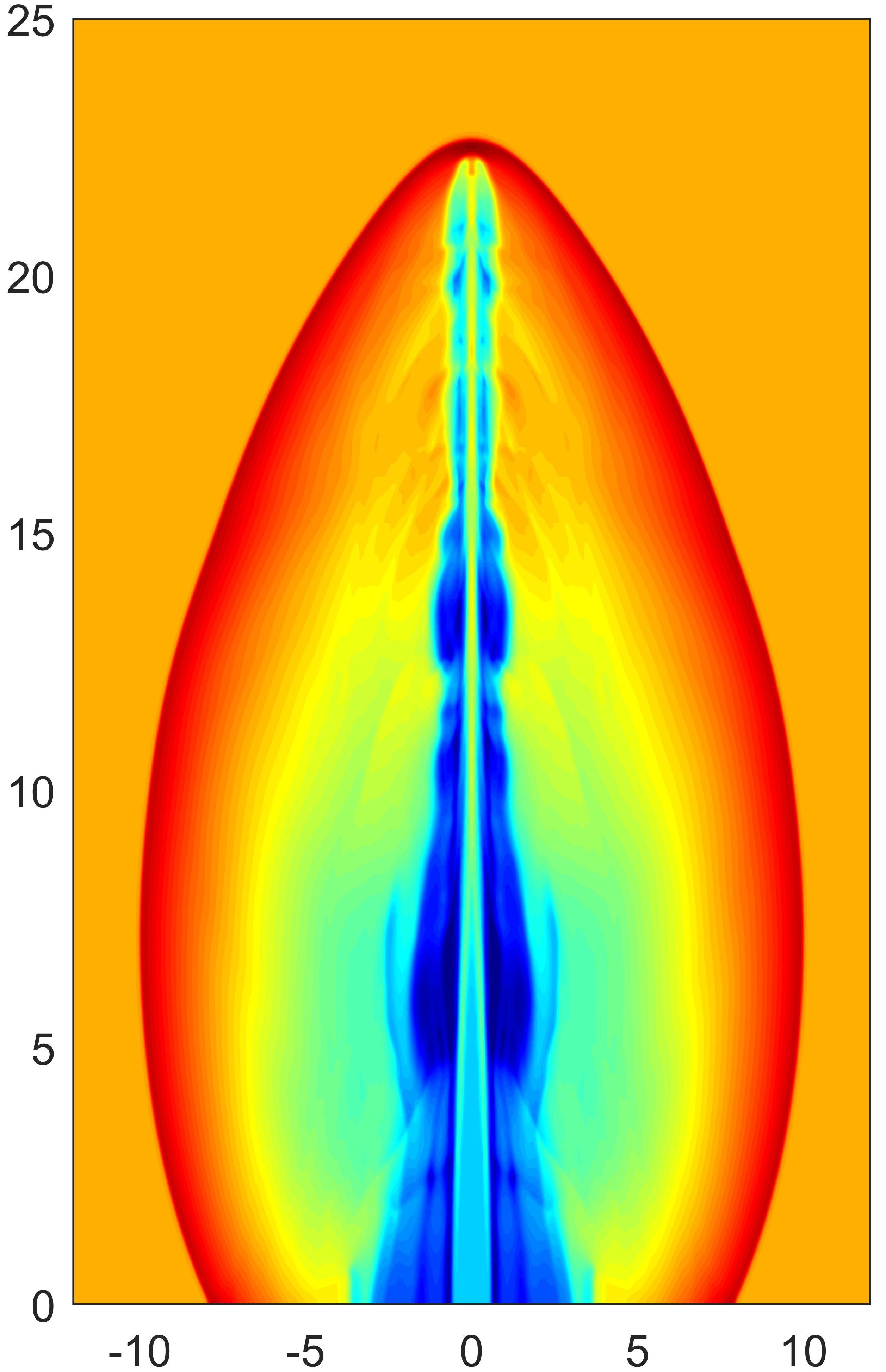}
		\end{center}
	\end{subfigure}
	\begin{subfigure}[b]{0.32\textwidth}
		\begin{center}
			\includegraphics[width=1.0\linewidth]{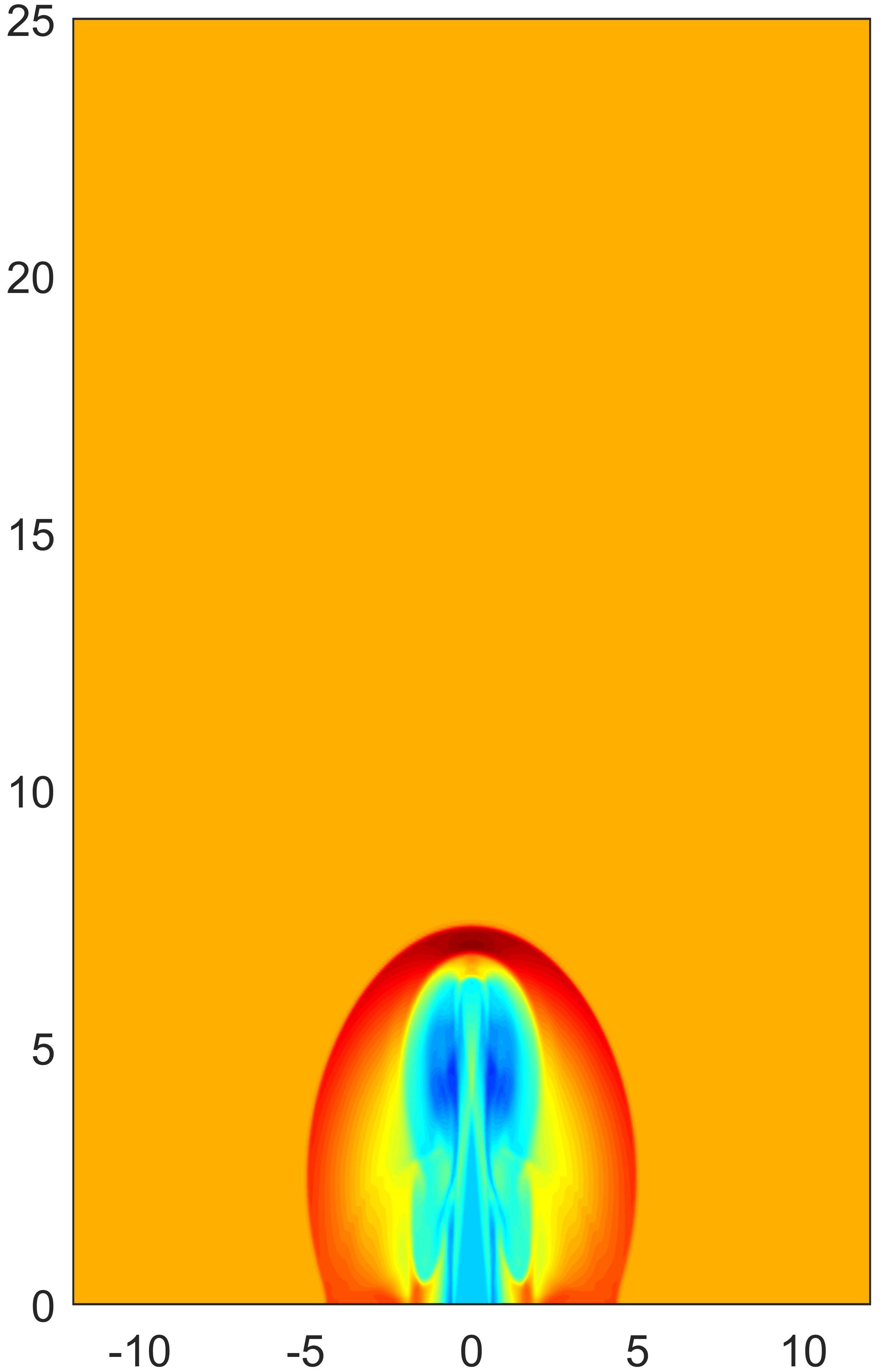}
		\end{center}
	\end{subfigure}
	\begin{subfigure}[b]{0.32\textwidth}
		\begin{center}
			\includegraphics[width=1.0\linewidth]{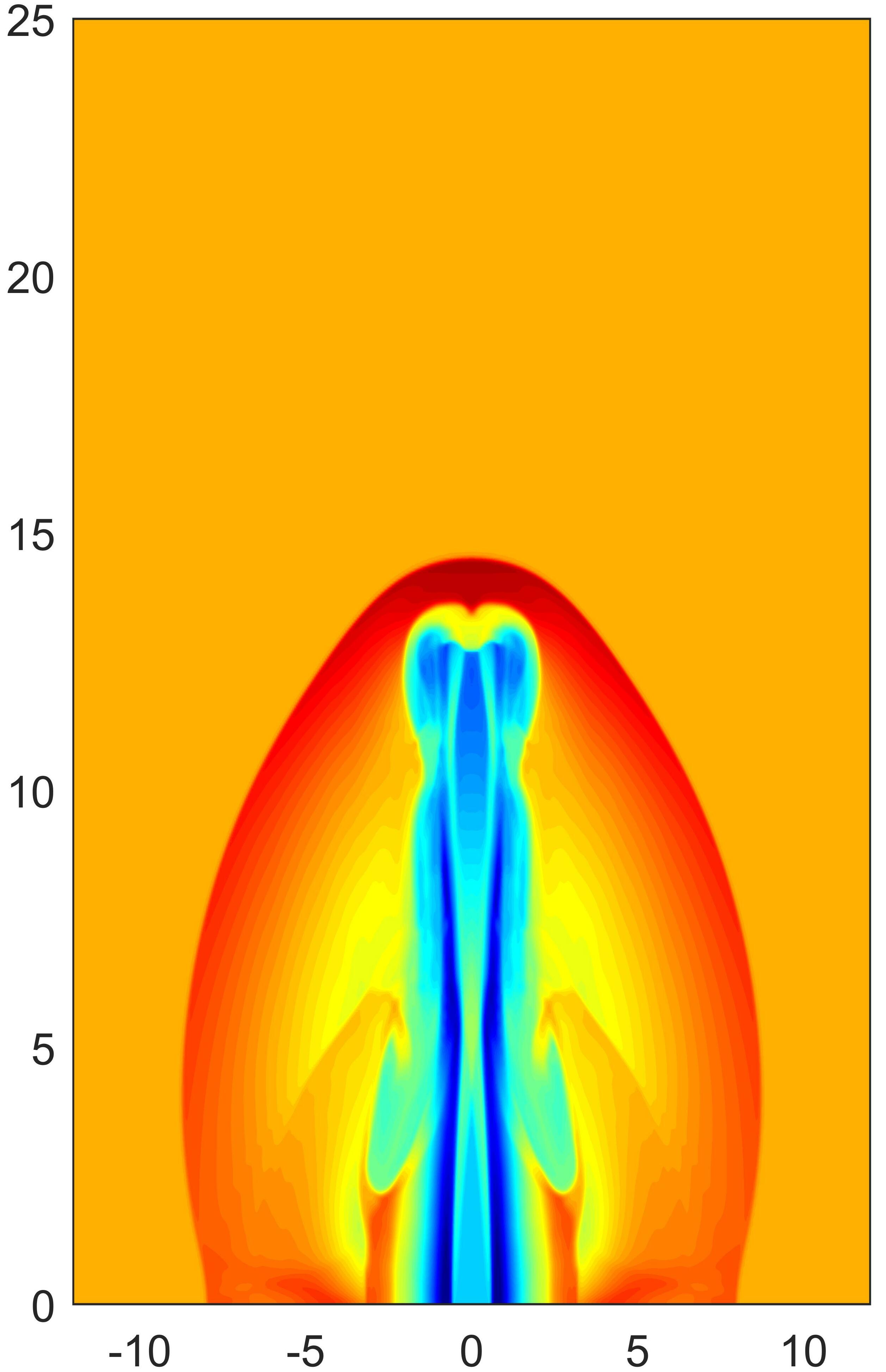}
		\end{center}
	\end{subfigure}
	\begin{subfigure}[b]{0.32\textwidth}
		\begin{center}
			\includegraphics[width=1.0\linewidth]{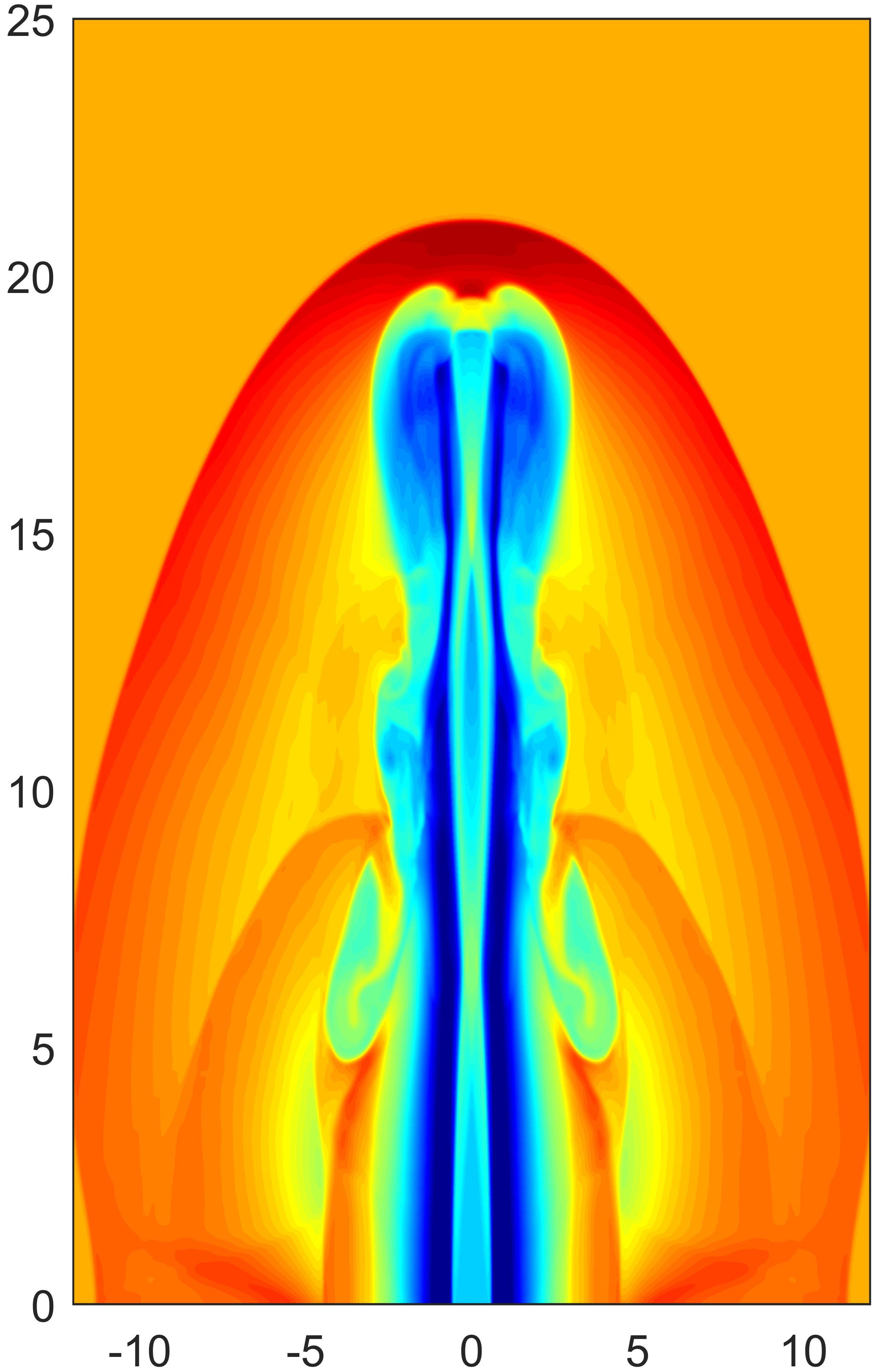}
		\end{center}
	\end{subfigure}
	\captionsetup{belowskip=-11pt}
	\caption{\small Schlieren images of $\log(\rho)$ for the relativistic jets at 
		 $t=10,$ $20$, and $30$ (from left to right). 
		Top: the non-magnetized case; 
		bottom: the strongly magnetized case.}
	\label{fig:Jet}
\end{figure}

\figref{fig:Jet} displays the schlieren images of the rest-mass density logarithm 
within the domain $[-12,12]\times [0,25]$ at $t=10,$ $20$ and $30$, computed by using 
the proposed third-order PCP DG method, for the non-magnetized case 
and the strongly magnetized case, respectively. 
Those plots clearly show the dynamics of the relativistic jets. 
We observe that the Mach shock wave at the jet head and  the beam/cocoon interface
are well captured during the whole simulation. 
The wave patterns for the non-magnetized case are in good agreement with 
those computed in \cite{WuTang2017ApJS}. 
Due to the effect of the strong magnetic field, the flow structures of the strongly magnetized case 
are quite different from those of the non-magnetized case. 
Our PCP method exhibits good robustness in such extreme tests. 
It is observed that if the PCP limiter is turned off, the simulation will break down after 
several time steps due to nonphysical numerical solutions. 
In addition, if dropping the discrete symmetrization source term ${\mathcal J}_K^{(2)} ( {\bf U}_h, {\bf u}  )$ in our PCP scheme \eqref{eq:2DDGUh-weak}, 
we find the cell averages of the DG solutions will 
exceed the set ${\mathcal G}_1$ and the scheme will fail in the strongly magnetized test.    
Again, this demonstrates the importance of including 
the proper discretization of the symmetrization source term for the PCP property.

\section{Conclusions}\label{sec:con}We have proposed a framework of constructing provably PCP high-order DG schemes for the multidimensional RMHD with a general EOS on general meshes. 
The schemes are based on a suitable, locally divergence-free DG discretization of 
symmetrizable RMHD equations, which have accommodated the PCP property at PDE level. 
The resulting DG schemes with strong-stability-preserving time discretizations 
satisfy a weak PCP property, which implies that a simple limiter can enforce the PCP property without losing conservation and high order accuracy. 
Most notably, we rigorously prove the PCP property by using  
a novel ``quasi-linearization'' approach to handle the highly nonlinear physical constraints, technical splitting to offset the influence of divergence error,  
and sophisticated estimates to analyze the beneficial effect of the symmetrization source term. 
Several demanding numerical examples, including strongly magnetized bast problem with extremely low plasma-beta ($2.5 \times 10^{-10}$) 
and highly supersonic RMHD jets, have been tested and demonstrated the effectiveness and  
robustness of the proposed PCP schemes. 
In the context of RMHD, 
our findings furthered the understanding, at both discrete and continuous levels, of the intrinsic connection between the PCP property and divergence-free condition, 
indicating 
the unity of discrete and continuous objects. 


\appendix

\section{Proof of Proposition \ref{thm:PCP_PDElevel}}\label{app:proof}Due to the assumption that the the strong solution of the initial-value problem exists for ${\bf x} \in {\mathbb R}^d$ and $0\le t \le T$, the Lorentz factor $W$ does not blow up, and then   
	$|{\bf v}( {\bf x}, t )|<1$ for all ${\bf x} \in {\mathbb R}^d$ and $0\le t \le T$. 
	Let $\frac{\D}{\D t}:= \frac{\partial}{\partial t} +  {\bf v}({\bf x},t) \nabla \cdot $ be the directional derivative along the direction
	\begin{equation}\label{eq:curve}
	\frac{d {\bf x}}{d t} =  {\bf v} ({\bf x},t).  
	\end{equation}
	For any $\left( \bar {\bf x}, \bar t \right) \in \mathbb R^d \times \mathbb R_+$, let ${\bf x}={\bf x}(t;\bar{\bf x},\bar t)$ be the integral curve of \eqref{eq:curve} through the point $\left( \bar {\bf x}, \bar t \right)$. 
	Denote ${\bf x}_0( \bar{\bf x},\bar t ):=  {\bf x}(0;\bar{\bf x},\bar t)$, then,  
	at $t=0$, the curve passes through the point $\left( {\bf x}_0( \bar{\bf x},\bar t ),0 \right)$. 
	Recall that, for smooth solutions, 
	the first equation of the system \eqref{ModRMHD} can be reformulated as 
	$
	\frac{ \D (\rho W) }{\D t} = - \rho W \nabla \cdot {\bf v}. 
	$ 
	Integrating this equation along the curve ${\bf x}={\bf x}(t;\bar{\bf x},\bar t)$ gives 
	$$
	\rho W ( \bar {\bf x}, \bar t  ) = \rho_0 W_0 ( {\bf x}_0( \bar{\bf x},\bar t ) ) \exp \left( - \int_{0}^{\bar t} \nabla \cdot {\bf v}( {\bf x}(t;\bar{\bf x},\bar t), t ) dt \right) > 0, 
	$$
	which, along with $W ( \bar {\bf x}, \bar t ) \ge 1$, imply $\rho ( \bar {\bf x}, \bar t  )>0$ for all $\left( \bar {\bf x}, \bar t \right) \in \mathbb R^3 \times \mathbb R_+$.  
	For smooth solutions of the modified RMHD system \eqref{ModRMHD}, one can derive that 
	$
	\frac{ \D \left( p \rho^{-\Gamma} \right)}{\D t} = 0, 
	$ 
	which implies 
	$
	p \rho^{-\Gamma} ( \bar {\bf x}, \bar t  ) = p_0 \rho^{-\Gamma}_0 ( {\bf x}_0( \bar{\bf x},\bar t ) )  > 0. 
	$
	It follows that $p  ( \bar {\bf x}, \bar t  ) >0 $ for any $\left( \bar {\bf x}, \bar t \right) \in \mathbb R^3 \times \mathbb R_+$. 
	It is shown in \cite{WuShu2019SISC} that, for smooth solutions of \eqref{ModRMHD}, the quantity  $\frac{\nabla \cdot {\bf B}}{\rho W}$ is constant along the curve ${\bf x}={\bf x}(t;\bar{\bf x},\bar t)$, which implies \eqref{maxminDivB}. The proof is complete.

\bibliographystyle{siamplain}
\bibliography{references}

\begin{thebibliography}{10}

\bibitem{BalsaraKim2016}
{\sc D.~S. Balsara and J.~Kim}, {\em A subluminal relativistic
  magnetohydrodynamics scheme with {ADER-WENO} predictor and multidimensional
  {Riemann} solver-based corrector}, J. Comput. Phys., 312 (2016),
  pp.~357--384.

\bibitem{BalsaraSpicer1999}
{\sc D.~S. Balsara and D.~Spicer}, {\em A staggered mesh algorithm using high
  order {Godunov} fluxes to ensure solenoidal magnetic fields in
  magnetohydrodynamic simulations}, J. Comput. Phys., 149 (1999), pp.~270--292.

\bibitem{Christlieb}
{\sc A.~J. Christlieb, Y.~Liu, Q.~Tang, and Z.~Xu}, {\em Positivity-preserving
  finite difference weighted {ENO} schemes with constrained transport for ideal
  magnetohydrodynamic equations}, SIAM J. Sci. Comput., 37 (2015),
  pp.~A1825--A1845.

\bibitem{CockburnHouShu1990}
{\sc B.~Cockburn, S.~Hou, and C.-W. Shu}, {\em The {Runge-Kutta} local
  projection discontinuous {Galerkin} finite element method for conservation
  laws. {IV.} {The} multidimensional case}, Math. Comp., 54 (1990),
  pp.~545--581.

\bibitem{CockburnShu1989}
{\sc B.~Cockburn and C.-W. Shu}, {\em Tvb runge-kutta local projection
  discontinuous {Galerkin} finite element method for conservation laws. {II.}
  {General} framework}, Math. Comp., 52 (1989), pp.~411--435.

\bibitem{Dedner2002}
{\sc A.~Dedner, F.~Kemm, D.~Kr{\"o}ner, C.-D. Munz, T.~Schnitzer, and
  M.~Wesenberg}, {\em Hyperbolic divergence cleaning for the {MHD} equations},
  J. Comput. Phys., 175 (2002), pp.~645--673.

\bibitem{del2007echo}
{\sc L.~Del~Zanna, O.~Zanotti, N.~Bucciantini, and P.~Londrillo}, {\em Echo: a
  {Eulerian} conservative high-order scheme for general relativistic
  magnetohydrodynamics and magnetodynamics}, Astron.~\& Astrophys., 473 (2007),
  pp.~11--30.

\bibitem{du2018positivity}
{\sc J.~Du and C.-W. Shu}, {\em Positivity-preserving high-order schemes for
  conservation laws on arbitrarily distributed point clouds with a simple
  {WENO} limiter}, Int. J. Numer. Anal. Model., 15 (2018), pp.~1--25.

\bibitem{Evans1988}
{\sc C.~R. Evans and J.~F. Hawley}, {\em Simulation of magnetohydrodynamic
  flows: a constrained transport method}, Astrophys. J., 332 (1988),
  pp.~659--677.

\bibitem{Fu2018}
{\sc P.~Fu, F.~Li, and Y.~Xu}, {\em Globally divergence-free discontinuous
  {Galerkin} methods for ideal magnetohydrodynamic equations}, J. Sci. Comput.,
  77 (2018), pp.~1621--1659.

\bibitem{Godunov1972}
{\sc S.~K. Godunov}, {\em Symmetric form of the equations of
  magnetohydrodynamics}, Numerical Methods for Mechanics of Continuum Medium, 1
  (1972), pp.~26--34.

\bibitem{GottliebShuTadmor2001}
{\sc S.~Gottlieb, C.-W. Shu, and E.~Tadmor}, {\em Strong stability-preserving
  high-order time discretization methods}, SIAM Rev., 43 (2001), pp.~89--112.

\bibitem{HeTang2012RMHD}
{\sc P.~He and H.~Tang}, {\em An adaptive moving mesh method for
  two-dimensional relativistic magnetohydrodynamics}, Comput. Fluids, 60
  (2012), pp.~1--20.

\bibitem{Hu2013}
{\sc X.~Y. Hu, N.~A. Adams, and C.-W. Shu}, {\em Positivity-preserving method
  for high-order conservative schemes solving compressible {Euler} equations},
  J. Comput. Phys., 242 (2013), pp.~169--180.

\bibitem{komissarov1999godunov}
{\sc S.~S. Komissarov}, {\em A {Godunov-type} scheme for relativistic
  magnetohydrodynamics}, Mon. Not. R. Astron. Soc., 303 (1999), pp.~343--366.

\bibitem{Krivodonova}
{\sc L.~Krivodonova, J.~Xin, J.-F. Remacle, N.~Chevaugeon, and J.~E. Flaherty},
  {\em Shock detection and limiting with discontinuous {Galerkin} methods for
  hyperbolic conservation laws}, Appl. Numer. Math., 48 (2004), pp.~323--338.

\bibitem{Li2005}
{\sc F.~Li and C.-W. Shu}, {\em Locally divergence-free discontinuous
  {Galerkin} methods for {MHD} equations}, J. Sci. Comput., 22 (2005),
  pp.~413--442.

\bibitem{Li2011}
{\sc F.~Li, L.~Xu, and S.~Yakovlev}, {\em Central discontinuous {Galerkin}
  methods for ideal {MHD} equations with the exactly divergence-free magnetic
  field}, J. Comput. Phys., 230 (2011), pp.~4828--4847.

\bibitem{Liang2014}
{\sc C.~Liang and Z.~Xu}, {\em Parametrized maximum principle preserving flux
  limiters for high order schemes solving multi-dimensional scalar hyperbolic
  conservation laws}, J. Sci. Comput., 58 (2014), pp.~41--60.

\bibitem{LingDuanTang2019}
{\sc D.~Ling, J.~Duan, and H.~Tang}, {\em Physical-constraints-preserving
  {Lagrangian} finite volume schemes for one- and two-dimensional special
  relativistic hydrodynamics}, J. Comput. Phys., 396 (2019), pp.~507--543.

\bibitem{MignoneHLLCRMHD}
{\sc A.~Mignone and G.~Bodo}, {\em An {HLLC} riemann solver for relativistic
  flows--{II}. magnetohydrodynamics}, Mon. Not. R. Astron. Soc., 368 (2006),
  pp.~1040--1054.

\bibitem{Powell1994}
{\sc K.~G. Powell}, {\em An approximate {Riemann} solver for
  magnetohydrodynamics (that works in more than one dimension)}, Tech. Report
  ICASE Report No. 94-24, NASA Langley, VA, 1994.

\bibitem{POWELL1999284}
{\sc K.~G. Powell, P.~L. Roe, T.~J. Linde, T.~I. Gombosi, and D.~L.~D. Zeeuw},
  {\em A solution-adaptive upwind scheme for ideal magnetohydrodynamics}, J.
  Comput. Phys., 154 (1999), pp.~284 -- 309.

\bibitem{QinShu2016}
{\sc T.~Qin, C.-W. Shu, and Y.~Yang}, {\em Bound-preserving discontinuous
  {Galerkin} methods for relativistic hydrodynamics}, J. Comput. Phys., 315
  (2016), pp.~323--347.

\bibitem{Qiu2005}
{\sc J.~Qiu and C.-W. Shu}, {\em {Runge--Kutta} discontinuous {Galerkin} method
  using {WENO} limiters}, SIAM J. Sci. Comput., 26 (2005), pp.~907--929.

\bibitem{Radice2014}
{\sc D.~Radice, L.~Rezzolla, and F.~Galeazzi}, {\em High-order fully
  general-relativistic hydrodynamics: new approaches and tests}, Classical and
  Quantum Gravity, 31 (2014), p.~075012.

\bibitem{Shu2018}
{\sc C.-W. Shu}, {\em Bound-preserving high-order schemes for hyperbolic
  equations: Survey and recent developments}, in Theory, Numerics and
  Applications of Hyperbolic Problems II, C.~Klingenberg and M.~Westdickenberg,
  eds., Cham, 2018, Springer International Publishing, pp.~591--603.

\bibitem{SunShu2019}
{\sc Z.~Sun and C.-w. Shu}, {\em Strong stability of explicit {Runge--Kutta}
  time discretizations}, SIAM J. Numer. Anal., 57 (2019), pp.~1158--1182.

\bibitem{Torrilhon2005}
{\sc M.~Torrilhon}, {\em Locally divergence-preserving upwind finite volume
  schemes for magnetohydrodynamic equations}, SIAM J. Sci. Comput., 26 (2005),
  pp.~1166--1191.

\bibitem{Toth2000}
{\sc G.~T{\'o}th}, {\em The {$\nabla \cdot {\bf {B}} = 0$} constraint in
  shock-capturing magnetohydrodynamics codes}, J. Comput. Phys., 161 (2000),
  pp.~605--652.

\bibitem{Host:2008}
{\sc B.~van~der Holst, R.~Keppens, and Z.~Meliani}, {\em A multidimensional
  grid-adaptive relativistic magnetofluid code}, Comput. Phys. Commun., 179
  (2008), pp.~617--627.

\bibitem{VILAR2016416}
{\sc F.~Vilar, C.-W. Shu, and P.-H. Maire}, {\em Positivity-preserving
  cell-centered lagrangian schemes for multi-material compressible flows:
  {From} first-order to high-orders. {Part II:} the two-dimensional case}, J.
  Comput. Phys., 312 (2016), pp.~416--442.

\bibitem{Wu2017}
{\sc K.~Wu}, {\em Design of provably physical-constraint-preserving methods for
  general relativistic hydrodynamics}, Phys. Rev. D, 95 (2017), 103001.

\bibitem{Wu2017a}
{\sc K.~Wu}, {\em Positivity-preserving analysis of numerical schemes for ideal
  magnetohydrodynamics}, SIAM J. Numer. Anal., 56 (2018), pp.~2124--2147.

\bibitem{WuShu2018}
{\sc K.~Wu and C.-W. Shu}, {\em A provably positive discontinuous {Galerkin}
  method for multidimensional ideal magnetohydrodynamics}, SIAM J. Sci.
  Comput., 40 (2018), pp.~B1302--B1329.

\bibitem{WuShu2019SISC}
{\sc K.~Wu and C.-W. Shu}, {\em Entropy symmetrization and high-order accurate
  entropy stable numerical schemes for relativistic mhd equations}, submitted
  to SIAM J. Sci. Comput., available from arXiv:1907.07467,  (2019).

\bibitem{WuShu2019}
{\sc K.~Wu and C.-W. Shu}, {\em Provably positive high-order schemes for ideal
  magnetohydrodynamics: analysis on general meshes}, Numer. Math., 142 (2019),
  pp.~995--1047.

\bibitem{WuTang2015}
{\sc K.~Wu and H.~Tang}, {\em High-order accurate
  physical-constraints-preserving finite difference {WENO} schemes for special
  relativistic hydrodynamics}, J. Comput. Phys., 298 (2015), pp.~539--564.

\bibitem{WuTangM3AS}
{\sc K.~Wu and H.~Tang}, {\em Admissible states and
  physical-constraints-preserving schemes for relativistic magnetohydrodynamic
  equations}, Math. Models Methods Appl. Sci., 27 (2017), pp.~1871--1928.

\bibitem{WuTang2017ApJS}
{\sc K.~Wu and H.~Tang}, {\em Physical-constraint-preserving central
  discontinuous {Galerkin} methods for special relativistic hydrodynamics with
  a general equation of state}, Astrophys. J. Suppl. Ser., 228 (2017), 3.

\bibitem{WuTangZAMP}
{\sc K.~Wu and H.~Tang}, {\em On physical-constraints-preserving schemes for
  special relativistic magnetohydrodynamics with a general equation of state},
  Z. Angew. Math. Phys., 69 (2018), 84.

\bibitem{Xu2014}
{\sc Z.~Xu}, {\em Parametrized maximum principle preserving flux limiters for
  high order schemes solving hyperbolic conservation laws: one-dimensional
  scalar problem}, Math. Comp., 83 (2014), pp.~2213--2238.

\bibitem{XuLiu2016}
{\sc Z.~Xu and Y.~Liu}, {\em New central and central discontinuous galerkin
  schemes on overlapping cells of unstructured grids for solving ideal
  magnetohydrodynamic equations with globally divergence-free magnetic field},
  J. Comput. Phys., 327 (2016), pp.~203--224.

\bibitem{Zanotti2015}
{\sc O.~Zanotti, F.~Fambri, and M.~Dumbser}, {\em Solving the relativistic
  magnetohydrodynamics equations with {ADER} discontinuous {Galerkin} methods,
  a posteriori subcell limiting and adaptive mesh refinement}, Mon. Not. R.
  Astron. Soc., 452 (2015), pp.~3010--3029.

\bibitem{ZHANG2017301}
{\sc X.~Zhang}, {\em On positivity-preserving high order discontinuous
  {Galerkin} schemes for compressible {Navier-Stokes} equations}, J. Comput.
  Phys., 328 (2017), pp.~301--343.

\bibitem{zhang2010}
{\sc X.~Zhang and C.-W. Shu}, {\em On maximum-principle-satisfying high order
  schemes for scalar conservation laws}, J. Comput. Phys., 229 (2010),
  pp.~3091--3120.

\bibitem{zhang2010b}
{\sc X.~Zhang and C.-W. Shu}, {\em On positivity-preserving high order
  discontinuous {Galerkin} schemes for compressible {Euler} equations on
  rectangular meshes}, J. Comput. Phys., 229 (2010), pp.~8918--8934.

\bibitem{zhang2012maximum}
{\sc X.~Zhang, Y.~Xia, and C.-W. Shu}, {\em Maximum-principle-satisfying and
  positivity-preserving high order discontinuous galerkin schemes for
  conservation laws on triangular meshes}, J. Sci. Comput., 50 (2012),
  pp.~29--62.

\bibitem{ZhaoTang2017}
{\sc J.~Zhao and H.~Tang}, {\em {Runge-Kutta} discontinuous {Galerkin} methods
  for the special relativistic magnetohydrodynamics}, J. Comput. Phys., 343
  (2017), pp.~33--72.

\bibitem{ZouYuDai2019}
{\sc S.~Zou, X.~Yu, and Z.~Dai}, {\em A positivity-preserving {Lagrangian}
  discontinuous {Galerkin} method for ideal magnetohydrodynamics equations in
  one-dimension}, J. Comput. Phys., 405 (2020), p.~109144.

\end{thebibliography}

\end{document}